\documentclass[a4paper,11pt]{article}
\pdfoutput=1
\usepackage[utf8]{inputenc}  
\usepackage[T1]{fontenc}         % Police contenant les caractÃƒÆ’Ã†â€™Ãƒâ€šÃ‚Â¨res franÃƒÆ’Ã†â€™Ãƒâ€šÃ‚Â§ais
\usepackage{geometry}         % DÃƒÆ’Ã†â€™Ãƒâ€šÃ‚Â©finir les marges
\usepackage[english]{babel}  % Placez ici une liste de langues, la
\usepackage{cite}
\usepackage{amssymb,  xfrac, stmaryrd,array, multirow, enumerate,
  verbatim,xifthen,amsmath,bm,mathtools}
\usepackage{multicol}
\usepackage{float}
\usepackage[usenames,dvipsnames]{xcolor}
\usepackage{amsmath} 
\usepackage{amscd}
\usepackage{graphicx, color, pinlabel}    
\usepackage{caption}
\usepackage{subcaption}
\usepackage{amsthm}                 
\usepackage{latexsym}     % derniÃƒÆ’Ã†â€™Ãƒâ€šÃ‚Â¨re ÃƒÆ’Ã†â€™Ãƒâ€šÃ‚Â©tant la langue principale
\usepackage[all]{xy}
\usepackage{hyperref}
\usepackage[all]{hypcap} % makes internal links to figures point at
\input{alexpreamble}
\def\makeautorefname#1#2{\expandafter\def\csname#1autorefname\endcsname{#2}}
\let\fullref\autoref

\newtheorem{theorem}{Theorem}[section]
\makeautorefname{theorem}{Theorem}

\theoremstyle{plain}
\newtheorem{lemma}{Lemma}[section]
\makeautorefname{lemma}{Lemma} 
\makeatletter 
\let\c@lemma=\c@theorem 
\makeatother

\newtheorem{proposition}{Proposition}[section]
\makeautorefname{proposition}{Proposition} 
\makeatletter 
\let\c@proposition=\c@theorem 
\makeatother

\newtheorem{corollary}{Corollary}[section]
\makeautorefname{corollary}{Corollary} 
\makeatletter 
\let\c@corollary=\c@theorem 
\makeatother

\newtheorem{thm}{Theorem}[section]
\makeautorefname{thm}{Theorem} 
\makeatletter
\let\c@thm\c@theorem
\makeatother

\newtheorem*{thm*}{Theorem}

\newtheorem{cor}{Corollary}[section]
\makeautorefname{cor}{Corollary} 
\makeatletter
\let\c@cor\c@theorem
\makeatother

\newtheorem*{cor*}{Corollary}

\makeautorefname{prop}{Proposition} 
\makeatletter
\let\c@prop\c@theorem
\makeatother

\newtheorem*{prop*}{Proposition}

\makeautorefname{lem}{Lemma} 
\makeatletter
\let\c@lem\c@theorem
\makeatother

\newtheorem{quest}{Question}

\theoremstyle{remark}
\newtheorem{rem}{Remark}[section]
\makeautorefname{rem}{Remark} 
\makeatletter
\let\c@rem\c@theorem
\makeatother              

\makeautorefname{claim}{Claim}

\newtheorem*{remark}{Remark}
\newtheorem*{observation}{Observation}

\theoremstyle{definition}
\newtheorem{definition}{Definition}[section]
\makeautorefname{definition}{Definition}
\makeatletter 
\let\c@definition=\c@theorem 
\makeatother

\newtheorem{example}{Example}[section]
\makeautorefname{example}{Example}
\makeatletter 
\let\c@example=\c@theorem 
\makeatother

\makeautorefname{assumption}{Assumption}
\makeatletter 
\let\c@assumption=\c@theorem 
\makeatother

\makeautorefname{deflem}{Lemma-Definition}
\makeatletter 
\let\c@deflem=\c@theorem 
\makeatother

\newtheorem{rmk}{Remark}[section]
\makeautorefname{rmk}{Remark}
\makeatletter 
\let\c@rmk=\c@theorem 
\makeatother

\makeautorefname{notation}{Notation}
\makeatletter 
\let\c@notation=\c@theorem 
\makeatother

\makeautorefname{defprop}{Proposition-Definition}
\makeatletter 
\let\c@defprop=\c@theorem 
\makeatother

\makeautorefname{defi}{Definition}
\makeatletter 
\let\c@defi=\c@theorem 
\makeatother

\AtBeginDocument{
\makeautorefname{section}{Section}%
\makeautorefname{subsection}{Section}%
\makeautorefname{subsubsection}{Section}%
\makeautorefname{quest}{Question}%
}

\newcommand{\needsattention}[1][]{%
{\color{red} Needs Attention!!\ifthenelse{\isempty{#1}}{}{: #1}}}

\newcommand{\per}{\mathcal{P}}
\DeclareMathOperator{\rig}{Rig}
\DeclareMathOperator{\rs}{relStr}
\newcommand{\decmap}{\delta}
\newcommand{\modelmap}{\mu}
\newcommand{\modelspace}{X}
\newcommand{\fixedmodel}{Z}
\newcommand{\fixedmodelmap}{\nu}
\newcommand{\gog}{\mathbf{\Gamma}}
\newcommand{\QItypes}{\mathrm{QItypes}}

\DeclareMathOperator{\PQS}{PQS} % power quasi-symmetries
\DeclareMathOperator{\bilipschitz}{BiLip} % biLipschitz maps
\DeclareMathOperator{\Isomgp}{Isom} % isometries
\DeclareMathOperator{\Isom}{Isom} % isometries
\DeclareMathOperator{\Igp}{\mathcal{I}} % isometries in quasi-isometry gp
\DeclareMathOperator{\QIgp}{\mathcal{QI}} % quasi-isometry group
\DeclareMathOperator{\QIsom}{QIsom} % set of quasi-isometries
\DeclareMathOperator{\CIsom}{CIsom} % set of coarse isometries
\DeclareMathOperator{\CIgp}{\mathcal{CI}} % coarse isometries in qi gp
\DeclareMathOperator{\Aut}{Aut} % automorphism group
\DeclareMathOperator{\Homeo}{Homeo}
\def\bdry{\partial} % boundary
\DeclareMathOperator{\len}{\ell}

\newcommand{\clball}[3][]{\overline{\mathcal{B}}^{#1}_{#2}(#3)}

\newcommand{\quot}[1]{\bm{#1}}
\newcommand{\lift}[1]{\widetilde{#1}}

\makeatletter
\newsavebox\myboxA
\newsavebox\myboxB
\newlength\mylenA

\newcommand*\xoverline[2][0.75]{%
    \sbox{\myboxA}{$\m@th#2$}%
    \setbox\myboxB\null% Phantom box
    \ht\myboxB=\ht\myboxA%
    \dp\myboxB=\dp\myboxA%
    \wd\myboxB=#1\wd\myboxA% Scale phantom
    \sbox\myboxB{$\m@th\overline{\copy\myboxB}$}%  Overlined phantom
    \setlength\mylenA{\the\wd\myboxA}%   calc width diff
    \addtolength\mylenA{-\the\wd\myboxB}%
    \ifdim\wd\myboxB<\wd\myboxA%
       \rlap{\hskip 0.5\mylenA\usebox\myboxB}{\usebox\myboxA}%
    \else
        \hskip -0.5\mylenA\rlap{\usebox\myboxA}{\hskip 0.5\mylenA\usebox\myboxB}%
    \fi}
\makeatother
\newcommand{\inv}[1]{\xoverline{#1}}
\def\act{\curvearrowright} % group action
\def\from{\colon\thinspace}

\def\into{\hookrightarrow}

\newcommand{\extN}{\overline{\mathbb{N}}}
\DeclareMathOperator{\image}{Im}
\newcommand\tree{T}
\DeclareMathOperator{\cyl}{Cyl}
\newcommand\ornaments{\mathcal{O}}
\newcommand\ornament{o}
\newcommand\partialorientation{\zeta}
\newcommand\imbalance{\Omega}
\newcommand\struc{\mathcal{S}}
\newcommand{\geo}{\zeta}
\newcommand{\terminal}{\tau}
\newcommand{\initial}{\iota}
\newcommand{\nullvar}{\texttt{`NULL'}}

\newcommand{\cmul}{M}% multiplicative constant
\newcommand{\cadd}{A}% additive constant
\newcommand{\cexp}{\alpha}% exponential constant

\DeclareMathOperator{\cone}{Con}% cone
\newcommand{\qi}{\phi}
\newcommand{\homeo}{\rho}
\newcommand{\iso}{\chi}
\newcommand{\attachmap}{\alpha}
\newcommand{\edgemap}{\varepsilon}
\newcommand{\model}{\mathrm{Model}}
\DeclareMathOperator{\map}{\mathrm{Map}}
\DeclareMathOperator{\starmap}{\map}

\DeclareMathOperator{\link}{lk}
\DeclareMathOperator{\sign}{\mathrm{sign}}

\newcommand{\verts}{\mathcal{V}}
\newcommand{\edges}{\mathcal{E}}

\mathtoolsset{centercolon} % makes the : and = in := be same height
\title{Quasi-isometries Between Groups with Two-Ended Splittings}  
\author{Christopher H. Cashen and Alexandre Martin}
\date{\today}

\begin{document}

\maketitle
\begin{abstract}
We construct `structure invariants' of a one-ended, finitely presented group  that describe
the way in which the factors of its JSJ decomposition over two-ended
subgroups fit together. 

For groups satisfying two technical conditions, these invariants
reduce the problem of quasi-isometry classification of such groups to
the problem of relative quasi-isometry classification of the factors
of their JSJ decompositions. 
The first condition is that their JSJ decompositions have two-ended cylinder stabilizers.
The second is that every factor in their JSJ decompositions is either
`relatively rigid' or `hanging'.
Hyperbolic groups always satisfy the first condition, and it is an
open question whether they always satisfy the second. 

The same methods also produce invariants that reduce the problem of
classification of one-ended hyperbolic groups up to homeomorphism of
their Gromov boundaries to the problem of classification of the
factors of their JSJ decompositions up to relative boundary
homeomorphism type.

\end{abstract}
%\tableofcontents

\section{Introduction}
Gromov proposed a program of classifying finitely generated groups up
to the geometric equivalence relation of quasi-isometry \cite{Gro84}. A natural approach to this problem is to first try to decompose the group into ``smaller'' pieces by means of a graph of groups decomposition, and then reduce the quasi-isometry classification problem to the problem of understanding the quasi-isometry types of the various vertex groups and the way these subgroups fit together.

The simplest such decomposition is the decomposition of a group as
a graph of groups over finite subgroups. 
Stallings's Theorem \cite{Sta68,Sta71} asserts that  a finitely generated group splits as an amalgamated product or an HNN extension over a finite group if and only if it has more than one end. In particular, the existence of a
splitting as a graph of groups
with finite edge groups is a quasi-isometry invariant. Finitely presented groups admit a maximal splitting over finite subgroups \cite{Dun85}, and a theorem of Papasoglu and Whyte \cite{PapWhy02} says that the
collection of quasi-isometry types of one-ended vertex groups of a
maximal decomposition of an infinite-ended, finitely presented group
is a complete quasi-isometry invariant.
This reduces the quasi-isometry classification problem to one-ended groups. 

We push this program to its next logical step, which is to decompose
one-ended groups over two-ended subgroups.
Papasoglu \cite{Pap05} shows that the existence of a splitting of a
finitely presented one-ended group over two-ended subgroups is
quasi-isometry invariant, provided that the group is not commensurable to a
surface group.
Moreover, such a group admits a maximal decomposition as a graph of
groups over two-ended subgroups, known as a JSJ decomposition, and
Papasoglu's results imply that quasi-isometries respect (in a
certain sense that will be made precise in \fullref{sec:qiinvariance}) the JSJ
decomposition.
In particular, the quasi-isometry types of the non-elementary vertex groups
of the JSJ decomposition are invariant under quasi-isometries.
Even more is true: In a non-elementary vertex group of the JSJ we see
the collection of conjugates of the two-ended subgroups corresponding
to incident edges of the JSJ decomposition. 
Quasi-isometries of the group must preserve the vertex groups together
with their patterns of
conjugates of incident edge subgroups.
We say a quasi-isometry must preserve the \emph{relative quasi-isometry
  type} of the vertex.
Such patterns have been exploited before
\cite{MosSagWhy11,Cas10,BisMj12, Mj12} to produce quasi-isometry
invariants from various `pattern rigidity' phenomena.
However, the quasi-isometry types, or relative quasi-isometry types,
of the vertex groups alone do not
give complete quasi-isometry invariants, because not all vertex
groups have the same coarse intersections:
Vertices that are adjacent in the Bass-Serre tree of the decomposition
have vertex groups that intersect in two-ended subgroups, while the
coarse intersection of non-adjacent vertices may or may not be
bounded.

In this paper,  we produce further quasi-isometry invariants of a finitely presented one-ended group from an appropriate JSJ decomposition over two-ended subgroups.  
Under a mild technical restriction, which is known to hold for large classes of groups (see \fullref{sec:rigidexamples} for a discussion),
we show that our invariants give complete quasi-isometry invariants, and thus reduce the quasi-isometry  classification problem for such groups to
relative versions of these problems in the vertex groups of their JSJ decomposition.\\

In the case of one-ended hyperbolic groups, a weaker classification is
possible, namely the classification up to homeomorphisms of Gromov
boundaries. 
Indeed, recall that quasi-isometric hyperbolic groups have homeomorphic Gromov
boundaries at infinity, but there are examples of hyperbolic groups
with homeomorphic boundary that are not quasi-isometric.
For a hyperbolic group, the existence of a splitting over a finite
subgroup amounts to having a disconnected Gromov boundary,
and is thus detected by the homeomorphism type of the boundary. 
Paralleling the results of Papasoglu and Whyte for quasi-isometries, Martin and \'Swi\k{a}tkowski \cite{MarSwi15}
show that hyperbolic groups with infinitely many ends have
homeomorphic boundaries if and only if they have the same sets of
homeomorphism types of boundaries of one-ended factors, reducing the classification problem to the case of one-ended hyperbolic groups.

For a one-ended hyperbolic group $G$ whose boundary is not a circle, Bowditch shows that the there exists a splitting over a two-ended subgroup if and only if there exists a cut pair in the boundary \cite{BowditchJSJ}.
From the structure of cut pairs in the boundary, he deduces the
existence of a simplicial tree on which $\Homeo(\bdry G)$ acts by isomorphisms.
Paralleling the quasi-isometry case, the action of $\Homeo(\bdry G)$
on this tree preserves the relative boundary homeomorphism types of
vertex stabilizers in $G$.

Using the same approach as for our quasi-isometry classification, we construct a complete system of  invariants of the homeomorphism type of the Gromov boundary of a hyperbolic group, which completely reduces the classification problem for hyperbolic groups to relative versions of this problem in the vertex groups of their JSJ decomposition. \\

Our results rely heavily on the existence of a canonical choice
of tree that is preserved by quasi-isometries, respectively,
homeomorphisms of the boundary, with the additional property  that any
such map induces maps of the same nature at the level of the vertex
groups.

In the case of hyperbolic groups, the tree  is Bowditch's canonical JSJ
tree constructed from cut pairs in the boundary.
For a more general finitely presented one-ended group $G$, let $\gog$ be a JSJ decomposition over two-ended subgroups.
Let $\tree:=\tree(\gog)$ be the Bass-Serre tree of $\gog$.
Commensurability of  stabilizers defines an equivalence relation
on edges of $\tree$ whose equivalence classes are called \emph{cylinders}.
Guirardel and Levitt \cite{GuiLev11} show that the dual tree to the
covering of $\tree$ by cylinders defines a new tree $\cyl(\tree)$, the \emph{tree of cylinders} of
$\tree$, with a cocompact $G$--action, and, in fact, this tree is
independent of $\gog$, so it makes sense to call it the \emph{tree of
  cylinders of $G$}, and denote it $\cyl(G)$.
It follows from Papasoglu's results, see \fullref{qiinvariance}, that a
quasi-isometry between finitely presented one-ended groups induces an
isomorphism between their trees of cylinders, and restricts to give a
quasi-isometry of each vertex group $G_v$ of the tree of cylinders that coarsely preserves the
pattern $\per_v$ of edge groups $G_e$, for edges $e$ incident to $v$.
When the edge stabilizer of $\cyl(G)$ in $G$ are two-ended then the
quotient graph of groups gives a canonical decomposition of $G$
over two-ended subgroups. 
Our results are strongest when additionally the cylinder stabilizers
are two-ended.
This is the case, in particular, when $G$ is hyperbolic, in which case
$\cyl(G)$ is equivariantly isomorphic to Bowditch's tree.
\\

Before giving more details about the ideas behind our classifications, 
we restrict to the case that $G$ is a one-ended hyperbolic group, and we
 adopt some notation that will allow us to to discuss simultaneously
 the quasi-isometry and boundary homeomorphism cases. 
Let $\starmap((X,\per^X),(Y,\per^Y))$ denote
alternately:
\begin{itemize}
\item The set of quasi-isometries from $X$ to $Y$ taking the 
collection of coarse equivalence classes of subsets $\per^X$ to the
collection of coarse equivalence classes of subsets $\per^Y$.
\item The set of boundary homeomorphisms from $\bdry X$ to $\bdry Y$
  taking the collection of subsets $\bdry \per^X$ to the collection of
  subsets $\bdry\per^Y$.
\end{itemize}
Similarly, $\starmap(G)$ is either $\QIgp(G)$ or $\Homeo(\bdry G)$. An element of $ \starmap(\cdots)$ will be referred to as a $\starmap$--equivalence.

The idea behind our classification result is the following. 
First remark that a finitely presented one-ended group is
quasi-isometric to a complex of spaces over the JSJ tree of cylinders
(this will be recalled in \fullref{sec:treesofspaces}). 
Moreover, such a decomposition as a complex of spaces is compatible
with Map-equivalences, in that a Map-equivalence between two finitely
presented one-ended groups coarsely preserves the structure of
complex of spaces, as follows from the aforementioned results of
Papasoglu and Bowditch. 
To determine whether two groups are Map-equivalent, we thus want to
decide whether their JSJ trees of cylinders are isomorphic and then try
to promote such an isomorphism to a Map-equivalence between the
complexes of spaces, hence, between the groups. 
In full generality, a $\starmap$-equivalence not only induces a
simplicial isomorphism between the trees of cylinders, but also
preserves extra information about the vertex groups that we want to
take into account: $\mbox{Map}$--class of vertex groups,  relative
$\mbox{Map}$--class with respect to incident edge groups, etc.  
In other words, we decorate the vertices of the trees of cylinders
with these additional pieces of information.
What we then want is to understand when there exists a decoration-preserving isomorphism between the tree of cylinders of two groups.
We want to add enough additional information so that a
decoration-preserving isomorphism between the trees of cylinders can 
be promoted to a $\mbox{Map}$--equivalence between the associated
groups. 
Moreover, we would like to have an algorithm telling us when two such decorated trees of cylinders are isomorphic. 

This latter point will be dealt with by generalizing to cocompact
decorated trees a theorem from graph theory giving a necessary and
sufficient condition for the universal covers of two graphs to be
isomorphic \cite{Lei82}, see \fullref{sec:structure}. 
To such a decorated tree we will associate a \textit{structure
  invariant} that completely determines the tree up to
decoration-preserving isomorphism.

To now get an intuition of the decorations we consider in this
article, let us start from the decoration of the tree of cylinders
that associates to each vertex $v$ the relative Map-equivalence
$\starmap(G_v, \cP_v)$ mentioned earlier, where $\cP_v$ is the
peripheral structure coming from the incident edge groups, and let us
try to promote a decoration-preserving isomorphism $\chi$ between trees of cylinders to a $\mbox{Map}$-equivalence between groups. 
The goal is to choose, for each vertex $v$, a $\starmap$--equivalence between $v$ and
$\chi(v)$, and piece them together to get $\mbox{Map}$-equivalence between groups.

The first potential problem is a realization problem. 
We know that  $v$ and $\chi(v)$ are $\starmap$--equivalent, but we
also need to know that there is such a $\starmap$--equivalence that
matches up peripheral subsets in the same way that $\chi$ matches up
edges incident to $v$ and $\chi(v)$.

The second potential problem is that the vertex
$\starmap$--equivalences must agree when their domains overlap. 
For boundary homeomorphisms the overlap is just the boundary of an
edge space, which is a pair of points, so this a matter of choosing
consistent \emph{orientations} on the edge spaces. 

Roughly, this is the phenomenon underlying the fact that two HNN extensions of the form

$$ G_+:= \langle G, t | tu\bar{t}=v \rangle ~ \mbox{ and } ~ G_-:= \langle G, t | tu\bar{t}= \bar{v} \rangle$$
might not be $\mbox{Map}$-equivalent in general, for a group $G$ and elements $u, v$ of $G$. 

It is worth noting that the orientation obstruction automatically 
disappears if all the edge groups contain an infinite dihedral group. 

For the classification up to boundary homeomorphisms, these three
obstructions, relative type, realizability, and orientation, are
essentially the only obstructions to constructing a
$\starmap$--equivalence between the groups.
We use these considerations to produce a finer decoration that yields
a structure invariant that is a complete invariant for boundary
homeomorphism type, see
\fullref{corollary:boundaryhomeomorphism}. 

In the case of quasi-isometries, orientation of the edge spaces is not
enough; we must choose $\starmap$--equivalence of the vertex spaces
that agree all
along the length of shared edge spaces.
There are two cases in which we can decide if this is possible. 
The first is that vertex spaces are extremely flexible so that we have
a lot of freedom to choose $\starmap$--equivalences and make them
agree on edge spaces.
This is the case for the so-called `hanging' vertices. 
The other case is the opposite one, in which the vertex space
is extremely `rigid' and we have very little choice about how to
choose the maps. 
In this case we define an invariant called a \emph{stretch factor}
that we then incorporate into the decoration.
If all vertices are either rigid or hanging then the decoration that
takes into account vertex relative quasi-isometry type, realizability,
orientation, and stretch factors yields a complete quasi-isometry
invariant, see \fullref{thm:qi}.
This notion of \emph{relative quasi-isometric rigidity} is known to hold for many classes of groups, see \fullref{sec:rigidexamples}.\\

The classifications we provide are inherently more technical than the
classifications for splittings over finite subgroups by Papasoglu and
Whyte and Martin and \'Swi\k{a}tkowski. 
Underlying those theorems is the fact that in splittings over finite
groups the vertex groups of the canonical decomposition are
essentially independent of one another, so only the pieces of the
decomposition matter. 
Our classification must handle not just the pieces, but also their
complex interactions.

\subsection{Applications}
Our main results give invariants that reduce classification problems
for finitely presented, one-ended groups to a relative version of the
classification problem on the vertex groups of a JSJ decomposition. 

Note that vertex groups of a two-ended JSJ decomposition of a
one-ended group are not necessarily one-ended, but we really have 
achieved a significant reduction of the classification problem because
the problem of classification of relative quasi-isometry types is far
more constrained than the general problem.
The most striking example of this phenomenon is a theorem of Cashen
and Macura \cite{CasMac11} that says that if a free group appears as a
vertex group of a JSJ decomposition then it is either hanging or it is
quasi-isometrically rigid relative to the peripheral structure induced
by incident edge groups. 
Moreover, in the rigid case the relative quasi-isometry question
reduces to a relative \emph{isometry} question in certain associated cube
complexes, which seems far more tractable. 

A sample application of our main results is a complete description, in
terms of \cite{CasMac11}, of the quasi-isometry and boundary
homeomorphism types of one-ended hyperbolic groups that split as
graphs of groups\footnote{ Such a decomposition can be algorithmically
  improved to a JSJ
decomposition of the same type \cite{Cas10splitting}.} with free vertex groups and cyclic edge groups. 

Consider the following example:
\begin{example}\label{ex:keyexample}
Let $G_i=\left<a, b, t\mid \inv{t}u_it=v_i\right>$, where $u_i$ and 
$v_i$ are words in $\left<a,b\right>$ given below. 
In each case $G_i$ should be thought of as an HNN extension of 
$\left<a, b\right>$ over $\mathbb{Z}$ with stable letter $t$. 

Let $u_0:=a$, $v_0:=ab\inv{a}\inv{b}^2$, $u_1:=ab$, $v_1:=a^2\inv{b}^2$, 
$u_2:=ab^2$, and $v_2:=a^2\inv{b}$.   

Then $G_0$, $G_1$, and $G_2$ are pairwise
non-quasi-isometric, but have homeomorphic Gromov boundaries.

The quasi-isometry claim  follows because the vertex groups in each case are
quasi-isometrically rigid relative to the incident edge groups, and
the stretch factor for $G_i$ is the ratio of the word lengths of
$u_i$ and $v_i$, which are different for each $i$; see
see \fullref{sec:obstructions} and
\fullref{sec:examplestretch}.
The groups have homeomorphic boundaries because the stretch factors
are the only difference, and boundary homeomorphism is not sensitive
to stretch factors; see \fullref{sec:obstructions} and \fullref{ex:exampleboundary}.
\end{example}

Our methods produce interesting invariants even in cases that the
groups are not hyperbolic or  that the
relative problems for the vertex groups is not be completely
understood:
\begin{example}\label{ex:mcg}
  Let $M$ be the  mapping class group of
  a non-sporadic hyperbolic surface, with a fixed finite generating set. 
Let $g_0,\,g_1,\,g_0',\,g_1'\in M$ be pseudo-Anosov elements that are
  not proper powers. 
Let $G:=M*_{\mathbb{Z}}M$
  and $G':=M*_{\mathbb{Z}}M$ be the amalgamated product groups
  obtained by identifying $g_0$ with $g_1$ and $g_0'$ with $g_1'$,
  respectively.
Then $G$ and $G'$ are not quasi-isometric if the ratio of
translation lengths of $g_0$ and $g_1$ is different, up to inversion,
from the ratio of translation lengths of $g_0'$ and $g_1'$.

This follows because mapping class groups are quasi-isometrically
rigid and the ratios of the translation lengths in this example are
the stretch factors; see \fullref{sec:rigidexamples} and \fullref{sec:stretchfactors}.
\end{example}

\subsection*{Acknowledgements}
The first author thanks Mladen Bestvina, Jason Behrstock, and Gilbert Levitt for interesting
conversations related to this work.
In particular, the idea that a version of Leighton's Theorem would
provide a concise description of our structure invariants was
suggested by Bestvina, and \fullref{ex:mcg} was suggested by Behrstock.

Both authors were supported by the European Research Council (ERC) grant of Goulnara
Arzhantseva, grant agreement \#259527 and the Erwin Schr\"odinger
Institute workshop ``Geometry of Computation in Groups''.
The first author is partially supported by the Austrian Science Fund
(FWF):M1717-N25.
The second author is partially supported by the Austrian Science Fund
(FWF):M1810-N25.

\section{Preliminaries}
We assume familiarity with standard concepts such as Cayley graphs,
ends of spaces, and (Gromov) hyperbolic geometry. See \cite{BriHae99} for
background.

A group is \emph{virtually cyclic} if it has an infinite cyclic
subgroup of finite index.
A group is \emph{non-elementary} is it is neither finite nor virtually
cyclic. 
A standard exercise is to show that a finitely generated, two-ended
group is virtually cyclic.

\subsection{Coarse geometry}
Throughout the paper the qualifier `coarse' is used to indicate `up to additive error'.

Let $(X,d_X)$ and $(Y,d_Y)$ be metric spaces.
Subsets of $X$ are \emph{coarsely equivalent} if they are bounded
Hausdorff distance from one another.
A subset $A$ is \emph{coarsely contained} in $B$ if $A$ is coarsely
equivalent to a subset of $B$.
Two maps $\phi$ and $\phi'$ from $X$ to $Y$ are said to be
\emph{coarsely equivalent} or \emph{bounded distance} from
each other if $\sup_{x\in X}d_Y(\phi(x),\phi'(x))<\infty$.
A map $\phi\from X\to Y$ is \emph{coarsely surjective} if its image is
coarsely equivalent to $Y$.

A map $\phi\from X \to Y$ is a \emph{controlled
  embedding}\footnote{This is more commonly called a `coarse embedding'.} if there exist
unbounded, non-decreasing real functions $\rho_0$ and $\rho_1$ such that for
all $x$ and $x'$ in $X$ we have:
\[\rho_0(d_X(x,x'))\leq
d_Y(\phi(x),\phi(x'))\leq \rho_1(d_X(x,x'))\]

There are several classes of controlled embeddings that have
special names.
If $\rho_i(r):=\cmul_ir$ for $i\in\{0,1\}$ then $\phi$ is
\begin{itemize}
\item an \emph{isometric embedding} if $\cmul_0=1=\cmul_1$,
\item an $\cmul$--\emph{similitude} if $\cmul_0=\cmul=\cmul_1$,
\item an $\cmul$--\emph{biLipschitz embedding} if 
  $\cmul_0^{-1}=\cmul=\cmul_1$.
\end{itemize}

If such a map is surjective then it is called, respectively, an \emph{isometry},
\emph{similarity}, or \emph{biLipschitz equivalence}.

For each of these, we can add the qualifier `coarse' and allow an
additive error, so that $\rho_0(r):=\cmul_0r-\cadd$ and
$\rho_1(r):=\cmul_1r+\cadd$.
For instance, $\phi$ is an $(\cmul,\cadd)$--coarse biLipschitz embedding if
$\rho_i$ is as above with  $\cmul_0=\cmul^{-1}$ and $\cmul_1=\cmul$.
In the `coarse' cases we drop the term `embedding' if the map is coarsely
surjective.

An $(\cmul,\cadd)$--coarse biLipschitz embedding is more commonly called an 
$(\cmul,\cadd)$--\emph{quasi-isometric embedding}.

Let $\QIsom(X,Y)$ denote the set of quasi-isometries from $X$ to $Y$,
and let $\CIsom(X,Y)$ denote the set of coarse isometries from $X$
to $Y$.
For a coarsely surjective map $\phi\from X\to Y$, a \emph{coarse inverse} is a map
$\inv{\phi}\from Y\to X$ such that $\inv{\phi}\circ\phi$ is coarsely
equivalent to $\mathrm{Id}_X$ and $\phi\circ\inv{\phi}$ is coarsely
equivalent to $\mathrm{Id}_Y$. 
If $\phi\in\QIsom(X,Y)$ then all coarse inverses of $\phi$ are
coarsely equivalent and belong to $\QIsom(Y,X)$.
If $\phi\in\CIsom(X,Y)$ then every coarse inverse of $\phi$ belongs to $\CIsom(Y,X)$.

Let $\Igp(X,Y)$, $\CIgp(X,Y)$, and $\QIgp(X,Y)$ denote, respectively,
the sets
 $\Isom(X,Y)$, $\CIsom(X,Y)$, and $\QIsom(X,Y)$, modulo
coarse equivalence.
When $Y=X$ we shorten the notation to 
 $\Igp(X)$, $\CIgp(X)$, $\QIgp(X)$, and each of these form a group
 under composition.

A subset of $\CIgp(X,Y)$ or $\QIgp(X,Y)$ is said to be \emph{uniform}
if there exists a $C$ such that every element of the subset is an
equivalence class of maps containing a $C$--coarse isometry or a
$(C,C)$--quasi-isometry, respectively. 

Quasi-isometries respect coarse equivalence of subsets.
If $\per$ is a set of coarse equivalence classes of subsets of $X$,
and $\per'$ is a set of coarse equivalence classes of subsets of $Y$,
let $\QIgp((X,\per),(Y,\per'))$ be the subset of $\QIgp(X,Y)$
consisting of quasi-isometries that
induce bijections between $\per$ and $\per'$.
Similarly, $\QIgp((X,\per)):=\QIgp((X,\per),(X,\per))$ is a subgroup of $\QIgp(X)$.

If $\qi\from X\to Y$ is a quasi-isometry, define $\qi_*\from
\QIgp(X)\to\QIgp(Y)$ by $\qi_*(\psi):=\qi\circ\psi\circ\inv{\qi}$.

A subset $Y\subset X$ is \emph{coarsely connected} if there exists
some $\cadd$ such that for every $y$
and $y'\in Y$ there is, for some $k$, a chain $y_0:=y,y_1,\dots,y_k:=y'$
of $(k+1)$ points
$y_i\in Y$ such that $d(y_i,y_{i+1})\leq \cadd$.

A \emph{geodesic} is an isometric embedding of a connected subset of $\mathbb{R}$.
A \emph{coarse geodesic} is a coarse geodesic embedding of a
coarsely connected subset of $\mathbb{R}$.
A \emph{quasi-geodesic} is a quasi-isometric embedding of a coarsely
connected subset of $\mathbb{R}$.

The space $X$ is said to be \emph{geodesic}, \emph{$\cadd$--coarse geodesic},
or \emph{$(\cmul,\cadd)$--quasi-geodesic} if for every pair of points in $X$
there exists, respectively, a geodesic, $\cadd$--coarse geodesic,
or $(\cmul,\cadd)$--quasi-geodesic connecting them.

Let $\llbracket X\rrbracket$ denote the set of proper geodesic metric spaces
quasi-isometric to $X$.
If $\per$ is a set of coarse equivalence classes of subsets of $X$, let
$\llbracket (X,\per)\rrbracket$ denote the set of pairs $(Y,\per')$
where $Y$ is a geodesic metric space and $\per'$ is a collection of
coarse equivalence classes of subsets of $Y$ such that there exists a
quasi-isometry from $X$ to $Y$ that induces a bijection from $\per$ to $\per'$.
We call $\llbracket X\rrbracket$ the \emph{quasi-isometry type} of $X$
and $\llbracket (X,\per)\rrbracket$ the \emph{relative quasi-isometry
  type} of $(X,\per)$.

If $L$ is a path connected subset of a geodesic metric space $(X,d_X)$, let $d_L$ denote the induced
length metric on $L$.
A \emph{quasi-line} in $X$ is a path connected subset $L$ such that
$(L,d_L)$ is quasi-isometric to $\mathbb{R}$ and the inclusion map of
$L$ into $X$ is a controlled embedding.

We define a \emph{peripheral structure $\per$}  on geodesic metric space $X$ to be a collection of coarse
equivalence classes of quasi-lines.
In particular, 
if $G$ is a finitely generated group and $\mathcal{H}$ is a finite collection
of two-ended subgroups of $G$, then $\mathcal{H}$ induces a peripheral
structure consisting of distinct coarse equivalence classes of
 conjugates of elements of $\mathcal{H}$.

\subsection{Graphs of groups}
Let $\Gamma$ be a finite oriented graph. 
Let $\verts\Gamma$ be the set of vertices of $\Gamma$, and let
$\edges^+\Gamma$ be the set of oriented edges.
For $e\in\edges^+\Gamma$, let $\initial(e)$ be its initial vertex, and let
$\terminal(e)$ be its terminal vertex.
For each $e\in\edges^+\Gamma$ formally define $\inv{e}$ to be an
inverse edge to $e$ with   $\initial(\inv{v}):=\terminal(e)$,
$\terminal(\inv{e}):=\initial(e)$, and $\bar{\bar{e}}:=e$.
The inverse edge $\inv{e}$  should be thought of as $e$ traversed against its given
orientation.
Let $\edges^-\Gamma$ denote the set of inverse edges, and $\edges\Gamma:=\edges^+\Gamma\cup\edges^-\Gamma$.

A graph of groups $\gog:=(\Gamma,\{G_{\gamma}\}_{\gamma\in\verts\Gamma\cup\edges^+\Gamma},\{\edgemap_e\}_{e\in\edges^+\Gamma})$ consists of a
finite directed graph $\Gamma$, groups $G_\gamma$ for
$\gamma\in\verts\Gamma\cup\edges^+\Gamma$ such that
$G_e<G_{\initial(e)}$ for $e\in\edges^+\Gamma$, and injections $\edgemap_e\from G_e\into
G_{\tau(e)}$ for $e\in\edges^+\Gamma$.

For symmetry in the notation it is convenient to define $G_{\bar{e}}:=G_e$ for
each $\inv{e}\in\edges^-\Gamma$ and let
$\edgemap_{\bar{e}}$ denote the inclusion of $G_{\bar{e}}:=G_e$ into
$G_{\terminal(\bar{e})}:=G_{\initial(e)}$.

A graph of groups $\gog$ has an associated fundamental group $G=G(\gog)$ obtained by
amalgamating the vertex groups  over the edge groups \cite{Ser03}.
We say that $\gog$ is \emph{a graph of groups decomposition} of $G$.

The Bass-Serre tree $\tree:=\tree(\gog)$ of $\gog$ is the tree on which $G$ acts
without edge inversions, such that $\lquotient{G}{\tree}=\Gamma$ and such
that the
stabilizer $G_{t}$ of  ${t}\in\verts\tree\cup\edges\tree$ is a conjugate in $G$ of
the group $G_{\quot{t}}$, where $\quot{t}$ is the image of $t$ under
the quotient map $\tree\to\Gamma$.

Throughout we use the notation $\quot{t}$ to denote the image of $t$
in $\Gamma$.
Conversely, for each $\gamma\in\verts\Gamma\cup\edges\Gamma$ we choose
some lift $\lift{\gamma}$ of $\gamma$ to $\tree$.
Given a maximal subtree in $\Gamma$ we can, and do, choose lifts of
vertices and edge in the subtree to get a subtree in $\tree$.

\begin{definition}
Given a vertex group $G_v$ of $\gog$, the \emph{peripheral structure 
  coming from incident edge groups}, $\per_v$, is the set of distinct
coarse equivalence classes in $G_v$ of $G_v$--conjugates of the images of the edge
injections $\edgemap_e\from G_e\into G_v$ for edges $e\in\edges\Gamma$ with $\tau(e)=v$.  
\end{definition}
\begin{definition}
Given a  vertex $v$ of $T(\gog)$, the \emph{peripheral structure 
  coming from incident edge groups}, $\per_{{v}}$, is the set of distinct
coarse equivalence classes in the stabilizer subgroup $G_{{v}}$
of ${v}$
of stabilizers of incident edge groups.
\end{definition}
It is immediate from the definitions that the quotient by the
$G$--action identifies 
$(G_v,\per_v)$ and $(G_{\quot{v}},\per_{\quot{v}})$.

We are interested in graphs of groups in which the edge groups are
two-ended, hence virtually cyclic.
Commensurability of edge stabilizers defines an equivalence relation
on the edges of the Bass-Serre tree $\tree$ of such a splitting.
The equivalence classes of edges are called \emph{cylinders}.
Every cylinder is a subtree of $\tree$ \cite[Lemma~4.2]{GuiLev11}.
It follows that we get another tree $\cyl(\tree)$, called the \emph{tree of
  cylinders} of $\tree$, by taking the dual tree to the covering of
$\tree$ by cylinders.

\begin{definition}\label{def:acylindrical}
  A graph of groups with two-ended edge groups is $k$--\emph{acylindrical}
  if the cylinders of its Bass-Serre tree have  diameter at most $k$.
It is \emph{acylindrical} if there exists $k$ such that it is $k$--acylindrical. 
\end{definition}

Let $C$ be a cylinder in $\tree$ and let $\mathrm{Stab}(C)$ be the
stabilizer of $C$ in $G$. Choose an infinite order element $z$ in
$G_e$ for some edge $e\in C$.
For any element $g\in\mathrm{Stab}(C)$, $G_e$ and $G_{ge}$ are
commensurable, virtually cyclic groups. 
Since $\langle z\rangle$ is a finite index subgroup of $G_e$, there
exist non-zero $a$ and $b$ such that $gz^ag^{-1}=z^b$.
Define $\Delta(g)=\frac{a}{b}$.
This defines a homomorphism $\Delta\from
\mathrm{Stab}(C)\to \mathbb{Q}^*$, called the
\emph{modular homomorphism} of $C$, that is independent of the choice
of $z$.

\begin{definition}\label{def:unimodular}
A cylinder is called \emph{unimodular} if the image of its modular
homomorphism is in $\{-1,1\}$.

  A graph of groups with two-ended edge groups is \emph{unimodular} if
  all of its cylinders are unimodular.
\end{definition}

We also define a \emph{modulus} $\hat{\Delta}$ on pairs of
edges of a cylinder $C$:
\begin{definition}\label{def:extendedmodulus}
Let $e_0$ and $e_1$ be edges in $C$. 
Let $\langle z_0\rangle<G_{e_0}$ and $\langle z_1\rangle<G_{e_1}$ be
infinite cyclic subgroups of minimal index.
Define $\hat{\Delta}(e_0,e_1)=\frac{[\langle z_1\rangle:\langle
  z_0\rangle\cap\langle z_1\rangle]}{[\langle z_0\rangle:\langle
  z_0\rangle\cap\langle z_1\rangle]}$.
\end{definition}

It is easy to check that $\hat{\Delta}$ does not depend on the choice
of minimal index infinite cyclic subgroups, and that $\hat{\Delta}(e,ge)=|\Delta(g)|$.

\subsection{The JSJ tree of cylinders}\label{sec:jsj}
\subsubsection{Definitions}
A JSJ decomposition of a finitely presented, one-ended group $G$ over
two-ended subgroups  is a
graph of groups with two-ended edge groups that encodes all splittings
of $G$ over two-ended subgroups.
Equivalent descriptions of such decompositions appear in Dunwoody and
Sageev \cite{DunSag99}, Fujiwara and Papasoglu \cite{FujPap06}, and
Guirardel and Levitt \cite{GuiLev09}.
See also Rips and Sela \cite{RipSel97}.

Following Papasoglu \cite{Pap05}, we will give a geometric
description of JSJ decompositions.
First, we need some terminology.

A quasi-line $L$ is \emph{separating} if its complement has at least
two \emph{essential}
components, that is, components that are not contained in any finite neighborhood of $L$.
In particular, if $G$ splits over a two-ended subgroup then that
two-ended subgroup is bounded distance from a separating quasi-line.
Separating quasi-lines \emph{cross} if each travels arbitrarily deeply
into two different essential complementary components of the other.

Let $G_v$ be a vertex group in a graph of groups decomposition.
Let $\per_v$ be the peripheral structure on $G_v$ coming from incident
edge groups.
Let $\Sigma$ be a hyperbolic pair of pants.
Let $\per_{\bdry\Sigma}$ be the peripheral structure on the universal
cover $\tilde{\Sigma}$ of $\Sigma$
consisting of the coarse equivalence classes of the components of the
preimages of the boundary curves.
\begin{definition}\label{def:hanging}
We say $v$ is \emph{hanging}\footnote{After the 
  `quadratically hanging' vertex groups of Rips and Sela \cite{RipSel94}.} if $(G_v,\per_v)$ is quasi-isometric to 
$(\tilde{\Sigma},\per_{\bdry\Sigma})$. 
We say $v$ is \emph{rigid} if it is not two-ended, not hanging, and 
does not split over a two-ended subgroup relative to its incident edge 
groups.  
\end{definition}

The rigid and hanging terminology extends to vertices in $\tree(\gog)$
in the obvious way.

\begin{definition}\label{def:JSJ}
  Let $G$ be a finitely presented one-ended group that is not
  commensurable to a surface group.
A \emph{JSJ decomposition} of $G$ is a (possibly
  trivial) graph of groups
  decomposition $\gog$ with two-ended edge groups satisfying the
  following conditions:
  \begin{enumerate}[(a)]
  \item Every vertex group is either two-ended, hanging, or
    rigid.\label{item:vertextypes}
\item If $v$ is a valence one vertex with two-ended vertex group then
  the incident edge group does not surject onto $G_v$.
\item Every cylinder in the Bass-Serre tree of $\gog$ that
  contains exactly two hanging vertices also contains a rigid vertex.\label{item:maxhanging}
  \end{enumerate}
\end{definition}

This definition is equivalent to those cited above. 
The essential facts are that: 
\begin{enumerate}
\item Hanging vertices contain crossing pairs
of separating quasi-lines.
\item Every pair of crossing separating 
quasi-lines is coarsely contained in a conjugate of a hanging vertex
group.
\item A separating quasi-line that is not crossed by any other
  separating quasi-line is coarsely equivalent to a conjugate of an edge group.
\item Every edge group is coarsely equivalent to a separating
  quasi-line that is not crossed by any other separating quasi-line.\label{item:edgenocross}
\end{enumerate}
\begin{remark}
Condition (\ref{item:maxhanging}) implies that the hanging vertex 
groups are maximal hanging, which is necessary for item \ref{item:edgenocross}.
\end{remark}
\begin{remark}
 The case that a vertex group is the fundamental group of a pair of
pants and the incident edge groups glue on to the boundary curves is
called `rigid' in the usual JSJ terminology because there are no
splittings of the pair of pants group relative to the boundary
subgroups. 
Algebraically, such a vertex behaves like our rigid vertices, but
geometrically this is a hanging vertex.
\end{remark}

In general a group does not have a
unique JSJ decomposition, but rather a deformation space of JSJ decompositions
\cite{For03,GuiLev09}.
Furthermore, all JSJ decompositions are in the same deformation space,
which means any one can be transformed into any other by meas of a
finite sequence of moves of a prescribed type.
The tree of cylinders of a decomposition depends only on the
deformation space \cite[Theorem~1]{GuiLev11}, up to $G$--equivariant
isomorphism, so there is a unique JSJ tree of cylinders. 

\begin{definition}
  Let $G$ be a finitely presented one-ended group. 
The \emph{JSJ tree of cylinders} $\cyl(G)$ of $G$ is the tree of cylinders of
the Bass-Serre tree of any JSJ decomposition of $G$ over two-ended subgroups.
\end{definition}

The quotient graph of groups $\lquotient{G}{\cyl(G)}$ gives a canonical 
decomposition of $G$.
It is canonical in the sense that its Bass-Serre tree is
$G$--equivariantly isomorphic to the tree of cylinders of any JSJ
decomposition of $G$ over two-ended subgroups.
However, such a graph of cylinders is \emph{not necessarily} a JSJ
decomposition, and it does not even have two-ended edge groups, in
general. 
We return to this issue in \fullref{sec:restrictedtoc}.

\subsubsection{Quasi-isometry invariance of the JSJ tree of cylinders}\label{sec:qiinvariance}
Since quasi-isometries coarsely preserve quasi-lines, and preserve the
crossing and separating properties of quasi-lines,
the following version of quasi-isometry invariance of JSJ
decompositions follows from Papasoglu's work:
\begin{theorem}[{cf \cite[Theorem~7.1]{Pap05}}]\label{qiinvariance}
  Let $G$ and $G'$ be finitely presented one-ended groups. 
Suppose $\qi\from G\to G'$ is a quasi-isometry.
Then there is a constant $C$ such that $\qi$ induces an isomorphism
$\qi_*\from\cyl(G)\to\cyl(G')$ that preserves vertex type ---
cylinder, hanging, or rigid --- and for $v\in\verts\cyl(G)$ takes $G_v$
to within distance $C$ of $G'_{\qi_*(v)}$.
\end{theorem}
\begin{proof}
  Two-ended subgroups of $G$ are bounded distance from each other if
  and only if they are commensurable, so cylinders are coarse
  equivalence classes of edge and two-ended vertex stabilizers. 
These are exactly the coarse equivalence classes of separating
quasi-lines that are not crossed by other separating quasi-lines, so
quasi-isometries induce a bijection between cylinders. 
The remaining statements are from Papasoglu \cite[Theorem~7.1]{Pap05}
for the rigid and hanging vertex stabilizers and Vavrichek
\cite{Vav13} for the cylinder stabilizers.
\end{proof}
\begin{corollary}\label{qirestricted}
  If $v$ is a rigid or hanging vertex in $\cyl(G)$ then
\[\qi_v:=\pi_{\qi_*(v)}\circ\qi|_{G_v}\in\QIgp((G_v,\per_v),(G'_{\qi_*(v)},\per_{\qi_*(v)}))\]
where $\pi_{\qi_*(v)}$ takes the image of $\qi|_{G_v}$ to
$G'_{\qi_*(v)}$ by closest point projection.
\end{corollary}
\begin{remark}
  $\pi_{\qi_*(v)}$ is coarsely well defined since $\qi(G_v)$ is
  within distance $C$ of $G'_{\qi_*(v)}$.
\end{remark}

\subsubsection{Boundary homeomorphism invariance of the JSJ tree of cylinders}
In the case of a one-ended hyperbolic group that is not cocompact Fuchsian, Bowditch constructed a canonical JSJ splitting of the group directly from the combinatorics of the local cut points of the Gromov boundary of the group \cite{BowditchJSJ}. In this case, he proves that the JSJ tree is unique, and such a tree is thus equivariantly isomorphic to the JSJ tree of cylinders of the group. As a homeomorphism between the Gromov boundaries of two hyperbolic groups preserves the topology, we get the following: 

\begin{theorem}[{cf \cite{BowditchJSJ}}]\label{boundaryinvariance}
  Let $G$ and $G'$ be one-ended hyperbolic groups that are not cocompact Fuchsian. 
Suppose $\homeo\from \bd G\to \bd G'$ is a homeomorphism between their Gromov boundaries.
Then  $\homeo$ induces an isomorphism
$\homeo_*\from\cyl(G)\to\cyl(G')$ that preserves vertex type ---
cylinder, hanging, or rigid --- and for $v\in\verts\cyl(G)$ the
homeomorphism $\homeo$ restricts to a homeomorphism $\homeo|_{\bd
  G_v}\from \bd G_v\to \bd G'_{\homeo_*(v)}$.
\end{theorem}
\begin{corollary}\label{corollary:boundryhomeorestriction}
For every vertex $v\in\cyl(G)$:
  \[\homeo_v:=\homeo|_{\bdry G_v}\in\Homeo((\bdry G_v,\bdry\per_v),(\bdry G'_{\homeo_*(v)},\bdry\per'_{\homeo_*(v)}))\]
\end{corollary}

\subsubsection{Improved invariants from restrictions on the JSJ tree of cylinders}\label{sec:restrictedtoc}
The JSJ tree of cylinders $\cyl(G)$ of a finitely presented one-ended
group $G$ suffices for the definition of the basic quasi-isometry invariants of \fullref{sec:structureinvariants}.
These are far from complete invariants, however.

We can refine the invariants in restricted classes of groups. 
For instance, if the edges of $\cyl(G)$ have two-ended stabilizers in
$G$ then we can define stretch factors as in
\fullref{sec:stretchfactors}.

When the cylinder stabilizers are two-ended then the full power of
\fullref{sec:vertexconstraints} can be brought to bear. 
In this case the graph of cylinders is a canonical JSJ decomposition
of $G$ over two-ended subgroups.  
This is the case of chief interest for this paper, and this is always the case if $G$ is hyperbolic.

If the cylindrical vertices of $\cyl(G)$ have two-ended stabilizers 
then they are all finite valence in $\cyl(G)$.
Furthermore, if $\gog$ is a JSJ decomposition of $G$ over two-ended subgroups
and $v\in\tree(\gog)$ is a vertex whose stabilizer is rigid or hanging
then $v$ belongs to more than one cylinder, so $\cyl(G)$ has a vertex
corresponding to $v$ with the same stabilizer subgroup in $G$.
Thus,  $\cyl(G)$ is bipartite, with one part, $V_C$, consisting of
finite valence cylindrical vertices, one for each cylinder of
$\tree(\gog)$, and  the other part consisting of the sets of vertices
$V_H$ and $V_R$, which are all of infinite valence, corresponding to hanging and rigid vertices of $\tree(\gog)$,
respectively.

See \cite[Proposition~5.2]{GuiLev11} for a general result about when
the tree of cylinders gives a JSJ decomposition.

\subsection{Trees of spaces}\label{sec:treesofspaces}
Let $\tree$ be an oriented simplicial tree.
For each vertex $v\in\tree$ let $X_v$ be a metric space.
For each edge $e\in\edges^+\tree$ let $X_e$ be a subspace of $X_{\initial(e)}$, and let $\attachmap_e\from X_e\to X_{\terminal(e)}$ be a  map such that for $X_{\bar{e}}:=\attachmap_e(X_e)$ there exists a map $\attachmap_{\bar{e}}\from X_{\bar{e}}\to X_{\initial(e)}$ such that $\attachmap_{\bar{e}}\circ\attachmap_e$ is bounded distance from $\mathrm{Id}_{X_e}$ and $\attachmap_e\circ\attachmap_{\bar{e}}$ is bounded distance from $\mathrm{Id}_{X_{\bar{e}}}$.

Let $X$ be the quotient of the set
\[\underset{v\in\verts\tree}{\coprod}X_v\sqcup\underset{e\in\edges^+\tree}{\coprod}\;\underset{x\in X_e}{\coprod}\{x\}\times e\]
by the identifications $x\sim (x,\initial(e))$ and $\attachmap_e(x)\sim (x,\terminal(e))$.
We call 
\[X:=X(\tree,\{X_t\}_{t\in\verts\tree\cup\edges\tree},\{\attachmap_e\}_{e\in\edges\tree})\] a \emph{tree of spaces over $\tree$}.
The $X_v$ are called \emph{vertex spaces} and the $X_e$ are called \emph{edge spaces}.
The sets $\{x\}\times e$ we call \emph{rungs}, and metrize them as unit intervals.
The maps $\attachmap_e$ are called \emph{attaching maps}.

We say $X$ has \emph{locally finite edge patterns} if for every vertex $v$, every $x\in X_v$, and every $R\geq 0$, there are finitely many  $e\in\edges\tree$ such that $X_e$ intersects $\clball[v]{R}{x}:=\{y\in X_v\mid d_{X_v}(x,y)\leq R\}$.

A \emph{$k$--chain} is a sequence of points $x_0,x'_0,x_1,\dots,x'_k$ such that $x_i$ and $x'_i$ are contained in a common vertex space or rung for all $0\leq i\leq k$ and $x'_i$ and $x_{i+1}$ are contained in a common vertex space or rung for $0\leq i<k$.
The \emph{length} of a $k$--chain is $d(x_0,x'_0)+\sum_{0< i\leq k}\left(d(x'_{i-1},x_i)+d(x_i,x'_i)\right)$, where the distance terms are interpreted in the vertex space or rung containing the pair.
The quotient pseudo-metric on $X$ defines the distance between two points to be the infinum of lengths of chains joining them.
When $x$ and $x'$ are points in vertex spaces then for a $k$--chain joining them we have that $x'_i$ and $x_{i+1}$ are the endpoints of a rung for each $0\leq i<k$, so the length of the chain is at least $k$.
The following properties follow easily from this observation.
\begin{lemma}
For a tree of spaces $X$:
  \begin{itemize}
  \item The quotient pseudo-metric is a metric.
\item A ball of radius at most 1 in a vertex space is isometrically embedded in $X$.
\item The rungs are isometrically embedded.
\end{itemize}
\end{lemma}

\begin{lemma}\label{lemma:propergeodesictreeofspaces}
If $X$ is a tree of spaces over $\tree$ such that each vertex space is proper and geodesic, each edge space is closed and discrete, and edge patterns are locally finite in each vertex space, then $X$ is a proper geodesic space.
\end{lemma}
\begin{proof}
For properness it suffices to show a closed ball centered at a point in a vertex space is compact, since a ball centered at an interior point in a rung is contained in the union of balls of radius 1 larger centered at the endpoints of the rung. 

Take $R>0$ and a point $x\in X_v$ for some vertex $v$.
Let $N_0:=\clball[v]{R}{x}$.
Since edge patterns are locally finite, there are finitely many edges $e$ such that $X_e$ intersects $\clball[v]{R}{x}$.
For $X_e\cap\clball[v]{R}{x}\neq \emptyset$, the set $X_e\cap\clball[v]{R}{x}$ is finite, since $X_e$ is closed and discrete and $\clball[v]{R}{x}$ is compact, so $\clball[\terminal(e)]{R-1}{\attachmap_e(X_e\cap\clball[v]{R}{x})}$ is compact.

Let $N_{1,1},\dots,N_{1,k_1}$ be these finitely many compact sets.
For each $N_{1,i}$ there are finitely many $e$ such that $X_e$ intersects $N_{1,i}$. Let $N_{1,i,1},\dots,N_{1,i,k_{1,i}}$ be the finitely many closed $R-2$ balls about the sets $\attachmap_e(X_e\cap N_{1,i})$ in their respective vertex spaces. 
Continue in this way for $R$ steps. Let $N$ be the union of the finitely many $N_*$, along with the finitely many rungs connecting them. 
This is a finite union of compact sets, so it is compact, and it contains $\clball{R}{x}$, since it contains every chain of length less than or equal to $R$ starting from $x$. 
Thus, $X$ is proper. 

To see that $X$ is geodesic, first consider two points $x$ and $x'$ contained in vertex spaces. 
By definition of the metric, there exists a sequence of chains joining $x$ and $x'$ with length decreasing to $d(x,x')$. 
Since $x$ and $x'$ are in vertex spaces, the length of a $k$--chain joining them is at least $k$, so there is a subsequence of approximating chains that are all $k$--chains for some fixed $k\leq d(x,x')$.
Let $x_{0,i}:=x,x_{0,i}',\dots,x'_{k,i}:=x'$ be the $i$--th $k$--chain in the sequence. 
Set $x_0:=x$.
Now, $(x'_{0,i})$ is a bounded sequence in $X_{v_0}$. Since edge spaces are closed and discrete and edge patterns are locally finite, $(x'_{0,1})$ contains a constant subsequence.
Pass to a subsequence of chains such that $(x'_{0,i})$ is constant and define $x'_0$ to be its value. 
Since edge patterns are locally finite, $x'_0$ belongs to finitely many $X_e$.
Since attaching maps are coarsely invertible and edge spaces are closed and discrete, there are only finitely many rungs with $x'_0$ as an endpoint. 
Therefore, the sequence $(x_{1,i})$ takes only finitely many values. 
Choose one, define it to be $x_1$, and pass to the subsequence with $x_{1,i}=x_1$.
Repeat this process $k-1$ times, passing each time to a subsequence of chains with one additional constant coordinate. 
We are left with a constant sequence of chains with length decreasing to $d(x,x')$, so the length is equal to $d(x,x')$.
Each vertex space and rung are geodesic spaces, so we connect successive points in the chain by geodesics in the vertex space or rung containing them to get a path with length equal to length of the chain, which is therefore a geodesic from $x$ to $x'$. 

If $x'$ is an interior point of a rung and $x$ is in a vertex space then find geodesics from $x$ to the two endpoints of the rung and concatenate with the subsegment of the rung leading to $x'$. The shorter of these two paths is a geodesic. A similar argument works if both $x$ and $x'$ are interior points of rungs. 
\end{proof}

\subsection{Algebraic trees of spaces}\label{sec:models}
Let $\gog$ be a graph of finitely generated groups.
In this section we construct a tree of spaces over $\tree:=\tree(\gog)$ that is quasi-isometric to $G:=G(\gog)$. 
The idea is to take the vertex spaces to be Cayley graphs of the vertex stabilizers and use the edge injections of $\gog$ to define the attaching maps.
The construction is standard, but there is some bookkeeping involved that will be useful in \fullref{sec:geometricmodel}.

For each $v\in\verts\Gamma$, choose a finite generating set for $G_v$
and coset representatives $h_{(v,i)}$ for $G/G_v$.
For each  $e\in\edges\Gamma$ choose coset representatives $g_{(e,i)}$ of $G_{\initial(e)}/G_e$.

For each  $v\in\verts\Gamma$, choose a
lift $\lift{v}\in\verts\tree$.
For each edge $e\in\edges\Gamma$, choose a lift
$\lift{e}\in\edges\tree$ with
$\initial(\lift{e})=\lift{\initial(e)}$.
Define $f_e:=h_{(\initial(e),j)}g_{(e,i)}$ for $i$ and $j$
such that $f_e\bar{\lift{e}}=\lift{\bar{e}}$.
Given a maximal subtree of $\Gamma$ it is possible to choose lifts so that 
$f_e=1$ for all edges $e$ such that $e$ or $\inv{e}$ belongs to the maximal subtree.

For $t\in\verts\tree\cup\edges\tree$, let $\quot{t}\in\Gamma$ denote the image
of $t$ under the quotient
by the $G$--action.

Let $v$ be a vertex of $\tree$. 
There is a representative $h_{(\quot{v},i)}$ such that
$v=h_{(\quot{v},i)}\lift{\quot{v}}\in\tree$.
 Define $Y_v$ to be a copy
of the Cayley graph of $G_{\quot{v}}$ with respect to the chosen generating
set, which we identify with the coset $h_{(\quot{v},i)}G_{\quot{v}}$ via left
multiplication by $h_{(\quot{v},i)}$. 

Take the edge spaces to be cosets of the edge stabilizers of
$\gog$, and define attaching maps via containment.
Specifically, for an edge $e=h_{(\quot{v},i)}g_{(\quot{e},j)}\lift{\quot{e}}\in\tree$ with $\initial(e)=v$ we define
$Y_{e}:=h_{(\quot{v},i)}g_{(\quot{e},j)}G_{\quot{e}}\subset
h_{(\quot{v},i)}G_{\quot{v}}=Y_v$
with attaching map: 
\[\attachmap_e(x):=h_{(\quot{v},i)}g_{(\quot{e},j)}f_{\quot{e}}\edgemap_{\quot{e}}(g_{(\quot{e},j)}^{-1}h_{(\quot{v},i)}^{-1}x)\subset Y_{\terminal(e)}\]

Let $Y:=Y(\tree,\{Y_t\},\{\attachmap_e\})$ be the resulting tree of
spaces over $\tree$, which we call an \emph{algebraic tree of spaces}
for $\gog$.

\begin{lemma}\label{lemma:algebraictreeofspaces}
  An algebraic tree of spaces $Y$ for $\gog$ is quasi-isometric to $G(\gog)$.
\end{lemma}
\begin{proof}
By \fullref{lemma:propergeodesictreeofspaces},  $Y$ is a proper
  geodesic metric space.
The group $G(\gog)$ acts by left multiplication, properly
  discontinuously and cocompactly by isometries on $Y$, so they are
  quasi-isometric, by the Milnor-\v{S}varc Lemma.
\end{proof}

 \subsection{Trees of maps}

\subsubsection{Trees of quasi-isometries}
\begin{proposition}\label{prop:treeofqis}
  Suppose $X$ and $X'$ are trees of spaces over $\tree$ and $\tree'$,
  respectively.
Suppose that $\iso\from\tree\to\tree'$ is an isomorphism.
Suppose there exists $\cmul\geq 1$ and $\cadd\geq 0$ 
 such that:
  \begin{itemize}
  \item  For each  $v\in\verts\tree$ there is an $(\cmul,\cadd)$--quasi-isometry $\qi_v\from X_v\to X'_{\iso(v)}$ and a quasi-isometry inverse $\bar{\qi}_v\from X'_{\chi(v)}\to X_v$.
\item For every $e\in\edges\tree$, the space $\qi_{\initial(e)}(X_e)$ is $\cadd$--coarsely equivalent to $X'_{\iso(e)}$ in $X'_{\iso(\initial(e))}$, and the space $\bar{\qi}_{\initial(e)}(X'_{\chi(e)})$ is $\cadd$--coarsely equivalent to $X_e$ in $X_{\initial(e)}$.
\item For every  $e\in\edges\tree$ and every $x\in X_e$ there exists a point $x'\in X'_{\iso(e)}$ such that $d(\qi_{\initial(e)}(x),x')\leq \cadd$ and $d(\qi_{\terminal(e)}(\attachmap_e(x)),\attachmap'_e(x'))\leq\cadd$.
\item For every  $e\in\edges\tree$ and every $x'\in X'_{\iso(e)}$ there exists a point $x\in X_e$ such that $d(\bar{\qi}_{\initial(e)}(x'),x)\leq \cadd$ and $d(\bar{\qi}_{\terminal(e)}(\attachmap'_{\iso(e)}(x')),\attachmap_e(x))\leq\cadd$.
  \end{itemize}
Then there is a quasi-isometry $\qi\from X\to X'$ with $\qi|_{X_v}=\qi_v$ for each vertex $v\in\tree$.
\end{proposition}

\begin{definition}
  A collection of quasi-isometries $(\qi_v)$ satisfying the conditions
  given in \fullref{prop:treeofqis} is called \emph{a tree of
    quasi-isometries over $\iso$ compatible with $X$ and $X'$}. 
\end{definition}

\begin{proof}[Proof of \fullref{prop:treeofqis}]
It suffices to consider the unions of the vertex spaces, which form
coarsely dense subsets of $X$ and $X'$, and define $\qi$ by
$\qi|_{X_v}:=\qi_v$. 
Define a coarse inverse $\bar{\qi}$ to $\qi$ by $\bar{\qi}|_{X'_{\iso(v)}}:=\bar{\qi}_v$.

Suppose $x$ and $x'$ are points in vertex spaces of $X$, and let $x_0,x_0',x_1,\dots,x'_k$ be a $k$--chain joining them with length at most $d(x,x')+1$.
Let $X_{v_i}$ be the vertex space containing $x_i$ and $x_i'$, and let $e_i$ be the edge such that $x'_{i-1}$ and $x_i$ are endpoints of a rung over $e_i$.

Let $y_0:=\qi_{v_0}(x_0)$ and $y'_k:=\qi_{v_k}(x'_k)$.
By hypothesis, for each $x'_i$ there exists a $y'_i\in X'_{\iso(e_i)}$ such that $d(\qi_{\initial(e)}(x'_i),y'_i)\leq \cadd$ and $d(\qi_{\terminal(e)}(x_{i+1}),y_{i+1})\leq\cadd$, where $y_{i+1}:=\attachmap'_{\iso(e)}(y'_i)$.
This gives a $k$--chain $y_0,y'_0,\dots,y'_k$ joining $\qi(x)$ and  $\qi(x')$ whose length is:
{\allowdisplaybreaks
\begin{align*}
d_{X'_{\iso(v_0)}}(y_0,y'_0)&+\sum_{0< i\leq k}\big(d(y'_{i-1},y_i)+d_{X'_{\iso(v_i)}}(y_i,y'_i)\big)\\
&\leq d_{X'_{\iso(v_0)}}(\qi_{v_0}(x_0),\qi_{v_0}(x'_0))+\cadd +\\*
&\qquad\sum_{0< i\leq k}\big(1+2\cadd+d_{X'_{\iso(v_i)}}(\qi_{v_i}(x_i),\qi_{v_i}(x'_i))\big)\\
&\leq \cmul d_{X_{v_0}}(x_0,x'_0)+2\cadd+\sum_{0< i\leq k}\big(1+3\cadd+\cmul d_{X_{v_i}}(x_i,x'_i)\big)\\
&=2\cadd+k(1+3\cadd)+\cmul\sum_{0\leq i\leq k} d_{X_{v_i}}(x_i,x'_i)\\
&\leq 2\cadd +\max\{1+3\cadd,\cmul\}\big(k+\sum_{0\leq i\leq k} d_{X_{v_i}}(x_i,x'_i)\big)\\
&\leq 2\cadd +\\*
&\qquad\max\{1+3\cadd,\cmul\}\left(d_{X_{v_0}}(x_0,x'_0)+\sum_{0< i\leq k}\left(d(x'_{i-1},x_i)+ d_{X_{v_i}}(x_i,x'_i)\right)\right)\\
&\leq 2\cadd+\max\{1+3\cadd,\cmul\}(d(x,x')+1)
\end{align*}}
Therefore, $d(\qi(x),\qi(x'))\leq \max\{1+3\cadd,\cmul\}d(x,x') + (2\cadd+\max\{1+3\cadd,\cmul\})$, so $\qi$ is coarsely Lipschitz. 
A similar computation shows $\bar{\qi}$ is coarsely Lipschitz, so $\qi$ is a quasi-isometry.
\end{proof}

\begin{figure}[h]
  \centering
  \[ \xymatrix{
  X_{\initial(e)}  \ar[rrr]^{\qi_{\initial(e)}} & & & X'_{\initial(\iso(e))}  \\
 X_e \ar[d]_{\attachmap_{e}} \ar@{^{(}->}[u]  \ar[rrr]^{\qi_e} & & & X'_{\iso(e)}  \ar[d]^{\attachmap'_{\iso(e)}} \ar@{^{(}->}[u]  \\
  X_{\terminal(e)}  \ar[rrr]_{ \qi_{\terminal(e)}} & & & X'_{\terminal(\iso(e))}
  } \]
  \caption{Commuting diagram for \fullref{corollary:treeofqis}.}\label{fig:com}
\end{figure}

\begin{corollary}\label{corollary:treeofqis}
Suppose $\iso\from\tree\to\tree'$ is an isomorphism of trees and $X$
and $X'$ are trees of spaces over $\tree$ and $\tree'$, respectively.
Suppose there are $\cmul\geq 1 $ and $\cadd\geq 0$ and
$(\cmul,\cadd)$--quasi-isometries $\qi_v\from X_v\to X'_{\iso(v)}$ for
each vertex and $\qi_e\from X_e\to X'_{\iso(e)}$ for each edge such
that the diagram in \fullref{fig:com} commutes up to uniformly bounded error. 
Then $(\qi_v)$ is a tree of quasi-isometries over $\chi$ compatible
with $X$ and $X'$, so $X$ is quasi-isometric to $X'$.
\end{corollary}

\subsubsection{Gromov boundaries and trees of homeomorphisms}\label{sec:boundaryoftreeofspaces}
In this section, let $X$ be a proper geodesic hyperbolic tree of spaces over a tree $\tree$, such that the vertex spaces are proper geodesic hyperbolic spaces, the edge spaces are uniformly quasi-convex in $X$, and the attaching maps are uniform quasi-isometries.

For example, if $\gog$ is a finite acylindrical graph of hyperbolic groups such that the edge injections are quasi-isometric embeddings then the algebraic tree of spaces has this structure \cite{BestvinaFeighnCombinationHyperbolic,Kap01}.
 If, moreover, the edge groups are two-ended then the edge injections are automatically quasi-isometric embeddings.

Quasi-convexity of edge spaces implies quasi-convexity of vertex spaces, so the Gromov boundary of each vertex space and edge space embeds into the boundary of $X$.

Consider the space $\bdry X:=\bdry\tree\sqcup \coprod_{v\in\verts\tree}\bdry X_v$ modulo identifying $\bdry X_e\subset\bdry X_{\initial(e)}$ with $\bdry X_{\bar{e}}\subset \bdry X_{\terminal(e)}$ via  $\bdry\psi_e$ for each edge $e$. 
Let $\bdry_{Stab}$ denote the image of $\coprod_{v\in\verts\tree}\bdry X_v$ in $\bdry X$.

\begin{lemma}
  The inclusion of $\bdry X$ into the Gromov boundary of $X$ is a surjection.
\end{lemma}
\begin{proof}
  Pick a basepoint $p\in X$ and let $\geo\from [0,\infty]\to \bar{X}$ be a geodesic ray based at $p$. 
Consider the projection $\pi(p)$ of $p$ to $\tree$.

If $\pi\circ\geo$ crosses some edge $e$ infinitely many times then there is an unbounded sequence of times $t_1,t_2,\dots$ such that $\geo(t_i)\in X_e$. Since the edge space is quasi-convex, $\geo$ stays bounded distance from $X_e$, so converges to a point in $\bdry X_e$.

A similar argument shows that if $\pi\circ \geo$ visits a vertex $v$ infinitely many times or is eventually constant at $v$ then $\geo$ converges to a point in $\bdry X_v$.

The remaining possibility is that $\pi\circ\geo$ limits to a point of $\bdry\tree$.
We claim there is a unique asymptotic class of geodesic rays in $X$ with $\pi\circ\geo\to\eta\in\bdry\tree$.
Suppose $\geo'$ is another such ray.
Let $v_0$ be the vertex such that $p\in X_{v_0}$. 
Let $e$ be an edge in $\tree$ on the geodesic from $v_0$ to $\eta$. 
Let $e'$ be another edge on the geodesic from $v_0$ to $\eta$ that is distance at least $\delta+Q$ from $e$, where $\delta$ is the hyperbolicity constant for $X$ and $Q$ is the quasi-convexity constant for edge spaces.
Let $x$ be a point of $\geo\cap X_e$ and $y$ a point of $\geo\cap X_{e'}$.
Let $x'$ be a point of $\geo'\cap X_e$ and $y'$ a point of $\geo'\cap X_{e'}$.
Consider a geodesic triangle whose sides are $\geo|_{[p,y]}$, $\geo'|_{[p,y']}$, and a geodesic $\geo''$ from $y$ to $y'$. 
By quasi-convexity, $d(x,\geo'')>\delta$ and $d(x',\geo'')>\delta$.
Therefore, hyperbolicity implies $d(x,\geo')\leq\delta$ and $d(x',\geo)\leq\delta$.
This implies $\geo$ and $\geo'$ are asymptotic.
\end{proof}

Now we define a topology on $\bdry X$ and show it is equivalent to the Gromov topology.

\begin{definition}[domains]\label{def:domain}
 The \textit{domain} $D(\eta)$ of a point $\eta \in \bd T$ is the singleton $\{\eta\}$. The \textit{domain} $D(\xi)$ of a point $\xi$ of $\bds X$ is the subtree of $T$ spanned by those vertices $v$ of $T$ such that $\bd X_v$ contains a point in the equivalence class $\xi$. 
\end{definition}

\begin{rmk}
 It can be proved that domains of points of $\bds X$ have a uniformly bounded number of edges, see \cite{MartinBoundaries}. 
\end{rmk}

We can now define neighborhoods of points of $\partial X$, starting with points of $\bd T$.

\begin{definition}[neighborhoods of points of $\bd T$]\label{def:nbhdtreepoint}
 Let $\eta$ be a point of $\bd T$, and $U$ be a neighborhood of $\eta$ in $T \cup \bd T$. We define the neighborhood $V_U(\eta)$ as the set of points of $\bd X$ whose domain is contained in $U$.  
\end{definition}

Before moving to neighborhoods of points of $\bds X$, we need a definition. 

\begin{definition}[$\xi$-family, cone]\label{def:cone}
 Let $\xi$ be a point of $\bds X$. 
For every vertex $v$ of $D(\xi)$, choose a neighborhood $U_v$ of $\xi$ in $X_v \cup \bd X_v$. Let $\cU$ be the collection of sets $U_v, v \in D(\xi)$, which we call a $\xi$-\textit{family}. We define the set $\mbox{Cone}_\cU(\xi)$, called a \textit{cone}, as the set of points $w \in (T \cup \bd T) \setminus D(\xi)$ such that if $e$ is the last edge of the geodesic from $w$ to $D(\xi)$ in $T$, we have $\bdry\attachmap_{e}(\bd X_e) \subset U_{\terminal(e)}$.
\end{definition}

\begin{definition}[neighborhoods of points of $\bds X$]\label{def:nbhdstabpoint}
 Let $\xi$ be a point of $\bds X$, and $\cU$ be a $\xi$-family. We define the neighborhood $V_{\cU}(\xi)$ as the set of points $\eta$ of $\bd X$ such that the following holds:
 \begin{itemize}
  \item $D(\eta)\setminus D(\xi)$ is contained in $\mbox{Cone}_\cU(\xi)$,
  \item for every vertex $v$ of $D(\xi) \cap D(\eta)$, we have $\eta \in U_v$.
 \end{itemize}
\end{definition}

\begin{thm}[{\cite[Corollary~9.19]{MartinBoundaries}}]\label{thm:boundaryhomeoGromovboundary}
 With the topology described above, the inclusion of $\partial X$ into the Gromov boundary of $X$ is a homeomorphism. \qed
\end{thm}

\begin{definition}\label{def:treehomeo}
  Let $X$ and $X'$ be proper geodesic hyperbolic trees of quasi-convex spaces over a trees $\tree$ and $\tree'$, respectively.
A \emph{tree of boundary homeomorphisms compatible with $X$ and $X'$ over an isomorphism $\chi\from\tree\to\tree'$} consists of homeomorphisms $\rho_v\from \bdry X_v\to\bdry X'_{\chi(v)}$ for every vertex $v\in\tree$ such that for $\xi\in \bdry X_v\cap \bdry X_w$ we have $\rho_v(\xi)=\rho_w(\xi)\in \bdry X'_{\chi(v)}\cap \bdry X'_{\chi(w)}$, and for $\xi\in\bdry X'_v\cap \bdry X'_w$ we have $\rho_v^{-1}(\xi)=\rho_w^{-1}(\xi)\in\bdry X_{\chi^{-1}(v)}\cap\bdry X_{\chi^{-1}(w)}$.
\end{definition}

\begin{proposition}\label{prop:treeofhomeosimplieshomeo}
   Let $X$ and $X'$ be proper geodesic hyperbolic trees of quasi-convex spaces over trees $\tree$ and $\tree'$, respectively, with a compatible tree of boundary homeomorphisms $(\rho_v)$ over $\chi\in\Isom(\tree,\tree')$.
Then there is a homeomorphism $\rho\from\bdry X\to\bdry Y$ defined by $\rho|_{\bdry X_v}:=\rho_v$ and $\rho|_{\bdry\tree}:=\bdry\chi$.
\end{proposition}
\begin{proof}
  We have a well defined map $\rho$ and its inverse $\rho^{-1}$ is defined by $\rho^{-1}|_{\bdry X'_v}:=\rho^{-1}_{\chi^{-1}(v)}$ and $\rho^{-1}|_{\bdry\tree}:=\bdry\chi^{-1}$.
It clear that $\rho$ and $\rho^{-1}$ are continuous with respect to
the topology given in \fullref{def:nbhdtreepoint} and
\fullref{def:nbhdstabpoint}, which is equivalent to the
standard topology by \fullref{thm:boundaryhomeoGromovboundary}.
\end{proof}

\begin{theorem}\label{thm:boundaryhomeoifftreeofboundaryhomeos}
  Let $G$ and $G'$ be one-ended hyperbolic groups with non-trivial two-ended JSJ decompositions.
Let $\modelspace$ and $\modelspace'$ be algebraic trees of spaces over the respective JSJ trees of cylinders $\tree:=\cyl(G)$ and $\tree':=\cyl(G')$.
Every homeomorphism $\rho\from\bdry \modelspace\to\bdry \modelspace'$ splits as a tree of compatible boundary homeomorphisms over the isomorphism $\rho_*\from\tree\to\tree'$.
Every tree of boundary homeomorphisms $(\rho_v)$ compatible with $\modelspace$ and $\modelspace'$ over an isomorphism $\chi\from\tree\to\tree'$ patches together to give a homeomorphism $\rho\from\bdry \modelspace\to\bdry \modelspace'$ with $\rho_*=\chi$ and $\rho|_{\bdry\modelspace_v}=\rho_v$.
\end{theorem}
\begin{proof}
Since the edge groups are two-ended they are virtually cyclic, hence quasi-convex.
Therefore, the boundary of each vertex space embeds into the boundary of its tree of spaces. 
 By \fullref{boundaryinvariance}, $\rho$ induces an isomorphism $\rho_*\from\tree\to\tree'$ and $\rho(\bdry\modelspace_v)=\bdry\modelspace'_{\rho_*(v)}$. 
Since these spaces are embedded, $\rho|_{\bdry\modelspace_v}\in\Homeo((\bdry\modelspace_v,\bdry\per_v),(\bdry\modelspace'_{\rho_*(v)},\bdry\per'_{\rho_*(v)}))$, so $\rho$ splits as a tree of boundary homeomorphisms over $\rho_*$.

The converse is \fullref{prop:treeofhomeosimplieshomeo}.
\end{proof}

\section{Distinguishing orbits: Basic examples}\label{sec:obstructions}
Given a group $G$, a common strategy in understanding groups
quasi-isometric to $G$ is to first understand self-quasi-isometries of
$G$. 
We explained in the previous section that $\QIgp(G)$ acts on
$\cyl(G)$, preserving the relative quasi-isometry types of vertices. 
A first step towards understanding $\QIgp(G)$ is to understand its
action on $\cyl(G)$.
In particular, we would like to determine the $\QIgp(G)$--orbits of vertices in
$\cyl(G)$. 
We know that there are finitely many $G$--orbits of vertices in $\cyl(G)$,
and that $\QIgp(G)$--orbits are unions of $G$--orbits, so the problem
reduces to distinguishing $G$--orbits that are not contained in a
common $\QIgp(G)$--orbit. 

In this section we show on some particular examples how one may distinguish $G$--orbits.
These examples motivate the technical machinery of the following
sections, in which we combine and iterate the considerations presented
here. 

Our examples are hyperbolic, so, as usual, we consider the cases of
quasi-isometry and boundary homeomorphism  simultaneously by letting $\starmap(G)$ mean either
$\QIgp(G)$ or $\Homeo(\bdry G)$, as appropriate.

In general our techniques reduce questions about graphs of groups to
questions about the vertex groups, so our results are strongest, and
the examples are most illuminating, when
the vertex groups are well understood. 
First we establish such a well understood vertex group, and then we
use it to construct graph of group examples.
\medskip

Consider the free group $F_2=\langle a,b\rangle$.
Let $X$ be the Cayley graph of $F_2$ with respect to the generating
set $\{a,b\}$.
Recall that if $\mathcal{H}$ is a finite collection of two-ended
subgroups of $F_2$ then there is a corresponding peripheral structure
$\per$ on $X$ consisting of distinct coarse equivalence classes of
conjugates of elements of $\mathcal{H}$. 
Since $X$ is a tree, the coarse equivalence class of a two-ended
subgroup contains a unique bi-infinite geodesic, so we can think of
$\per$ as an $F_2$--invariant collection of bi-infinite geodesics.

We claim that $\llbracket(X,\per_i)\rrbracket = \llbracket(X,\per_j)\rrbracket$ for all $i$ and $j$ where $\per_i$
is the peripheral structure induced by $\mathcal{H}_i$ defined as
follows:
\begin{multicols}{2}
\begin{itemize}
\item $\mathcal{H}_0=\{\langle a^2b^2\inv{a}\inv{b}\rangle\}$
\item $\mathcal{H}_1=\{\langle ab^2\rangle, \langle a^2\inv{b}\rangle\}$
\item $\mathcal{H}_2=\{\langle ab\rangle, \langle a^2\inv{b}^2\rangle\}$
\item $\mathcal{H}_3= \{\langle a\rangle, \langle ab\inv{a}\inv{b}^2\rangle\}$
\item $\mathcal{H}_4=\{\langle a\rangle, \langle b\rangle, \langle ab\inv{a}\inv{b}\rangle\}$
\item $\mathcal{H}_5=\{\langle a \rangle, \langle b \rangle, \langle ab
  \rangle, \langle a\inv{b} \rangle\}$
\end{itemize}
\end{multicols}

This follows from \cite[Theorem~6.3]{CasMac11}. 
The reason is that the Whitehead graph for each of these examples is the complete graph on four
vertices.
For each $i$, there are exactly 6 bi-infinite geodesics
representing elements of $\per_i$ passing through the vertex of $X$
corresponding to the identity element of $F_2$.
Given $i$ and $j$, any choice of bijection between these six elements
of $\per_i$ and the six corresponding elements of $\per_j$ extends uniquely to an
element of $\QIgp((X,\per_i),(X,\per_j))$.
Conversely, every element of $\QIgp((X,\per_i),(X,\per_j))$ is of this
form, up to pre- and post-composition with the $F_2$--action.

When $F_2$ is a vertex group of a graph of groups and the peripheral
structure coming from incident edge groups is one of the $\per_i$
above then the vertex is rigid.

For variety, we give two more rigid examples. 
\begin{itemize}
\item $F_3=\langle a,b,c\rangle$ relative to the peripheral structure
  induced by $\langle a^2b^2c^2acb\rangle$.
The Whitehead graph for this example is  the complete
bipartite graph with parts of three vertices each.
This peripheral structure is not quasi-isometrically equivalent to the
$\per_i$ given above.
\item $\pi_1(M)$, for $M$ a closed hyperbolic 3--manifold, relative to
  the peripheral structure induced by $\langle g\rangle$, where $g$
  is a non-trivial, indivisible element of $\pi_1(M)$.
\end{itemize}

\subsection{First example: wrong type, bad neighbors, stretch factors}\label{sec:firstbasicexample}
\medskip%room for labels
\begin{figure}[h]
  \centering
\labellist
\small
\pinlabel $\mathbb{Z}$ [br] at 3 101
\pinlabel $F_2$ [bl] at 79 197
\pinlabel $\mathbb{Z}$  [b] at 157 197
\pinlabel $\pi_1(M)$ [bl] at 233 197
\pinlabel $F_2$ [bl] at 79 149
\pinlabel $\mathbb{Z}$  [b] at 157 149
\pinlabel $F_2$ [bl] at 233 149
\pinlabel $F_2$ [bl] at 79 101
\pinlabel $\mathbb{Z}$  [b] at 157 101
\pinlabel $F_2$ [bl] at 233 101
\pinlabel $F_2$ [bl] at 79 52
\pinlabel $\mathbb{Z}$  [b] at 157 52
\pinlabel $F_2$ [bl] at 233 52
\pinlabel $F_2$ [bl] at 79 3
\pinlabel $\mathbb{Z}$  [b] at 157 3
\pinlabel $F_3$ [bl] at 233 3
\tiny
\pinlabel $c_1$ [tr] at 3 96
\pinlabel $r_1$ [tl] at 79 193
\pinlabel $c_2$ [t] at 157 193
\pinlabel $r_6$ [t] at 233 193
\pinlabel $r_2$ [tl] at 79 145
\pinlabel $c_3$  [t] at 157 145
\pinlabel $r_7$ [t] at 233 145
\pinlabel $r_3$ [t] at 79 97
\pinlabel $c_4$  [t] at 157 97
\pinlabel $r_8$ [t] at 233 97
\pinlabel $r_4$ [t] at 79 49
\pinlabel $c_5$  [t] at 157 49
\pinlabel $h$ [t] at 233 49
\pinlabel $r_5$ [t] at 79 1
\pinlabel $c_6$  [t] at 157 1
\pinlabel $r_9$ [t] at 233 1
\pinlabel $ab^2$ [br] at 70 182
\pinlabel $ab^2$ [br] at 70 141
\pinlabel $a^2\bar b^2$ [br] at 70 99
\pinlabel $ab^2$ [tr] at 70 56
\pinlabel $ab^2$ [tr] at 70 15
\pinlabel $a\bar{b}^2$ [bl] at 92 198
\pinlabel $a\bar{b}^2$ [bl] at 92 148
\pinlabel $ab$ [bl] at 92 100
\pinlabel $a\bar{b}^2$ [bl] at 92 51
\pinlabel $a\bar{b}^2$ [bl] at 92 3
\pinlabel $g$ [br] at 222 198
\pinlabel $a^2b^2\bar{a}\bar{b}$ [br] at 222 148
\pinlabel $a^2b^2\bar{a}\bar{b}$ [br] at 222 100
\pinlabel $ab\bar{a}\bar{b}$ [br] at 222 51
\pinlabel $a^2b^2c^2acb$ [br] at 222 3
\endlabellist
  \includegraphics{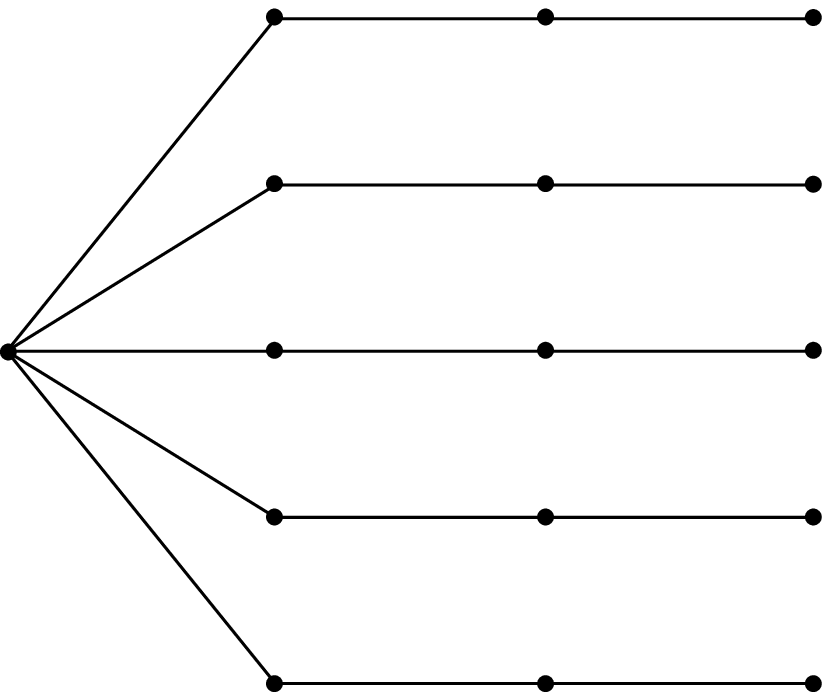}
  \caption{First example $\gog$}
  \label{fig:biggog}
\end{figure}

Consider the graph of groups $\gog$ shown in \fullref{fig:biggog}.
It is the graph of cylinders for the one-ended hyperbolic group
$G:=G(\gog)$, so $\tree:=\tree(\gog)=\cyl(G)$.
The label below a vertex is its name, and the label above a vertex is
the associated vertex group.
All of the edge groups are infinite cyclic.
Suppose that we have chosen a generator for each infinite cyclic
vertex and edge group. 
Each end of an edge is labelled to indicate the image of this
generator in the incident vertex group. 
For cyclic vertices we give a non-zero integer indicating that the
image of the edge generator is that power of the vertex generator. 
This label is omitted when it is equal to 1, which is the case for all
edges in this particular graph of groups.

For non-cyclic vertices we
indicate the element that is the image of the chosen generator of the edge group.
For economy of notation we reuse $a$ and $b$ as generators of all the
$F_2$ vertex groups.
These are distinct subgroups of $G$, so more properly we would say,
for instance, that $G_{r_1}=\langle a_1,b_1\rangle$ and the two
incident edge injections send the generators of their respective edge
groups to $a_1b_1^2$ and $a_1\inv{b}_1^2$.

In this section we describe the $\starmap(G)$--orbits of vertices in $\tree$.

\subsubsection{Wrong type}
The vertex $h$ is the only hanging vertex, so $G\lift{h}$ is
a $\starmap(G)$--orbit.

The vertices labelled $c_i$ are cylindrical, while the vertices
labelled $r_i$ are rigid, so there are no $i$ and $j$ such that
$\lift{r_i}$ and $\lift{c_j}$ are in the same $\starmap(G)$--orbit.

The vertex group of $r_6$ is not quasi-isometric or boundary
homeomorphic to any of the other vertex groups, so $G\lift{r_6}$ is
a $\starmap(G)$--orbit.

The vertex group of $r_9$ relative to the peripheral structure
induced by the incident edge is not the same relative quasi-isometry
or boundary homeomorphism type of any other vertex, so
$G\lift{r_9}$ is a $\starmap(G)$--orbit.

\subsubsection{Bad neighbors}
The cylindrical vertices have finite valence in $\tree$.
The valence of vertices in $G\lift{c_1}$ is 5, while for all other
cylindrical vertices it is 2.
Thus, $G\lift{c_1}$ is a $\starmap(G)$--orbit.

Vertices in $G\lift{c_2}$ are the only ones adjacent to vertices in
the $\starmap(G)$--orbit $G\lift{r_6}$, so $G\lift{c_2}$ is a
$\starmap(G)$--orbit.

Vertices in $G\lift{c_2}$ are adjacent only to vertices in
$G\lift{r_1}$ and $G\lift{r_6}$, but we already know that
$G\lift{r_6}$ is a $\starmap(G)$--orbit, so $G\lift{r_1}$ is a
$\starmap(G)$--orbit as well. 

By the same sort of argument, we quickly conclude that
$G\lift{c_5}$, $G\lift{r_4}$, $G\lift{c_6}$, and $G\lift{r_5}$ are distinct
$\starmap(G)$--orbits.

Such arguments based on distinguishing vertices by their neighbors are
the contents of \fullref{sec:structureinvariants}.

\subsubsection{Stretch factors}
It remains to differentiate the second and third branches of $\gog$.
This is more subtle, as the only difference is at vertices  $r_2$ and
$r_3$, which do have the same relative
quasi-isometry type.

Notices that the generator of $G_{c_3}$ is identified with an element
of word length 6 in $G_{r_7}$, and length 3 in $G_{r_2}$,
whereas the generator of $G_{c_4}$ is identified with elements of word
lengths 6 and 2. 
It is not at all obvious that these lengths are preserved by
quasi-isometries. 
Nevertheless, we will show that in this particular example $\lift{c_3}$ and
$\lift{c_4}$ can be distinguished up to quasi-isometry by the fact that the
ratios $\frac{6}{3}$ and $\frac{6}{2}$ are not equal. 
These ratios correspond to an invariant called the `stretch factor',
which is developed in \fullref{sec:stretchfactors}.

By neighbor arguments as above we can then conclude that
$\QIgp(G)$--orbits of $\tree$ are the same as $G$--orbits.

The stretch factor is only an invariant for quasi-isometries,
not boundary homeomorphism. 
In this example we cannot, in fact,
distinguish the second and third branch of $\gog$ up to boundary
homeomorphism, which is to say, $G\lift{r_2}\cup G\lift{r_3}$ is a
single $\Homeo(\bdry G)$--orbit, as are $G\lift{c_3}\cup
G\lift{c_4}$ and $G\lift{r_7}\cup G\lift{r_8}$.

\subsection{Second example: unbalanced cylinders}
In general, if $\cL$ is an element of a peripheral structure $\per$
on $\modelspace$, there is no reason to believe that there is an
element of $\starmap((\modelspace,\per))$ that preserves $\cL$ and reverses
its orientation. 

We briefly give an idea how one may construct such an example. 
Consider the genus 7 surface $\Sigma$ in \fullref{fig:orientation},
with a hyperbolic metric.

\begin{figure}[h]
  \centering
  \includegraphics[scale=.5]{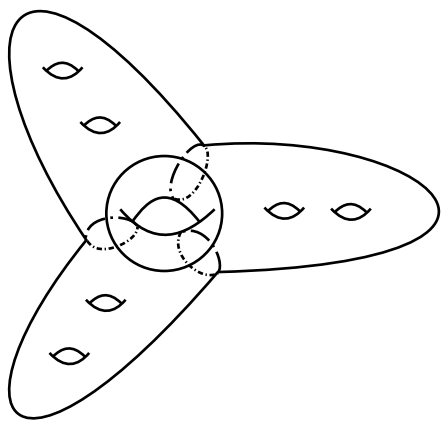}
  \caption{The surface $\Sigma$}
  \label{fig:orientation}
\end{figure}

Let $f\in\pi_1(\Sigma)$ be the element represented by the curve
running around the central hole. 
Let $\cL$ be a component of the preimage of this curve in $\widetilde{\Sigma}=\mathbb{H}^2$.
The figure shows three curves the separate $\Sigma$ into subsurfaces
$\Sigma_1$, $\Sigma_2$, and $\Sigma_3$, each of which is a genus two
surface with two boundary components. 
Suppose each $\Sigma_i$ contains a collection of curves that is
complicated enough so that all the complementary components are discs,
and different enough so that there does not exist a  crossing-preserving bijection between the components
of the preimages of the curves of $\Sigma_i$ contained in one
component of the preimage of $\Sigma_i$ in $\widetilde{\Sigma}$
 and the components
of the preimages of the curves of $\Sigma_j$ contained in one
component of the preimage of $\Sigma_j$.
Let $\per$ be the peripheral structure on $\widetilde{\Sigma}$ induced
by all of the above curves in $\Sigma$, which is to say, each element of
$\per$ is the coarse equivalence class of a component of the preimage
in $\widetilde{\Sigma}$ of one of the curves in $\Sigma$.

Then there is no element of $\map(\pi_1(\Sigma),\per)$ that reverses
$\cL$.
The reason is that $\cL$ passes through components of the preimages of
$\Sigma_1$, $\Sigma_2$, and $\Sigma_3$ in order. 
If this order were reversible by an element of
$\starmap((\pi_1(\Sigma),\per))$ we would contradict the restriction that
the pattern of curves in $\Sigma_1$, $\Sigma_2$, and $\Sigma_3$ were
chosen to be `different'. 

Now consider the graph of groups $\gog$ of
\fullref{fig:disorientation}, where $\Sigma$ and $f$ are as above, and
$g$ is a non-trivial, indecomposable element of a closed hyperbolic
3--manifold $M$.
Then $G:=G(\gog)$ is a one-ended hyperbolic group with  $\tree:=\tree(\gog)=\cyl(G)$.

\begin{figure}[h]
  \centering
\labellist
\small
\pinlabel $\pi_1(\Sigma)$ [l] at 52 107
\pinlabel $\pi_1(\Sigma)$ [br] at 122 80
\pinlabel $\pi_1(\Sigma)$ [l] at 58 43
\pinlabel $\mathbb{Z}$ [bl] at 76 77
\pinlabel $\mathbb{Z}$ [l] at 54 130
\pinlabel $\pi_1(M)$ [l] at 55 154
\tiny
\pinlabel $r_1$ [tr] at 50 105
\pinlabel $r_2$ [t] at 122 76
\pinlabel $r_3$ [br] at 55 43
\pinlabel $c_1$ [tl] at 76 73
\pinlabel $-1$ [br] at 70 65
\pinlabel $f$ [tl] at 63 55
\pinlabel $f$ [bl] at 58 94
\pinlabel $f$ [t] at 108 77
\pinlabel $g$ [l] at 55 145
\pinlabel $r_4$ [r] at 52 154
\pinlabel $c_2$ [rt] at 52 129
\endlabellist
  \includegraphics{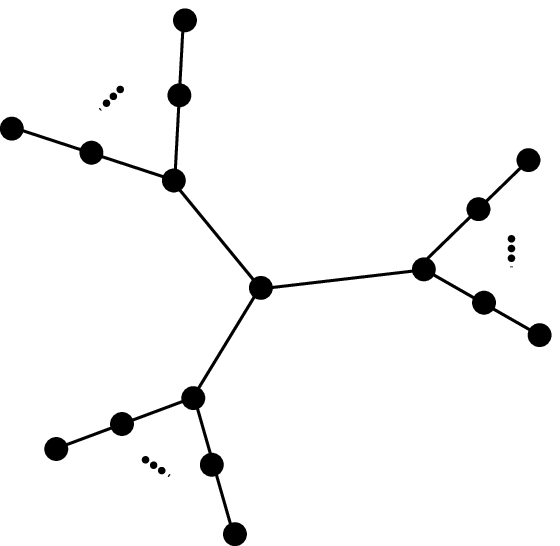}
  \caption{Second example $\gog$}
  \label{fig:disorientation}
\end{figure}

The graph of groups $\gog$ has a central cylindrical vertex $c_1$ that
attaches to three rigid vertices $r_1$, $r_2$, and $r_3$, each of
which has local group $\pi_1(\Sigma)$, with each attachment along a
copy of $f$.
For each $r_i$ there are some number of other incident edges, each
corresponding to one of the curves in $\Sigma$ described above, so
that the peripheral structure on $G_{r_i}$ induced by edge inclusions
is the peripheral structure $\per$ described above for
$\widetilde{\Sigma}$.
Each of these edges we attach to another cylindrical vertex, and then
to a rigid vertex carrying a copy of $\pi_1(M)$.

We consider a kind of `$f$--parity' by counting the number of vertices
adjacent to $\lift{c_1}$ for which the generator of $G_{\lift{c_1}}$ is identified
with $f$ minus the number of vertices adjacent to $\lift{c_1}$ for which the generator of $G_{\lift{c_1}}$ is identified
with $\inv{f}$.
Since $\lift{c_1}$ has non-zero $f$--parity, we call it an \emph{unbalanced
  cylinder}.
We claim that there is not an element of $\starmap(G)$ that reverses
the orientation of an unbalanced cylinder.
In this example, we then conclude that $G\lift{r_3}$ is a
$\starmap(G)$--orbit and $G\lift{r_1}\cup G\lift{r_2}$ is a different $\starmap(G)$--orbit.

The proof is as follows.
Since $G\lift{c_1}$ is the only $G$--orbit of cylindrical vertices in
$\tree$ of
valence 3, if there were an element of $\starmap(G)$ taking
$\lift{r_3}$ to, say, $\lift{r_1}$, then it would fix $\lift{c_1}$.
There is clearly an element of $\starmap((G_{r_3},\per_{r_3}),
(G_{r_1},\per_{r_1}))$ taking $f$ to $f$, but, since $f$ is not
reversible, if such a map extended to an element of $\starmap(G)$
it would necessarily reverse the
orientation of $G_{\lift{c_1}}$.
Since $f$ is not reversible, this would mean that every vertex in
which $f$ is identified with the generator of $G_{\lift{c_1}}$ must be sent
to a vertex in which $\inv{f}$ is identified with the generator of
$G_{\lift{c_1}}$.
This means that both $\lift{r_1}$ and $\lift{r_2}$ are sent to
$\lift{r_3}$, contradicting the fact that $\starmap(G)$ acts by
isomorphisms on $\tree$.

Dealing with issues of orientation and unbalanced cylinders is the
focus of \fullref{sec:cylinderrefinement}.

\section{Decorated trees and structure invariants}\label{sec:structureinvariants}
Let $G$ be a group and let $\tree$ be a simplicial tree of countable valence upon which $G$ acts cocompactly and
without inverting an edge. 
In this section, we explain how to associate to $\tree$ an invariant, the \emph{structure invariant}, that completely characterizes it up to decoration-preserving isomorphism.
The construction of this invariant generalizes the `degree refinement'
invariant of graph theory, which is an invariant that determines when
two finite graphs have isomorphic universal covers.

\begin{definition}\label{def:decoration}
A \emph{decoration} is a   $G$--invariant map $\decmap\from \tree \to
\ornaments$ that assigns to each vertex of $\tree$ an
\emph{ornament} $\ornament\in\ornaments$.
\end{definition}

For simplicity, in this section we decorate only vertices. In the next
section we will also decorate edges, with the condition that
$\decmap(e)=\decmap(\bar{e})$ for every edge.
Formally, this can be accomplished by subdividing each edge of $\tree$
and decorating the new vertices.

Corresponding to a decoration there is a partition of $\tree$ as
$\coprod_{\ornament\in\ornaments}\decmap^{-1}(\ornament)$.
We say that a decoration $\decmap'\from\tree\to\ornaments'$
\emph{is a refinement of} $\decmap$ if the $\decmap'$--partition is finer than
the $\decmap$--partition. 
Equivalently, a decoration $\decmap'\from\tree\to\ornaments'$ is a refinement of  $\decmap\from\tree\to\ornaments$ if there exists a surjective map $\pi\from \mbox{Im } \decmap' \to \mbox{Im }\decmap$ such that $\pi \circ \decmap' = \decmap$.
We say $\decmap'$ is a \emph{strict refinement} if the
$\decmap'$--partition is strictly finer than
the $\decmap$--partition.
A refinement that is not strict is a \emph{trivial refinement}.

In this section, we will define a \emph{structure invariant} for a decorated tree, such that we can  reconstruct the tree, up to
decoration-preserving tree automorphism,
from the data of the structure invariant. In order to do so, we want to identify orbits in $\tree$ under the action of the group 
$\Aut(\tree,\decmap):=\{\iso\in\Aut(\tree)\mid
\decmap\circ\iso=\decmap\}$, and then to say how they
fit together. In a nutshell, the idea is as follows.

Suppose $v$ and $w$ are two vertices.
If there is a $\iso\in\Aut(\tree,\decmap)$ with $\iso(v)=w$, then, of course, $\decmap(v)=\decmap(w)$. 
Additionally, $\iso$ gives a decoration-preserving bijection from the
neighbors of $v$ to the neighbors of $w$.
Thus, for each ornament $\ornament$, the number of neighbors of $v$ bearing $\ornament$ must be equal to the number of neighbors of $w$ bearing $\ornament$. 

Conversely, if $\decmap(v)=\decmap(w)$, but for some ornament $\ornament$
there are differing numbers of neighbors of $v$ bearing $\ornament$ and
neighbors of $w$ bearing $\ornament$, then there is no decoration-preserving
automorphism taking $v$ to $w$, so we ought refine the decoration to distinguish $v$ from
$w$.
We then repeat this refinement process until vertices with the same
ornament can no longer be distinguished by the ornaments of their neighbors.
This happens after finitely many steps because $\lquotient{G}{\tree}$ is
compact, see \fullref{prop:stabilization}.
\fullref{sec:refinement} formalizes this process, which we call
\emph{neighbor refinement}.

\subsection{Neighbor refinement}\label{sec:refinement}
Let $\extN:=\mathbb{N}\cup\{0,\infty\}$.
Call $\ornaments_0:=\ornaments$ and $\decmap_0:=\decmap$ the `initial set of
ornaments' and the `initial decoration', respectively.
Beginning with $i=0$,
for each $v\in\verts\tree$ define: 
\[f_{v,i}\from \ornaments_{i}\to \extN: \ornament\mapsto
\#\{w\in\decmap_{i}^{-1}(\ornament)\mid w \text{ is adjacent to }v\}\]

Define
$\ornaments_{i+1}:=\ornaments_0\times \extN^{\ornaments_i}$, and
define $\decmap_{i+1}(v):=(\decmap_0(v), f_{v,i})$.

\begin{lemma}\label{lemma:refinement}
For all $i$, the map
  $\decmap_{i+1}\from\tree\to\ornaments_{i+1}$ is a decoration that refines $\decmap_i\from\tree\to\ornaments_i$.
\end{lemma}
\begin{proof}
  Let $v$ be a vertex, and let $g\in G$. 
Suppose $\decmap_i$ is
  $G$--invariant.
Then,
$\decmap_{i+1}(gv)=(\decmap_i(gv),f_{gv,i})=(\decmap_i(v),f_{v,i})=\decmap_{i+1}(v)$.
Since $\decmap_0$ is $G$--invariant, all the $\decmap_i$ are
decorations by induction.

For each $i$, let $N(v,i,\ornament)$ denote the number of neighbors of $v$ in
$\decmap_i^{-1}(\ornament)$.

Clearly, $\decmap_1$ refines $\decmap_0$, since $\decmap_0$ is the
composition of $\decmap_1$ with projection to the first coordinate of
the image. 
Suppose that $\decmap_i$ refines $\decmap_{i-1}$.
Then for every $\ornament'\in\ornaments_{i-1}$, we have $N(v,i-1,\ornament')=\sum_{\ornament\in\decmap_i\circ\decmap_{i-1}^{-1}(\ornament')}N(v,i,\ornament)$.

If $\decmap_{i+1}(v)=\decmap_{i+1}(w)$ then
$\decmap_0(v)=\decmap_0(w)$ and $N(v,i,\ornament)=N(w,i,\ornament)$ for each $\ornament\in\ornaments_i$.
Thus, for all $\ornament'\in\ornaments_{i-1}$, \[N(v,i-1,\ornament')=\sum_{\ornament\in\decmap_i\circ\decmap_{i-1}^{-1}(\ornament')}N(v,i,\ornament)=\sum_{\ornament\in\decmap_i\circ\decmap_{i-1}^{-1}(\ornament')}N(w,i,\ornament)=N(w,i-1,\ornament'),\]
so $\decmap_i(v)=\decmap_i(w)$.
Hence $\decmap_{i+1}$ refines $\decmap_i$. 
The lemma follows by induction.
\end{proof}

\begin{proposition}\label{prop:stabilization}
  There exists an $s\geq 0$ such that $\decmap_{i+1}$ is a strict
  refinement of $\decmap_i$ for all $i+1\leq s$ and $\decmap_{i+1}$ is a trivial
  refinement of $\decmap_i$ for all $i\geq s$.
\end{proposition}
\begin{proof}
Since $\decmap_{i+1}$ refines $\decmap_i$ there is a function
$\decmap_i\circ
\decmap_{i+1}^{-1}\from\image\decmap_{i+1}\subset\ornaments_{i+1}\to\ornaments_i$.
The refinement is trivial precisely when this function is injective. 

Suppose there exists an $s$ such that $\decmap_{s+1}$ is a trivial
refinement of $\decmap_s$. 
For every
$\ornament\in\image\decmap_{s+1}\subset\ornaments_{s+1}$ we have $f_{v,s+1}(\ornament)=f_{v,s}\circ \decmap_s\circ
\decmap_{s+1}^{-1}(\ornament)$, and $f_{v,s+1}(\ornament)=0$
otherwise.
Therefore, $\decmap_{s+1}(v)=\decmap_{s+1}(w)$ implies
$f_{v,s}=f_{w,s}$, which implies $f_{v,s+1}=f_{w,s+1}$, which implies $\decmap_{s+2}(v)=\decmap_{s+2}(w)$.
Thus, once one refinement in the sequence is trivial so are all further refinements. 

To see that eventually some refinement is trivial, note that
$G$--invariance implies that for all $i$ we have a partition of
$\tree$ by
$\tree=\coprod_{\ornament\in\ornaments_i}\decmap_i^{-1}(\ornament)$ in
which each part is a union of $G$--orbits. 
A refinement is strict if and only if the new partition has strictly
more parts than the previous one. However, the number of parts is
bounded above by the number of $G$--orbits, which is finite.
\end{proof}

\begin{definition}
  The \emph{neighbor refinement} of $\decmap$ is the decoration
  $\decmap_s\from\tree\to\ornaments_s$ at which the neighbor refinement process stabilizes.
\end{definition}

\begin{proposition}\label{prop:neighborrefinementpartitionisautdecpartition}
Let $\decmap'\from\tree\to\ornaments'$ be the neighbor refinement of $\decmap\from\tree\to\ornaments$.
  The $\decmap'$--partition of $\tree$ is equal to the partition
  into $\Aut(\tree,\decmap)$--orbits and to the partition into $\Aut(\tree,\decmap')$--orbits.
\end{proposition}
\begin{proof}
  The partition into $\Aut(\tree,\decmap)$--orbits is finer than the
  $\decmap'$--partition because each refinement step is
  $\Aut(\tree,\decmap)$--equivariant. 
The $\Aut(\tree,\decmap')$--orbit partition is finer than the
$\Aut(\tree,\decmap)$--orbit partition since $\decmap'$ is a
refinement of $\decmap$.
We show the $\decmap'$--partition is finer than the $\Aut(\tree,\decmap')$--orbit partition
by supposing $\decmap'(v)=\decmap'(w)$ and
producing $\iso\in\Aut(\tree,\decmap')$ with $\iso(v)=w$.

The automorphism $\iso$ is constructed inductively. Start by defining $\iso(v):=w$.
Since $\decmap'$ is stable under neighbor refinement, for every
$\ornament\in\ornaments'$ we have
$\#\link(v)\cap(\decmap')^{-1}(\ornament)=\#\link(w)\cap(\decmap')^{-1}(\ornament)$.
Extend $\iso$ by choosing any bijection between these sets. 
This extends $\iso$ to the 1--neighborhood of $v$.

Now suppose $\iso$ is defined on a subtree $\tree'$ of $\tree$ such
that for every $v'\in\verts\tree'$, either $v'$ is a leaf or $\tree'$ contains every edge
incident to $v'$. 
Let $v'$ be a leaf, and let $u$ be the vertex of $\tree'$ adjacent to $v'$.
For $\decmap'(u)$ extend $\iso$ via a bijection between
$\big(\link(v)\setminus\{u\}\big)\cap(\decmap')^{-1}(\decmap'(u))$ and
$\big(\link(\iso(v'))\setminus\{\iso(u)\}\big)\cap(\decmap')^{-1}(\decmap'(u))$.
For $\ornament\neq \decmap'(u)$ extend $\iso$ just as in the base case.
\end{proof}

\subsection{Structure invariants}\label{sec:structure}
Let $\decmap_s\from\tree\to\ornaments_s$ be a neighbor refinement of $\decmap$ as in \fullref{prop:stabilization}.
As defined, $\decmap_s$ actually encodes the history of the refinement
process, not just the structure of $\tree$. 
We define the structure invariant by forgetting this extraneous
history. 
Let $\pi_0\from\ornaments_s\to\ornaments$ be projection to the first coordinate.
Choose an ordering of $\image\decmap$ and let $\ornaments[j]$ denote the $j$--th
ornament in the image.
Similarly, for each $1\leq j\leq \#\image\decmap$, choose an ordering of
$\pi_0^{-1}(\ornaments[j])\cap\image\decmap_s$.
Order $\image\decmap_s$ lexicographically, and let $\ornaments_s[i]$ denote the
$i$--th ornament.

\begin{definition}
The \emph{structure invariant} $\struc(T,\decmap,\ornaments)$ is the
$\#\image\decmap_s\times\#\image\decmap_s$ matrix whose $j,k$--entry is the tuple consisting of the
number of vertices in $\decmap_s^{-1}(\ornaments_s[j])$ adjacent to each vertex of
$\decmap_s^{-1}(\ornaments_s[k])$, the $\ornaments$--ornament
$\pi_0(\ornaments_s[j])$ and the $\ornaments$--ornament
$\pi_0(\ornaments_s[k])$. 
The last two coordinates are called respectively the \emph{row} and \emph{column ornament} of the entry.
\end{definition}
$\struc(T,\decmap,\ornaments)$ can be seen as a block matrix, with  blocks consisting of
entries with the same row ornaments and the same column ornaments.
The structure invariant is well defined up to permuting the $\ornaments$--blocks and permuting
rows and columns within $\ornaments$--blocks, ie, up to the choice of
orderings of $\ornaments$ and the $\pi_0^{-1}(\ornaments[j])$.

\begin{proposition}\label{structureinvariant}
Let $\decmap\from\tree\to\ornaments$ be a $G$--invariant decoration of a
cocompact $G$--tree.
Let $\decmap'\from\tree'\to\ornaments$ be a $G'$--invariant decoration of a
cocompact $G'$--tree.
There exists  a decoration-preserving isomorphism $\phi\from T\to T'$ if and only if
$\struc(T,\decmap,\ornaments)=\struc(\tree'\!,\decmap'\!,\ornaments)$, up to permuting
rows and columns within $\ornaments$--blocks.
\end{proposition}

In particular, with the above notations, $T$ and $T'$ must have the
same sets of ornaments for a decoration-preserving isomorphism $\phi$
to exist.

\begin{proof}
  It is clear that isomorphic  decorated trees have the
  same structure invariants, up to choosing the orderings of the ornaments.
For the converse, assume that we have reordered within $\ornaments$--blocks
so that
$\struc(T,\decmap,\ornaments)=\struc(\tree'\!,\decmap'\!,\ornaments)=\struc$.
Construct a decoration-preserving tree isomorphism exactly as in the
proof of \fullref{prop:neighborrefinementpartitionisautdecpartition}.
\end{proof}

\begin{remark}
  When $T$ is the universal cover of a finite graph $\Gamma$ and the initial
  set of ornaments is trivial then the structure invariant we have
  defined is just the well known \emph{degree refinement} of
  $\Gamma$. The lemma says that two graphs have the same degree
  refinement if and only if they have isomorphic universal covers.
A theorem of Leighton \cite{Lei82} says that such graphs
in fact have a common finite cover.
There are also decorated versions of Leighton's Theorem, eg
\cite{Neu10}. 
\end{remark}

\begin{observation}
We get a quasi-isometry invariant of a group $G$ by taking the structure 
invariant of a cocompact $\QIgp(G)$-tree with a $\QIgp(G)$--invariant decoration.  
\end{observation}

This is a simple observation, but it does not seem to have
appeared in the literature in this generality.

Behrstock and Neumann \cite{BehNeu08, BehNeu12} have used special
cases of this type of invariant, in a different guise, to classify
fundamental groups of some families of compact irreducible
3--manifolds of zero Euler characteristic.
In both papers the tree is the Bass-Serre tree for the geometric
decomposition of such a 3--manifold along tori and Klein bottles,
which is the
higher dimensional
antecedent of the JSJ decompositions considered in this paper.

When the geometric decomposition has only Seifert fibered pieces the vertices
are decorated by the quasi-isometry type of the universal cover of the
corresponding Seifert fibered manifold. 
There are only two possible quasi-isometry types, according to whether
or not the Seifert fibered piece has boundary.
Every vertex in the Bass-Serre tree has infinite valence, so
each entry of the structure invariant is either $0$ or $\infty$. 

Behrstock and Neumann \cite{BehNeu08} state their result in terms of
`bi-similarity'\footnote{Their meaning of `similarity' is different
  than in
  this paper.}classes of bi-colored graphs.
They show that each bi-similarity class is represented by a unique
minimal graph, and that two such 3--manifolds are quasi-isometric if
and only if the bi-colored Bass-Serre tree of their geometric
decompositions have the same representative minimal graph.
Their minimal bi-colored graphs carry exactly the same
information as the structure invariant of the decorated Bass-Serre
tree.
One can construct their graph by taking the vertex set to be the stable
decoration set $\ornaments_s$ and connecting vertex $\ornaments_s[j]$ to vertex
$\ornaments_s[k]$ by an edge if and only if the $j,k$--entry of $\struc$ is
$\infty$.
The vertices of the graph are `bi-colored' by the projection
$\pi_0\from\ornaments_s\to\ornaments$.
Conversely, $\struc$ can be recovered by replacing each edge in the
graph by infinitely many edges, lifting the bi-coloring to the
universal covering tree, and calculating the structure invariant.

The second paper \cite{BehNeu12} extends their results to cases where the
decomposition involves some hyperbolic pieces.
The decorations there are more complex.

\subsection{Structure invariants for the JSJ tree of cylinders}
Combining \fullref{structureinvariant} with \fullref{qiinvariance} and
\fullref{qirestricted} proves:

\begin{theorem}\label{main}
If $G$ is a finitely presented one-ended group not commensurable to a
surface group, then the structure invariant for the JSJ tree of
cylinders is a quasi-isometry invariant of $G$, with respect to any of
the following initial decorations:
\begin{enumerate}
\item Vertex type: rigid, hanging, or cylinder.\label{initialdecorationbytype}
\item Vertex type and, if $v$ is rigid, $\llbracket G_v\rrbracket$.\label{initialdecorationbyqitype}
\item Vertex type and, if $v$ is rigid, $\llbracket (G_v,\per_v)\rrbracket$.\label{initialdecorationbyrelativeqitype}
\end{enumerate}
\end{theorem}

\begin{theorem}\label{thm:allhanging}
  If $G$ is hyperbolic and the JSJ decomposition of $G$ has no rigid vertices then the
  invariant of \fullref{main} is a complete quasi-isometry invariant.
\end{theorem}
Later we will prove a more general result, \fullref{thm:qi}, that 
includes \fullref{thm:allhanging} as a special case.
A brief sketch of a direct proof of
\fullref{thm:allhanging} goes like this:
Given two hyperbolic groups as in \fullref{thm:allhanging}, the
structure invariants of \fullref{main} are equivalent if and only if the groups
have isomorphic decorated JSJ trees of cylinders. 
Since all non-cylindrical vertices are hanging, it
follows, using techniques of of Behrstock and Neumann
   \cite{BehNeu08}, that the groups are quasi-isometric if and only
  if they have isomorphic decorated JSJ trees of cylinders.
Details of the last claim can be found in Dani and Thomas
\cite[Section~4]{DanTho14}.
 The torsion-free case was previously written up in the thesis of Malone
 \cite{Mal10}.

\subsection{Examples}
\subsubsection{An example with symmetry}
Consider the graph of groups $\gog$ in \fullref{fig:twofold}, for
which $\tree(\gog)=\cyl(G(\gog))$.
Let $G:=G(\gog)$ and $\tree:=\tree(\gog)$.
\medskip%space for labels
\begin{figure}[h]
  \centering
  \labellist
\tiny
\pinlabel $r$ [t] at 312 0
\pinlabel $c$ [t] at 235 0
\pinlabel $h$ [t] at 157 0
\pinlabel $c'$ [t] at 80 0
\pinlabel $r'$ [t] at 3 0
\pinlabel $a^2b^2\bar{a}\bar{b}$ [br] at 296 3
\pinlabel $zy$ [bl] at 174 3
\pinlabel $xy\bar{x}\bar{z}$ [br] at 141 3
\pinlabel $a^2b^2\bar{a}\bar{b}$ [bl] at 15 3
\small
\pinlabel $\langle a,b\rangle$ [b] at 312 4
\pinlabel $\mathbb{Z}$ [b] at 235 4
\pinlabel $\langle x,y,z\rangle$ [b] at 157 4
\pinlabel $\mathbb{Z}$ [b] at 80 4
\pinlabel $\langle a,b\rangle$ [b] at 3 4 
\endlabellist
\includegraphics{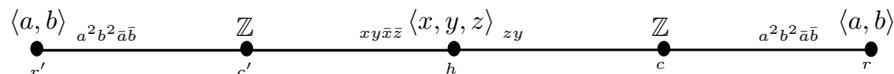}
  \caption{Symmetric example}
  \label{fig:twofold}
\end{figure}

Take the initial decoration be vertex type, so
$\decmap_0(r),\,\decmap_0(r'):=\text{`rigid'}$;
$\decmap_0(c),\,\decmap_0(c'):=\text{`cylindrical'}$; and $\decmap_0(h):=\text{`hanging'}$.

Let us compute the first neighbor refinement:
The image of an edge injection into a rigid vertex group is an
infinite index subgroup, so rigid vertices of $\tree$ are infinite
valence, and adjacent only to cylindrical vertices. Since $\decmap_0$
does not distinguish between cylindrical vertices, we have:
\[ f_{\lift{r},0},\,f_{\lift{r'},0}:=
\begin{cases}
\text{`cylindrical'}&\mapsto \infty\\
\text{`rigid'}&\mapsto 0\\
\text{`hanging'}&\mapsto 0
\end{cases}\]

In this example each edge group surjects onto its adjacent cylindrical
vertex group, so cylindrical vertices of $\tree$ have valence 2. 
One neighbor is rigid and one is hanging, and $\decmap_0$ does not
distinguish between the rigid vertices, so we have:
\[ f_{\lift{c},0},\,f_{\lift{c'},0}:=
\begin{cases}
\text{`cylindrical'}&\mapsto 0\\
\text{`rigid'}&\mapsto 1\\
\text{`hanging'}&\mapsto 1
\end{cases}\]

Finally, the image of an edge injection into a hanging vertex group is an
infinite index subgroup, so hanging vertices of $\tree$ are infinite
valence, and adjacent only to cylindrical vertices. Since $\decmap_0$
does not distinguish between cylindrical vertices, we have:
 \[f_{\lift{h},0}:=
\begin{cases}
\text{`cylindrical'}&\mapsto \infty\\
\text{`rigid'}&\mapsto 0\\
\text{`hanging'}&\mapsto 0
\end{cases}\]

The first refinement is therefore:
\begin{align*}
  \decmap_1(\lift{r}),\,\decmap_1(\lift{r'})&:=(\text{`rigid'},f_{\lift{r},0})\\ 
\decmap_1(\lift{c}),\,\decmap_1(\lift{c'})&:=(\text{`cylindrical'},f_{\lift{c},0})\\
\decmap_1(\lift{h})&:=(\text{`hanging'},f_{\lift{h},0})
\end{align*}
This is a trivial refinement, so $\decmap_0$ is stable, and the
structure invariant is given in \fullref{tab:sym}.
\begin{table}[h]
  \centering
  \[
  \begin{array}{l|c|c|c|}
    &\text{`cylindrical'}&\text{`rigid'}&\text{`hanging'}\\
\text{`cylindrical'}&0&1&1\\
\cline{1-4}
\text{`rigid'}&\infty&0&0\\
\cline{1-4}
\text{`hanging'}&\infty&0&0
  \end{array}
\]
  \caption{Structure invariant for symmetric example}
  \label{tab:sym}
\end{table}

Notice that we get the same structure invariant for the group
$G':=G(\gog')$ defined by the graph of groups in
\fullref{fig:twofoldcovered}.
This is expected, as the groups are quasi-isometric. In fact, $G$ is isomorphic to an index 2 subgroup of $G'$.
\medskip%space for labels
\begin{figure}[h]
  \centering
\labellist 
\tiny
\pinlabel $r$ [t] at 157 0
\pinlabel $c$ [t] at 80 0
\pinlabel $h$ [t] at 3 0
\pinlabel $a^2b^2\bar{a}\bar{b}$ [br] at 141 3
\pinlabel $pq\bar{p}\bar{q}$ [bl] at 15 3
\small
\pinlabel $\langle a,b\rangle$ [b] at 157 4
\pinlabel $\mathbb{Z}$ [b] at 80 4
\pinlabel $\langle p,q\rangle$ [b] at 3 4
\endlabellist
  \includegraphics{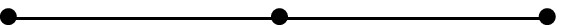}
  \caption{$\gog'$}
  \label{fig:twofoldcovered}
\end{figure}

\subsubsection{The example of \fullref{sec:firstbasicexample}}
Let $\gog$ be the graph of groups of \fullref{fig:biggog}, for which
$\tree(\gog)=\cyl(G(\gog))$.
Let $G:=G(\gog)$ and $\tree:=\tree(\gog)$.
In this example we take the initial decoration to be by vertex type
and quasi-isometry type of vertex stabilizer.
The vertices of $\tree$ carry four possible ornaments, $(r,\llbracket\pi_1(M)\rrbracket)$,
$(r,\llbracket F_2\rrbracket)$, $(h,\llbracket F_2\rrbracket)$, and $(c,\llbracket \mathbb{Z}\rrbracket)$, where in the
first coordinate $r$, $h$, and $c$ stand for `rigid', hanging', and
`cylindrical', respectively.

In the first refinement step, observe that each rigid and hanging
vertex of $\tree$ is adjacent to infinitely many cylindrical
vertices. 
The initial decoration does not distinguish any of the cylindrical
vertices, so $\decmap_1$ does not give any new information about rigid
and hanging vertices. 

On the other hand, we do distinguish some of the cylindrical
vertices. 
This is because $\decmap_0$ does distinguish between three
different kinds of rigid and hanging vertex, so we can distinguish
cylindrical vertices according to the number and kind of their neighbors.

  \begin{table}[h]
    \centering
    \[
  \begin{array}{l l|c|c|c|c|}
&&  (r,\llbracket\pi_1(M)\rrbracket) & (r,\llbracket F_2\rrbracket) &
(h,\llbracket F_2\rrbracket) &(c,\llbracket \mathbb{Z}\rrbracket)\\
\lift{r_6}&(r,\llbracket\pi_1(M)\rrbracket) &0&0&0&\infty\\
\cline{2-6}
\lift{r_1}, \lift{r_2}, \lift{r_3}, \lift{r_4}, \lift{r_5},
    \lift{r_7}, \lift{r_8}, \lift{r_9}&(r,\llbracket F_2\rrbracket)&0&0&0&\infty\\
\cline{2-6}
\lift{h}&(h,\llbracket F_2\rrbracket)&0&0&0&\infty\\
\cline{2-6}
\lift{c_1}&\multirow{4}{*}{$(c,\llbracket \mathbb{Z}\rrbracket)$}&0&5&0&0\\
\lift{c_2}&&1&1&0&0\\
\lift{c_3}, \lift{c_4}, \lift{c_6}&&0&2&0&0\\
\lift{c_5}&&0&1&1&0
  \end{array}\]
    \caption{First refinement}
    \label{tab:firstrefinement}
  \end{table}

\fullref{tab:firstrefinement} shows $\decmap_1$.
In the first column, for convenience, we give representatives of
$G$--orbits of vertex $v$ in $\tree$ with a given ornament in $\ornaments_1$.
The second column contains $\decmap_0(v)$.
The remaining columns in each row give a row vector encoding
$f_{v,0}$.
For instance, the last row says that a vertex  $v\in G\lift{c_5}$ has
$\decmap_0(v)=(c,\llbracket\mathbb{Z}\rrbracket)$ and is adjacent to
one vertex in $\decmap_0^{-1}((r,\llbracket F_2\rrbracket))$ and one
vertex in $\decmap_0^{-1}((h,\llbracket F_2\rrbracket))$.

 \begin{table}[h]
    \centering
    \[
  \begin{array}{l l|c|c|c|cccc|}
&&  (r,\llbracket\pi_1(M)\rrbracket)& (r,\llbracket F_2\rrbracket) &
(h,\llbracket F_2\rrbracket) &\multicolumn{4}{c|}{(c,\llbracket \mathbb{Z}\rrbracket)}\\
\lift{r_6}&(r,\llbracket\pi_1(M)\rrbracket) &0&0&0&0&\infty&0&0\\
\cline{2-9}
\lift{r_1}&\multirow{4}{*}{$(r,\llbracket F_2\rrbracket)$}&0&0&0&\infty&\infty&0&0\\
\lift{r_4}&&0&0&0&\infty&0&0&\infty\\
\lift{r_2}, \lift{r_3}, \lift{r_5}&&0&0&0&\infty&0&\infty&0\\ 
\lift{r_7}, \lift{r_8}, \lift{r_9}&&0&0&0&0&0&\infty&0\\
\cline{2-9}
\lift{h}&(h,\llbracket F_2\rrbracket)&0&0&0&0&0&0&\infty\\
\cline{2-9}
\lift{c_1}&\multirow{4}{*}{$(c,\llbracket \mathbb{Z}\rrbracket)$}&0&5&0&0&0&0&0\\
\lift{c_2}&&1&1&0&0&0&0&0\\
\lift{c_3}, \lift{c_4}, \lift{c_6}&&0&2&0&0&0&0&0\\
\lift{c_5}&&0&1&1&0&0&0&0
  \end{array}\]
    \caption{Second refinement}
    \label{tab:secondrefinement}
  \end{table}

\fullref{tab:secondrefinement} shows $\decmap_2$. 
The $(c,\llbracket\mathbb{Z}\rrbracket)$ columns are listed in the
same order as the $(c,\llbracket\mathbb{Z}\rrbracket)$ rows of
\fullref{tab:firstrefinement}.
For instance, the fourth row says that for any given vertex in
$G\lift{r_2}$, $G\lift{r_3}$, or $G\lift{r_5}$ in $\tree$, it has
infinitely many neighbors in $G\lift{c_1}$ and infinitely many
neighbors in $G\lift{c_3}\cup G\lift{c_4}\cup G\lift{c_6}$ and no
other neighbors.

\begin{table}[h!]
    \centering
\[  \begin{array}{l l|c|cccc|c|cccc|}
&& (r,\llbracket\pi_1(M)\rrbracket)& \multicolumn{4}{c|}{(r,\llbracket F_2\rrbracket)} &
(h,\llbracket F_2\rrbracket) &\multicolumn{4}{c|}{(c,\llbracket \mathbb{Z}\rrbracket)}\\
\lift{r_6}&(r,\llbracket\pi_1(M)\rrbracket) &0&0&0&0&0&0&0&\infty&0&0\\
\cline{2-12}
\lift{r_1}&\multirow{4}{*}{$(r,\llbracket F_2\rrbracket)$}&0&0&0&0&0&0&\infty&\infty&0&0\\
\lift{r_4}&&0&0&0&0&0&0&\infty&0&0&\infty\\
\lift{r_2}, \lift{r_3}, \lift{r_5}&&0&0&0&0&0&0&\infty&0&\infty&0\\ 
    \lift{r_7}, \lift{r_8}, \lift{r_9}&&0&0&0&0&0&0&0&0&\infty&0\\
\cline{2-12}
\lift{h}&(h,\llbracket F_2\rrbracket)&0&0&0&0&0&0&0&0&0&\infty\\
\cline{2-12}
\lift{c_1}&\multirow{4}{*}{$(c,\llbracket \mathbb{Z}\rrbracket)$}&0&1&1&3&0&0&0&0&0&0\\
\lift{c_2}&&1&1&0&0&0&0&0&0&0&0\\
\lift{c_3}, \lift{c_4}, \lift{c_6}&&0&0&0&1&1&0&0&0&0&0\\
\lift{c_5}&&0&0&1&0&0&1&0&0&0&0
  \end{array}\]
    \caption{Structure invariant}
    \label{tab:thirdrefinement}
  \end{table}
We can again count neighbors according to the new decoration, but we
see in \fullref{tab:thirdrefinement} that this does not distinguish
any additional vertices.
Thus, the refinement process has stabilized, and
\fullref{tab:thirdrefinement} is the structure invariant.
Up to permuting the ordering on $\ornaments_0$ and the ordering of
rows and columns within each $\ornaments_0$--block, this structure
invariant completely determines $\tree$ up to $\decmap_0$--preserving
tree isomorphism.

\section{A new decoration: Stretch factors}\label{sec:stretchfactors}
\subsection{Relative quasi-isometric rigidity}\label{sec:relativerigidity}

In this section, we associate to the JSJ tree of a cylinders of a group new quasi-isometry invariants that take into account the metric information carried by the various two-ended edge groups.
Indeed, consider an infinite order element of an edge stabilizer.
Its image in each of the adjacent vertex groups has some translation
length, and the ratio of these translation lengths gives a stretch
factor that describes how the amalgamation distorts distance as
measured in the vertex groups. 
This stretch factor clearly depends on the choice of metrics of the
vertex groups, so it is not an intrinsic invariant of the group, and,
in general, it is not preserved by quasi-isometries. However, we show
that when the vertex groups satisfy an appropriate notion of
quasi-isometric rigidity---a notion that is satisfied by many interesting classes of groups---then such stretch factors are
indeed quasi-isometry invariants.

\begin{definition}\label{def:weakrigidity}
A finitely generated group  $G$ is \emph{quasi-isometrically rigid
  relative to the peripheral structure}
$\per$, or $(G,\per)$ is
quasi-isometrically rigid, if there exists a proper geodesic metric space
$\modelspace$ with peripheral structure $\per'$ and a quasi-isometry
$\modelmap\from (G,\per) \to (\modelspace,\per')$
such that
\begin{enumerate}
\item $\modelmap_*(
\QIgp(G,\per))$ is a uniform subgroup of 
$\CIgp((\modelspace,\modelmap(\per)))$.\label{item:UCI}
\item If $g\in G$ is an infinite order element fixing an element of
  $\per$ then $i\mapsto \modelmap(g^i)$ is a coarse similitude.\label{item:CS}
\end{enumerate}
 
The pair $(\modelspace,\per')$ is called a \emph{rigid model} for
$(G,\per)$.
\end{definition}
\begin{remark}
  In \fullref{prop:almostisometryhomomorphism} we prove that \ref{item:UCI} implies \ref{item:CS} if $G$
  is hyperbolic.
\end{remark}
\begin{lemma}\label{lemma:ci}
  If $(X,\per')$ is a rigid model for $(G,\per)$ then
  $\modelmap'\circ\inv{\modelmap}\in\CIsom((X,\per'))$ for any $\modelmap,\modelmap'\in \QIsom((G,\per),(X,\per'))$.
\end{lemma}
This fact motivates the terminology `rigid'.
\begin{proof}
For $\modelmap,\modelmap'\in \QIsom((G,\per),(X,\per'))$ we have $\modelmap'\circ\inv{\modelmap}\in\QIsom((X,\per'),(X,\per'))$.
For any $\modelmap\in\QIsom((G,\per),(X,\per'))$ we have
$\CIgp((\modelspace,\per'))<\QIgp((\modelspace,\per'))=\modelmap_*(\QIgp((G,\per)))$,
so if $\modelmap_*(\QIgp((G,\per)))<\CIgp((\modelspace,\per'))$ then $\QIgp((X,\per'))=\CIgp((X,\per'))$.
\end{proof}

\begin{definition}\label{def:length}
  Let $(\modelspace,\per')$ be a rigid model for $(G,\per)$, and let
  $g$ be an infinite order element of $G$ that fixes an element of $\per$.
Define the \emph{$\modelspace$--length of $g$}, $\len_\modelspace(g)$, to be the multiplicative constant of
the coarse similitude from $\mathbb{Z}$ to $\modelspace$ defined by $i\mapsto
\modelmap(g^i)$, where
$\modelmap\in\QIsom((G,\per),(\modelspace,\per'))$.

If a positive power $g^k$ of an infinite order element $g$ fixes an
element of $\per$ define $\len_\modelspace(g):=\frac{1}{k}\len_\modelspace(g^k)$.
\end{definition}

\fullref{lemma:ci} implies that $\len_\modelspace(g)$ is independent
of the choice of quasi-isometry $\modelmap\in\QIsom((G,\per),(\modelspace,\per'))$.

We remark in the second case that if $\pi\from \mathbb{Z}\to
k\mathbb{Z}$ is a closest point projection then $i\mapsto
\modelmap(g^{\pi(i)})$ is a coarse similitude from $\mathbb{Z}$ to
$\modelspace$ with multiplicative constant
$\frac{1}{k}\len_\modelspace(g^k)$, so this is a
sensible definition for $\len_\modelspace(g)$.

\subsection{Examples of relative quasi-isometric rigidity}\label{sec:rigidexamples}
Let $G$ be a finitely presented group.
Let $\mathcal{H}$ be a finite collection of two-ended subgroups of
$G$.
Let $\per$ be the peripheral structure consisting of distinct coarse equivalence classes
of conjugates of elements of $\mathcal{H}$.
\begin{enumerate}
\item If $G$ is quasi-isometric to a space $\modelspace$ such that
  $\Igp(X)=\QIgp(X)$
then $(G,\per)$ is rigid.
The peripheral structure plays no role in this case.
Examples include:
\begin{enumerate}
\item Irreducible symmetric spaces other
than real or complex hyperbolic space; thick Euclidean buildings;
and products of such
\cite{Pan89a, KleLee97}.
\item The `topologically rigid' hyperbolic groups of Kapovich and
  Kleiner \cite{KapKle00}.
\item Certain Fuchsian buildings \cite{BouPaj00,Xie06}.
\item Mapping class groups of non-sporadic hyperbolic surfaces \cite{BehKleMin12}.
\end{enumerate}

\item If $G$ is quasi-isometric to a space $\modelspace$ such that
  $\CIgp(X)=\QIgp(X)$
then $(G,\per)$ is rigid.
Again, the peripheral structure plays no role in this case.
Xie gives an example of a certain solvable Lie group with this property
\cite{Xie12}.

\item If $\modelspace$ is a real or complex hyperbolic space of dimension at least 3
and $G$ is quasi-isometric to $\modelspace$ then $(G,\per)$ is quasi-isometrically
rigid whenever $\mathcal{H}$ is non-empty, by a theorem of
Schwartz \cite{Sch97}.
\item If $\modelspace'$ is the 3-valent tree, $G$ is quasi-isometric to
$\modelspace'$ (so $G$ is virtually free), and $G$ does not virtually split over 0 or 2--ended subgroups
relative to $\mathcal{H}$, then $(G,\per)$ is quasi-isometrically rigid \cite{CasMac11,Cas10splitting}.
In this case the model space $\modelspace$ depends on $\per$, and is
not necessarily isometric to $\modelspace'$.
\item If $\modelspace=\mathbb{H}^2$, $\phi\from G\to \modelspace$ is a quasi-isometry, and $G$ does not
virtually split over 2--ended subgroups relative to $\mathcal{H}$,
then $(G,\per)$ is quasi-isometrically rigid, as follows.
A result of Kapovich and Kleiner \cite{KapKle00} shows that $G$ has finite index in $\QIgp((G,\per))$.
Therefore, $\QIgp((G,\per))$ is a finitely generated group
quasi-isometric to $\modelspace$.
This quasi-isometry induces a
cobounded quasi-action of $\QIgp((G,\per))$ on $\modelspace$.
Such a quasi-action is quasi-isometrically conjugate to an isometric
action on $\modelspace$, by a theorem of Markovic \cite{Mar06}.
\end{enumerate}

The first four cases actually satisfy the following stronger version of quasi-isometric rigidity: 

\begin{definition}\label{def:rigid}
We say  $G$ is \emph{strongly quasi-isometrically rigid relative to} $\per$, or $(G,\per)$ is
 strongly quasi-isometrically rigid, if there is a proper geodesic space
 $\modelspace$ such that if $(X,\per')$ and $(X,\per'')$ are rigid
 models for $(G,\per)$, then there is a coarse isometry $\phi$ of $X$
 such that $\phi(\per')=\per''$.
\end{definition}

By contrast, the last case only satisfies the weaker version of rigidity.

For a non-example, consider $G:=\mathbb{Z}^n$ and
$X:=\mathbb{R}^n$. 
For any $\mathcal{H}$ the group $\QIgp((G,\per))$ contains maps
conjugate to homotheties of $\mathbb{R}^n$.
This implies that the multiplicative constants in $\QIgp((G,\per))$ are
unbounded, so $\QIgp((G,\per))$ cannot be conjugate into some coarse
isometry group.

It is also easy to find non-examples of relative rigidity via splittings:
  If $G$ virtually splits over a zero or two-ended group relative to
  $\mathcal{H}$ then $(G,\per)$ is not quasi-isometrically rigid. 
In such an example there are generalized Dehn twist quasi-isometries
preserving $\per$,
powers of which again produce unbounded multiplicative constants in $\QIgp((G,\per))$.

The previous examples and non-examples naturally lead to the following question:

\begin{quest}\label{question:arehyperbolicgroupsrigid}
  If $G$ is a hyperbolic group that is not quasi-isometric
  to $\mathbb{H}^2$ and does not virtually split over a zero or two-ended
  subgroup relative to $\mathcal{H}$, is $(G,\per)$
  quasi-isometrically rigid? Strongly quasi-isometrically rigid?
\end{quest}

\subsection{Relative quasi-isometric rigidity for hyperbolic groups}
In \fullref{theorem:boundaryrigidity} we give a characterization of relative quasi-isometric rigidity for
hyperbolic groups.
 This combines with \fullref{prop:almostisometryhomomorphism} to show
 that the first condition of \fullref{def:weakrigidity} implies the
 second for hyperbolic groups.
\fullref{theorem:boundaryrigidity} also provides an alternative
viewpoint that may be useful for resolving \fullref{question:arehyperbolicgroupsrigid}.

A space $X$ is called \emph{visual} if
there exists an $\cadd\geq 0$ and $x\in X$ such that
for every $y\in X$ there exists an $\cadd$--coarse-geodesic ray
starting at $x$ and passing within distance $\cadd$ of $y$.
It follows from \cite[Proposition~5.2 and Proposition~5.6]{BonSch00}
that if $X$ is quasi-isometric to a visual hyperbolic space then $X$
is a visual hyperbolic space.

\begin{definition}
  A map $\phi\from X\to Y$ is an
  $(\cexp,\cmul)$--\emph{power-quasi-symmetric embedding} if
  for all distinct $x,y,z\in X$
\[\frac{d_Y(\phi(x),\phi(z))}{d_Y(\phi(x),\phi(y))}\leq
\eta\left(\frac{d_X(x,z)}{d_X(x,y)}\right),\]
where
\[\eta(r)=
\begin{cases}
  \cmul r^{\sfrac{1}{\cexp}} &\text{ for } 0<r<1,\\
\cmul r^\cexp &\text{ for } 1\leq r.
\end{cases}\]
\end{definition}

Let $\PQS((X,d_X),(Y,d_Y))$ denote the set of power-quasi-symmetric
homeomorphisms, and abbreviate $\PQS(X,d_X):=\PQS((X,d_X),(X,d_X))$,
which is a group.
If $\mathcal{Z}$ is a collection of subsets of $X$, define:
\[\PQS(X,d,\mathcal{Z}):=\{\phi\in\PQS((X,d))\mid \forall
Z\in\mathcal{Z},\, \phi(Z)\in\mathcal{Z}\text{ and }\exists !Z'\in\mathcal{Z},\,
\phi(Z')=Z\}\]

\begin{theorem}\label{theorem:boundaryrigidity}
  Let $G$ be a non-elementary hyperbolic group with a peripheral structure $\per$
  consisting of coarse equivalence classes of conjugates of finitely
  many two-ended subgroups. 
Fix, arbitrarily, a word metric $d$ on $G$, a basepoint $p\in G$, and
a visual metric $d_\infty$ on $\bdry G$. The following are
equivalent:
\begin{enumerate}
\item $(G,\per)$ is quasi-isometrically rigid.\label{item:wrigid}
\item There exists a proper, geodesic, visual hyperbolic space $\modelspace$ and a
  quasi-isometry $\modelmap\from G\to \modelspace$ such that $\modelmap_*(
\QIgp((G,\per)))$ is a uniform subgroup of $\CIgp((\modelspace,\modelmap(\per)))$.\label{item:modelspace}
\item There exists a visual metric $d_\infty'$ on $\bdry G$ such that
  $\PQS(\bdry G,d_\infty,\bdry\per)$ is power-quasi-symmetrically conjugate to
  a uniform subgroup of $\bilipschitz(\bdry G,d_\infty')$.\label{item:bilipschitz}
\end{enumerate}
\end{theorem}
\begin{proof}
  Since $G$ is a visual hyperbolic space, so is $\modelspace$.
Thus, equivalence of \ref{item:wrigid} and \ref{item:modelspace} 
  is immediate from the definition of relative rigidity.

Item \ref{item:modelspace} implies item \ref{item:bilipschitz} by taking
$d_\infty'$ to be a visual metric on $\bdry X=\bdry G$ and applying
\cite[Theorem~6.5]{BonSch00}, which shows that $\modelmap$ extends to
a power-quasi-symmetry of $\bdry G$, and a uniform subgroup of
$\CIgp(\modelspace)$ extends to a uniform subgroup of $\bilipschitz(\bdry G,d_\infty')$.

Conversely, there are several `hyperbolic cone' constructions in the
literature \cite{Gro87,BouPaj03,BonSch00,BuySch07} that
take a metric space $Z$ and produce a hyperbolic metric space
$\cone(Z)$ such that a visual metric on $\bdry\cone(Z)$ recovers $Z$
with the given metric.
We take $\cone_r(Z)$ be the `truncated hyperbolic approximation with
parameter $r$' of
Buyalo and Schroeder \cite{BuySch07}.

Item \ref{item:bilipschitz}  implies item \ref{item:modelspace} 
by taking $\modelspace$ to be the hyperbolic cone $\cone_r(\bdry
G,d_\infty')$ for $r$ sufficiently small.
\end{proof}

Let $g$ be an infinite order element of a hyperbolic group $G$ with a
fixed word metric.
The isometry $L_g$ defined by left multiplication by $g$ has a well
defined translation length, which is positive.
If $\modelmap\from G\to\modelspace$ is a quasi-isometry such that
$\modelmap_*(L_g)$ is a coarse isometry, it is not true in general
that $\modelmap_*(L_g)$ has a well defined translation length. 
The next proposition shows that we do still get a positive translation
length for $\modelmap_*(L_g)$ in the special case of relative rigidity.

\begin{proposition}\label{prop:almostisometryhomomorphism}
  Let $G$ be a hyperbolic group and $\per$ a peripheral structure such
that there exists a proper geodesic space $\modelspace$ and a
  quasi-isometry $\modelmap\from G\to \modelspace$ such that $\modelmap_*(
\QIgp((G,\per)))$ is a uniform subgroup of $\CIgp((\modelspace,\modelmap(\per)))$.
For every infinite order element $g\in G$ the map $i\mapsto
\modelmap(g^i)$ is a coarse similitude.
\end{proposition}

Before proving the proposition we need a few lemmas.

\begin{lemma}\label{lemma:distboundbyqiconst}
Let $X$ be proper geodesic space quasi-isometric to a non-elementary hyperbolic group.
  For
  $i\in\{0,1\}$, let $\phi_i$ be a quasi-isometry
  of $X$, with $[\phi_0]=[\phi_1]\in\QIgp(X)$.
The distance between $\phi_0$ and $\phi_1$ is bounded in terms of the
quasi-isometry constants of $\phi_0$ and $\phi_1$ and the constants of $X$.
\end{lemma}
\begin{proof}
The following argument is standard, see for instance \cite{Tuk85}. 
We briefly outline the proof. 

Since $X$ is quasi-isometric to a visual hyperbolic space, it is a
visual hyperbolic space.
Furthermore, every point $x\in X$ can be realized as a quasi-center of
an ideal geodesic triangle $\Delta$. 
Since $\phi_0$ and $\phi_1$ are coarsely equivalent,
$\bdry\phi_0=\bdry\phi_1$, so $\phi_0(\Delta)$ and $\phi_1(\Delta)$
are ideal quasi-geodesic triangles with the same ideal vertices, and
quasi-geodesic constants depending on those of $\phi_0$ and $\phi_1$, respectively. 
The set of quasi-centers of uniformly quasi-geodesic triangles with
the same vertices is bounded in terms of the quasi-geodesic constants
and the hyperbolicity constant of $X$, and $\phi_0(x)$ and $\phi_1(x)$ both lie
in this set.
\end{proof}

\begin{cor}\label{cor:coarseAIisAI}
  If $\phi$ is an $(\cmul,\cadd)$--quasi-isometry and $\psi$ is an
  $\cadd'$--coarse isometry with $[\phi]=[\psi]\in\QIgp(X)$ then there
  is an $\cadd''$ depending only on $\cmul$, $\cadd$, $\cadd'$ and $X$ such that $\phi$ is an $\cadd''$--coarse isometry.
\end{cor}

\begin{lemma}\label{lemma:almostsymmetricembedding}
Let $\modelmap\from Y\to\modelspace$ be a quasi-isometry between visual hyperbolic
spaces. 
Suppose that $\phi$ is a loxodromic isometry of $Y$ with translation length
$\tau$.
Let $y_0$ be a point on an axis of $\phi$, and set $y_i=\phi^i(y_0)$ and
$x_i=\modelmap(y_i)$.
Suppose that $\{\modelmap_*(\phi^i)\}_{i\in\mathbb{Z}}$ are uniform coarse isometries.
Then
\[L:=\lim_{i\to\infty}\frac{d(x_0,x_i)}{i}\]
exists, and there exists an $\cadd$ such that $i\mapsto \modelmap(\phi^i(y_0))$
is an $(L,2\cadd)$--coarse similitude.
\end{lemma}
\begin{proof}
The fact that $\{\modelmap_*(\phi^i)\}_{i\in\mathbb{Z}}$ are uniform coarse
isometries implies that the difference $|d(x_i,x_j)-d(x_0,x_{j-i})|$ is uniformly bounded.
Quasi-geodesic stability further implies that
$|d(x_0,x_{i+j})-d(x_0,x_i)-d(x_0,x_j)|$ is uniformly bounded. 
Let $\cadd$ be the greater of these two bounds.

Suppose $L^+:=\limsup\sfrac{d(x_0,x_i)}{i}$ and
$L^-:=\liminf\sfrac{d(x_0,x_i)}{i}$ are different.
Note $L^->0$ since $i\mapsto x_i$ is a quasi-geodesic. 
Take $\epsilon:=\sfrac{(L^+-L^-)}{3}$.
Choose some $i$ such that $\sfrac{d(x_0,x_i)}{i}<L^-+\epsilon$ and
such that $\alpha:=\frac{d(x_0,x_i)+2\cadd}{d(x_0,x_i)}<\sqrt{\frac{2L^++L^-}{L^++2L^-}}$.
Choose some $j$ such that
$\sfrac{d(x_0,x_i)}{i}>L^+-\epsilon$ and such that
$\frac{q+1}{q}<\sqrt{\frac{2L^++L^-}{L^++2L^-}}$, where $q$ is the
integer such that $qi\leq j<(q+1)i$. 

The previous inequalities, together with the triangle inequality to decompose $d(x_0, x_j)$ along  $x_0, x_i, \ldots, x_{qi}, x_j$, yields:

\begin{align*} 
  L^+-\epsilon<&\frac{d(x_0,x_j)}{j}\leq\frac{d(x_0,x_j)}{qi}\leq\frac{(q+1)(d(x_0,x_i)+2\cadd)}{qi}\\
&=\frac{\alpha(q+1)
  d(x_0,x_i)}{qi}\leq \frac{\alpha(q+1)}{q}\cdot
  (L^-+\epsilon)\\
&< \frac{2L^++L^-}{L^++2L^-}\cdot \frac{L^++2L^-}{3}=L^+-\epsilon
\end{align*}
This is a contradiction, so $L^+=L^-=L$. 

Suppose there exists an $i$ such that $d(x_0,x_i)<Li-2\cadd$. 
Then
$L=\lim_{j\to\infty}\frac{d(x_0,x_{ij})}{ij}<\lim_{j\to\infty}\frac{Lij}{ij}=L$,
which is a contradiction. 

If there exists an $i$ such that $d(x_0,x_i)>Li+2\cadd$, then a
similar computation leads to a contradiction.
Therefore, $|d(x_0,x_i)-Li|\leq 2\cadd$, which means $i\mapsto
x_i=\modelmap(\phi^i(y_0))$ is an $(L,2\cadd)$--coarse similitude.
\end{proof}

\begin{proof}[{Proof of \fullref{prop:almostisometryhomomorphism}}]
Let $L_g$ be left multiplication on $G$ by an infinite order element.
For all $n\in\mathbb{Z}$ the map $L_g^n$ is an isometry, so
$\modelmap_*(L_g^n)$ is a quasi-isometry whose constants depend only
on those of $\modelmap$.
By hypothesis, there exists an $\cadd$ such that for each $n$ there
exists an $\cadd$--coarse isometry $c_n\in \CIgp(X)$ with $[c_n]=[\modelmap_*(L_g^n)]$.
By \fullref{cor:coarseAIisAI}, there exists an $\cadd'$ independent of
$n$ such that $\modelmap_*(L_g^n)$ is an $\cadd'$--coarse isometry.
Now apply \fullref{lemma:almostsymmetricembedding}.
\end{proof}

\subsection{Stretch factors}
Let $G$ be a finitely presented, one-ended group such that $\tree:=\cyl(G)$
has two-ended edge stabilizers. 
Let $\gog:=\lquotient{G}{\cyl(G)}$. 
Let $Y$ be an algebraic tree of spaces for $G$ over $\tree$.

Vertices $v_0$ and $v_1$ of $\tree$ that belong to a common cylinder
have stabilizer groups that intersect in a virtually cyclic subgroup.

Recall that $\hat{\Delta}$ denotes the modulus of \fullref{def:extendedmodulus}.

\begin{definition}
  Let $v_0$ and $v_1$ be distinct quasi-isometrically rigid vertices of
  $\tree$ contained in a common cylinder $c$.
For $i\in\{0,1\}$, choose a rigid model $(\modelspace_{v_i},\per_{v_i})$ for $(G_{v_i},\per^\gog_{v_i})$.
Let $\langle z\rangle$ be an infinite cyclic subgroup of $G_{v_0}\cap
G_{v_1}$, and define the \emph{relative stretch} from $v_0$ to $v_1$ to be:
\[\rs(v_0,v_1,(\modelspace_{v_0},\per_{v_0}),(\modelspace_{v_1},\per_{v_1})):=\frac{\len_{\modelspace_{v_1}}(z)}{\len_{\modelspace_{v_0}}(z)}\]
\end{definition}

Clearly,
$\rs(v_0,v_1,(\modelspace_{v_0},\per_{v_0}),(\modelspace_{v_1},\per_{v_1}))$
depends on the choices of $(\modelspace_{v_i},\per_{v_i})$.
Recall, by \fullref{lemma:ci}, it does not depend on the choice of
quasi-isometries $(G_{v_0},\per^\gog_{v_0})\to (\modelspace_{v_0},\per_{v_0})$ and $(G_{v_1},\per^\gog_{v_1})\to (\modelspace_{v_1},\per_{v_1})$.
\begin{lemma}
$\rs(v_0,v_1,(\modelspace_{v_0},\per_{v_0}),(\modelspace_{v_1},\per_{v_1}))$ does not
depend on the choice of $\langle
z\rangle<G_{v_0}\cap G_{v_1}$.
\end{lemma}
\begin{proof}
For $i\in\{0,1\}$, let $e_i$ be an edge on the geodesic in $\tree$ between $v_0$ and
$v_1$, with $\initial(e_i)=v_i$.
  Let $\langle z_{0}\rangle<G_{e_0}$ and $\langle z_{1}\rangle<G_{e_1}$ be infinite cyclic subgroups of
  minimal index.
Since $G_{e_0}$ and $G_{e_1}$ are virtually cyclic, $\langle z\rangle$
has finite index in each of them.
\[\frac{\len_{\modelspace_{v_1}}(z)}{\len_{\modelspace_{v_0}}(z)}%
=\frac{\len_{\modelspace_{v_1}}(z_{1})\cdot
\frac{[\langle z_{1}\rangle : \langle z_{1}\rangle\cap \langle z\rangle]}{[\langle z\rangle : \langle z_{1}\rangle\cap \langle z\rangle]}}{\len_{\modelspace_{v_0}}(z_{0})\cdot
\frac{[\langle z_{0}\rangle : \langle z_{0}\rangle\cap \langle
  z\rangle]}{[\langle z\rangle : \langle z_{0}\rangle\cap \langle
  z\rangle]}}%
=\frac{\len_{\modelspace_{v_1}}(z_{1})}{\len_{\modelspace_{v_0}}(z_{0})}\cdot\frac{[\langle
  z_1\rangle:\langle z_0\rangle\cap\langle z_1\rangle]}{[\langle
  z_0\rangle:\langle z_0\rangle\cap\langle z_1\rangle]}%
=\frac{\len_{\modelspace_{v_1}}(z_{1})}{\len_{\modelspace_{v_0}}(z_{0})}\cdot\hat{\Delta}(e_0,e_1)\]
The right-hand side is independent of the choice of $z_0$ and $z_1$,
since if, say, $\langle z_0'\rangle$ is another infinite cyclic subgroup of minimal
index in $G_{e_0}$ then: \[\len_{\modelspace_{v_0}}(z_0')=\len_{\modelspace_{v_0}}(z_0)\cdot\frac{[\langle
  z_0\rangle : \langle z_0\rangle\cap\langle z_0'\rangle]}{[\langle
  z_0'\rangle : \langle z_0\rangle\cap\langle z_0'\rangle]}=\len_{\modelspace_{v_0}}(z_0)\qedhere\]
\end{proof}

\begin{corollary}\label{corollary:stretchesmultiply}
  If $v_0$, $v_1$, and $v_2$ are quasi-isometrically rigid
  vertices in a common cylinder then:
\begin{align*}
  \rs(v_0,v_1,(\modelspace_{v_0},\per_{v_0}),(\modelspace_{v_1},\per_{v_1}))&\cdot
\rs(v_1,v_2,(\modelspace_{v_1},\per_{v_1}),(\modelspace_{v_2},\per_{v_2}))\\
&=\rs(v_0,v_2,(\modelspace_{v_0},\per_{v_0}),(\modelspace_{v_2},\per_{v_2}))
\end{align*}
\end{corollary}

\begin{proposition}\label{prop:stretchpreserved}
  Let $\qi$ be a quasi-isometry between finitely presented, one-ended
  groups $G$ and $G'$ whose JSJ trees of cylinders have two-ended edge
  stabilizers.
Let $\gog:=\lquotient{G}{\cyl(G)}$ and $\gog':=\lquotient{G'}{\cyl(G')}$.
Suppose that $v_0$ and $v_1$ are distinct quasi-isometrically rigid vertices
of $\tree(\gog)$ contained in a cylinder $c$.
Choose rigid models $(\modelspace_{v_0},\per_{v_0})$ and
$(\modelspace_{v_1},\per_{v_1})$ for $(G_{v_0},\per^\gog_{v_0})$ and
$(G_{v_1},\per^\gog_{v_1})$, respectively. 
Then: \[\rs(v_0,v_1,(\modelspace_{v_0},\per_{v_0}),(\modelspace_{v_1},\per_{v_1}))=\rs(\qi_*(v_0),\qi_*(v_1),(\modelspace_{v_0},\per_{v_0}),(\modelspace_{v_1},\per_{v_1}))\]
\end{proposition}
\begin{proof}
Let $Y$ be an algebraic tree of spaces for $\gog$, and let
$Y'$ be an algebraic tree of spaces for $\gog'$.
For $v\in\{v_0,v_1\}$ choose
$\modelmap_{v}\in\QIsom((Y_{v},\per^\gog_{v}),(\modelspace_{v},\per_{v}))$.
Note that $\modelmap_v\circ\overline{\qi}_v\in\QIsom((Y'_{\qi_*(v)},\per_{\qi_*(v)}^{\gog'}) ,(\modelspace_{v},\per_{v}))$.
Define: 
\[R:=\rs(v_0,v_1,(\modelspace_{v_0},\per_{v_0}),(\modelspace_{v_1},\per_{v_1}))\]
\[R':=\rs(\qi_*(v_0),\qi_*(v_1),(\modelspace_{v_0},\per_{v_0}),(\modelspace_{v_1},\per_{v_1}))\]

Choose two points $x_0$ and $x_1$ in $\modelmap_{v_0}(Y_{e_0})$ such
that $d_{\modelspace_{v_0}}(x_0,x_1)\gg 0$.
The idea of the proof is to approximate $R$ by a quantity $Q(x_0,x_1)$
depending on $x_0$ and $x_1$, and similarly approximate $R'$ by
$Q'(x_0,x_1)$, and then show:
\[R=\lim_{d(x_0,x_1)\to\infty} Q(x_0,x_1)=\lim_{d(x_0,x_1)\to\infty} Q'(x_0,x_1)=R'\]
In the following, quantities are `coarsely well defined' if they are
well defined up to additive error independent of the choice of $x_0$
and $x_1$.

By construction, $Y_{e_0}$ is a coset of $G_{\quot{e_0}}$, and $\langle
z_{\quot{e_0}}\rangle$ is an infinite cyclic subgroup of minimal index in
$G_{\quot{e_0}}$.
Let $g_0\in G$ such that $Y_{e_0}=g_0G_{\quot{e_0}}$.
Since $g_0\langle z_{\quot{e_0}}\rangle$ is coarsely dense in $Y_{e_0}$, there
exist integers $k_i$ such that
$d_{Y_{v_0}}(\modelmap^{-1}(x_i),g_0z_{\quot{e_0}}^{k_i})$ is small. 

$Y_{e_0}$ and $Y_{e_1}$ are coarsely equivalent, so closest point projection
$\pi\from Y_{e_0}\to Y_{e_1}$ is coarsely well defined.
Moreover, since $G_{e_0}$ and $G_{e_1}$ are commensurable, there exist
$\epsilon\in\{\pm 1\}$ and $l\in\mathbb{Z}$ such that
$\pi(g_0z^j_{\quot{e_0}})$ is bounded distance from $g_1z^{\epsilon j\hat{\Delta}(e_0,e_1)
  +l}_{\quot{e_1}}$, where $z^{\epsilon j\hat{\Delta}(e_0,e_1)
  +l}_{\quot{e_1}}$ is to be interpreted as $z_{\quot{e_1}}$ raised to the
greatest integer less than or equal to $\epsilon j\hat{\Delta}(e_0,e_1)
  +l$.
Now we have the following string of relations, where $\sim$ indicates
equality up to additive error in the numerator and denominator,
independent of $x_0$ and $x_1$.

{\allowdisplaybreaks\begin{align*}
R&=\frac{\len_{\modelspace_{v_1}}(g_1z_{\quot{e_1}}g^{-1}_1)}{\len_{\modelspace_{v_0}}(g_0z_{\quot{e_0}}g^{-1}_0)}\cdot\hat{\Delta}(e_0,e_1)\\
&\sim\frac{\left(\frac{d_{\modelspace_{v_1}}(\modelmap_{v_1}(g_1z_{\quot{e_1}}^{\epsilon
  k_0\hat{\Delta}(e_0,e_1)+l}),\modelmap_{v_1}(g_1z_{\quot{e_1}}^{\epsilon
  k_1\hat{\Delta}(e_0,e_1)+l}))}{|(\epsilon
  k_1\hat{\Delta}(e_0,e_1)+l)-(\epsilon
  k_0\hat{\Delta}(e_0,e_1)+l)|}\right)}{\left(\frac{d_{\modelspace_{v_0}}(\modelmap_{v_0}(g_0z_{\quot{e_0}}^{k_0}),\modelmap_{v_0}(g_0z_{\quot{e_0}}^{k_1}))}{|k_1-k_0|}\right)}\cdot\hat{\Delta}(e_0,e_1)\\
&=\frac{d_{\modelspace_{v_1}}(\modelmap_{v_1}(g_1z_{\quot{e_1}}^{\epsilon
  k_0\hat{\Delta}(e_0,e_1)+l}),\modelmap_{v_1}(g_1z_{\quot{e_1}}^{\epsilon
  k_1\hat{\Delta}(e_0,e_1)+l}))}{d_{\modelspace_{v_0}}(\modelmap_{v_0}(g_0z_{\quot{e_0}}^{k_0}),\modelmap_{v_0}(g_0z_{\quot{e_0}}^{k_1}))}\\
&\sim
  \frac{d_{\modelspace_{v_1}}(\modelmap_{v_1}(\pi(g_0z_{\quot{e_0}}^{k_0})),\modelmap_{v_1}(\pi(g_0z_{\quot{e_0}}^{k_1})))}{d_{\modelspace_{v_0}}(\modelmap_{v_0}(g_0z_{\quot{e_0}}^{k_0}),\modelmap_{v_0}(g_0z_{\quot{e_0}}^{k_1}))}\\
&\sim\frac{d_{\modelspace_{v_1}}(\modelmap_{v_1}\circ\pi\circ\modelmap_{v_0}^{-1}(x_0),\modelmap_{v_1}\circ\pi\circ\modelmap_{v_0}^{-1}(x_1))}{d_{\modelspace_{v_0}}(x_0,x_1)}=:Q(x_0,x_1)
\end{align*}}
We conclude $R=\lim_{d(x_0,x_1)\to\infty} Q(x_0,x_1)$.

Similarly, if $\pi'$ is closest point projection from $\qi(Y_{e_0})$
to $\qi(Y_{e_1})$ define:
\[Q'(x_0,x_1):=\frac{d_{\modelspace_{v_1}}(\modelmap_{v_1}\circ\inv{\qi}\circ\pi'\circ\qi\circ\modelmap_{v_0}^{-1}(x_0),\modelmap_{v_1}\circ\inv{\qi}\circ\pi'\circ\qi\circ\modelmap_{v_0}^{-1}(x_1))}{d_{\modelspace_{v_0}}(x_0,x_1)}\]
We have $R'=\lim_{d(x_0,x_1)\to\infty}Q'(x_0,x_1)$.  
However, since $Y_{e_0}$ and $Y_{e_1}$ are coarsely equivalent, $\inv{\qi}\circ\pi'\circ\qi$ is coarsely equivalent to
$\pi$, so $Q(x_0,x_1)\sim Q'(x_0,x_1)$.
We conclude:
\[R=\lim_{d(x_0,x_1)\to\infty} Q(x_0,x_1)=\lim_{d(x_0,x_1)\to\infty}
Q'(x_0,x_1)=R'\qedhere\]
\end{proof}

\subsection{Uniformization}\label{sec:uniformization}
The stretch factors defined in the previous section depend on the
choice of rigid model for the vertex groups.
We suppress this dependence by choosing models uniformly:
\begin{definition}\label{def:qitypes}
Let $\QItypes:=\{ \llbracket (G,\per)\rrbracket\}$ be the set of
quasi-isometry classes of finitely presented groups relative to
peripheral structures.
For each $Q\in\QItypes$ choose a proper, geodesic space $Z_Q$ with
peripheral structure $\per_Q$ such that $Q=\llbracket
(Z_Q,\per_Q)\rrbracket$.
Define $\model(Q):=(Z_Q,\per_Q)$.
If $(Z_Q,\per_Q)$ is 
quasi-isometrically rigid then we choose $(Z_Q,\per_Q)$ to be a
rigid model as in \fullref{sec:relativerigidity}.
We choose $Z_{\llbracket \mathbb{R},\mathbb{R}\rrbracket}=\mathbb{R}$.
\end{definition}

\begin{definition}
If $v_0$ and $v_1$ are quasi-isometrically rigid vertices in a
cylinder, define:
  \[\rs(v_0,v_1)=\rs(v_0,v_1,\model(G_{v_0},\per_{v_0}^\gog),\model(G_{v_1},\per_{v_1}^\gog))\]
\end{definition}

\subsection{Normalization for unimodular graphs of groups}\label{sec:normalization}
Suppose that $\tree:=\cyl(G)$ has two-ended edge stabilizers and $c$ is a
unimodular cylinder in $\tree$.
Suppose that $c$ contains some quasi-isometrically rigid vertices.
Unimodularity implies $\{\rs(v_0,v_1)\mid
v_0,v_1\in c\text{ are qi rigid}\}$ is bounded.
Since stretch factors multiply by
\fullref{corollary:stretchesmultiply}, there exists a
quasi-isometrically rigid $v_0$ such that for every other 
quasi-isometrically rigid vertex $v_1$ in $c$ we have
$\rs(v_0,v_1)\geq 1$.
Define $\rs(c,v_1):=\rs(v_0,v_1)$.
\begin{definition}\label{def:relstretch}
Suppose that $\tree:=\cyl(G)$ has two-ended edge stabilizers.
Let $e$ be an edge of $\tree$ connecting a cylindrical vertex $c$ to a quasi-isometrically
rigid vertex $v$.
Define $\rs(e):=\rs(c,v)$.
\end{definition}

\subsection{An example}\label{sec:examplestretch}
Recall \fullref{ex:keyexample}:
Let $G_i:=\left<a, b, t\mid \inv{t}u_it=v_i\right>$, where $u_i$ and
$v_i$ are words in $\left<a,b\right>$ given below. 
In each case $G_i$ should be thought of as an HNN extension of
$\left<a, b\right>$ over $\mathbb{Z}$ with stable letter $t$.
Subdividing the edge gives a unimodular JSJ decomposition $\gog_i$ of $G_i$ whose Bass-Serre $\tree_i$ is equal to its tree of cylinders.

Let $u_0:=a$, $v_0:=ab\inv{a}\inv{b}^2$, $u_1:=ab$, $v_1:=a^2\inv{b}^2$,
$u_2:=ab^2$, and $v_2:=a^2\inv{b}$.

The methods of \cite{CasMac11}
show that for each $i$, the pair $\{u_i,v_i\}$ is Whitehead
minimal, with Whitehead graph equal to the complete graph on 4
vertices, so $\left<a,b\right>$ is quasi-isometrically rigid relative
to the peripheral structure $\per_i$ coming from incident edge groups, and
the rigid model space is just the Cayley graph for $\left<a,b\right>$
with respect to $\{a,b\}$, ie, the 4--valent tree.
Furthermore, $\QIgp((\left<a,b\right>,\per_i))$ is transitive on
$\per_i$, and $\llbracket(\left<a,b\right>,
\per_1)\rrbracket=\llbracket(\left<a,b\right>,\per_2)\rrbracket=\llbracket(\left<a,b\right>,\per_3)\rrbracket$.
Therefore, \fullref{main} cannot distinguish $G_1$, $G_2$, and $G_3$.

In $\tree_i$, $V_C$ is a single orbit, and each vertex $c\in V_C$ has
valence two. Fix some $c\in V_C$, and suppose $\rig(c)=\{e,e'\}$.
Let us assume that $e$ is an edge that attaches to
$\modelspace_{\terminal(e)}$ along the image of a conjugate of $\left<v_i\right>$ and $e'$ is an edge that attaches to
$\modelspace_{\terminal(e')}$ along the image of a conjugate of $\left<u_i\right>$.
The stabilizers of $e$ and $e'$ are equal to the stabilizer
of $c$, which is infinite cyclic.
This, together with the fact that the rigid model vertex space is the Cayley
graph (tree) for $\left<a,b\right>$ with respect to $\{a,b\}$, and the
fact that each $u_i$ and $v_i$ is cyclically reduced with respect to $\{a,b\}$, means that $\mathrm{Str}(e)$ is the word
length of $v_i$ in $\left<a,b\right>$, and $\mathrm{Str}(e')$ is the
word length of $u_i$ in $\left<a,b\right>$.
Thus, $\rs(e')=1$ and $\rs(e)=\sfrac{|v_i|}{|u_i|}=5,\,2,\,\text{ or
}1$, as $i=0,\,1\,\text{ or }2$, respectively.
By \fullref{prop:stretchpreserved}, no one of these groups is
quasi-isometric to the other.

%\section{Model spaces and trees of quasi-isometries}\label{sec:modelspace}

\section{Vertex constraints}\label{sec:vertexconstraints}
In this section we assume that $G$ is a one-ended, finitely presented
group such that $\tree:=\cyl(G)$ has two-ended cylinder stabilizers. 
Let $\gog$ be the quotient graph of cylinders, which is therefore a
canonical JSJ decomposition of $G$ over two-ended subgroups; recall \fullref{sec:restrictedtoc}.

We suppose that $\modelspace$ is a tree of spaces over $\tree$,
quasi-isometric to $G$.

In \fullref{sec:structureinvariants} we saw how to decide if two vertices of
$\tree$ are in the same $\Aut(\tree,\decmap)$--orbit.
In this section we would like to restrict further to subgroups of
$\Aut(\tree,\decmap)$ induced by $\QIgp(\modelspace)$, or, in the case
that $G$ is hyperbolic, by $\Homeo(\bdry\modelspace)$.
We will actually do something that is weaker in the quasi-isometry
case, but has the advantage that the same approach works for both
quasi-isometries and boundary homeomorphisms. 
What we do is restrict to elements of $\Aut(\tree,\decmap)$ that
at each vertex look like they are induced by a quasi-isometry or boundary
homeomorphism of the appropriate vertex space.
We also add a compatibility condition below.
First we explain the notation.

For $[\qi]\in\QIgp(\modelspace)$ we can choose a representative
$\qi$ that induces an automorphism $\qi_*$ of $\tree$ and splits as a tree of quasi-isometries  $\qi_v:=\qi|_{\modelspace_v}\in\QIgp((\modelspace_v,\per_v),(\modelspace_{\qi_*(v)},\per_{\qi_*(v)}))$ over $\tree$.

Similarly, if $G$ is hyperbolic, then $\modelspace$ is hyperbolic and
$\qi\in\Homeo(\bdry\modelspace)$ induces an automorphism
$\qi_*$ of $\tree$ and splits as a tree of boundary homeomorphisms
$\qi_v:=\qi|_{\bdry\modelspace_v}\in\Homeo((\bdry\modelspace_v,\bdry\per_v),(\bdry\modelspace_{\qi_*(v)},\bdry\per_{\qi_*(v)}))$
over $\tree$.

Since cylinders are two-ended, each edge space is a quasi-line $\cL$ in its
respective vertex space.
Recall this means that there is a controlled embedding $\Xi$ of
$\mathbb{R}$ with image $\cL$.
In the hyperbolic case $\Xi$ is actually a quasi-isometric embedding, and $\cL$
has distinct endpoints at infinity in the boundary of the vertex space
containing it.
In this case we define an orientation of $\cL$ to be a choice of one
of these boundary points, and a boundary homeomorphism of the vertex
space that preserves $\bdry\cL$ is said to be orientation preserving
if it fixes $\bdry\cL$ and orientation reversing if it exchanges the
two points of $\bdry\cL$.

In the quasi-isometry case we know that $\Xi([0,\infty))$ and
$\Xi([0,-\infty))$ are not coarsely equivalent.
We define an orientation of $\cL$ to be a choice of coarse equivalence class
of either  $\Xi([0,\infty))$ or $\Xi([0,-\infty))$. 
A quasi-isometry that coarsely preserves $\cL$ is said to be
orientation preserving on $\cL$ if it fixes the coarse equivalence classes of
$\Xi([0,\infty))$ and $\Xi([0,-\infty))$, and orientation reversing if
it exchanges them.

We seek $\iso\in\Aut(\tree,\decmap)$ such that for
every  $v\in\verts\tree$ there exists
an element $\qi_v\in\starmap((\modelspace_v,\per_v),(\modelspace_{\iso(v)},\per_{\iso(v)}))$
such that $(\qi_v)_*=\iso|_{\link(v)}$, subject to the following
compatibility condition.
In the quasi-isometry case we require that
$(\attachmap_{\iso(e)}\circ\qi_{\initial(e)})\circ(\qi_{\terminal(e)}\circ\attachmap_e)^{-1}$
is orientation preserving on $\modelspace_{\overline{\iso(e)}}$.
In the boundary homeomorphism case we require that
$(\bdry\attachmap_{\iso(e)}\circ\qi_{\initial(e)})\circ(\qi_{\terminal(e)}\circ\bdry\attachmap_e)^{-1}$
is the identity on $\bdry\modelspace_{\overline{\iso(e)}}$ for every edge $e$.
For brevity, we say ``$(\attachmap_{\iso(e)}\circ\qi_{\initial(e)})\circ(\qi_{\terminal(e)}\circ\attachmap_e)^{-1}$
is orientation preserving on $\modelspace_{\overline{\iso(e)}}$'' in both cases.

In the boundary homeomorphism case we conclude, in
\fullref{corollary:boundaryhomeomorphism}, that such a collection of
$\qi_v$ patch together to give $\qi\in\Homeo(\bdry\modelspace)$ with $\qi_*=\iso$.

The analogous statement is not true for quasi-isometries.
To patch together quasi-isometries we need
$\attachmap_{\iso(e)}\circ\qi_{\initial(e)}$ and
$\qi_{\terminal(e)}\circ\attachmap_e$ to be coarsely equivalent \emph{as
maps}, but we have only assumed that they have coarsely equivalent
image sets with the same orientations.
We also need to know that the $\qi_v$ have uniform quasi-isometry constants.
These points will be addressed in subsequent sections.

\subsection{Partial Orientations}
  A \emph{partial orientation} $\partialorientation$ of $\modelspace$
  assigns to each cylindrical vertex space and to each peripheral set
  in each non-elementary vertex space
either an orientation of that space or the value
  $\nullvar$.

A cylindrical vertex space or peripheral set is said to be $\partialorientation$--\emph{oriented}
if
its $\partialorientation$ value is not $\nullvar$, and
$\partialorientation$--\emph{unoriented} otherwise.

A cylindrical vertex is said to be $\partialorientation$--oriented or
$\partialorientation$--unoriented if its vertex space is.

An edge $e\in\tree$ is said to be $\partialorientation$--oriented or
$\partialorientation$--unoriented if the corresponding edge space in
its incident non-elementary vertex is.

The \emph{sign} of a map $\qi$ that takes an oriented space $A$ to
an oriented space $B$ is 1 if the map is orientation preserving and $-1$ if it is
orientation reversing.
For a partial orientation $\partialorientation$ we define
$\sign_\partialorientation\qi$ as usual when $A$ and $B$ are both $\partialorientation$--oriented, and we define
$\sign_\partialorientation\qi:=0$ if either of them is $\partialorientation$--unoriented.

One partial orientation, $\partialorientation'$, \emph{extends}
another, $\partialorientation$, if they agree on all
$\partialorientation$--oriented sets.

\subsection{Cylindrical vertices}\label{sec:cylinderrefinement}
\begin{definition}\label{def:imbalance}
Let $\partialorientation$ be a partial orientation.
Let $c$ be a cylindrical vertex.
The \emph{orientation imbalance} at $c$ with
respect to a decoration $\decmap\from\tree\to\ornaments$ and a partial orientation
$\partialorientation$ is the function
$\imbalance_c^{\decmap,\partialorientation}\from\ornaments \to
\mathbb{Z}^\ornaments/\{-1,1\}$, with the action by
coordinate-wise multiplication,
defined as follows.
Choose an orientation of $\modelspace_c$ and for each $e\in\link(c)$
let $\sign\attachmap_e$ denote the sign of
$\attachmap_e$ with respect to the chosen orientation of
$\modelspace_c$ and the $\partialorientation$--orientation of
$\modelspace_{\bar{e}}$, which we take to be 0 if $e$ is
$\partialorientation$--unoriented. Define:

\[\imbalance_c^{\decmap,\partialorientation}(\ornament):=[\sum_{e\in\link(c)\cap\decmap^{-1}(\ornament)}\sign
\attachmap_e]\]

If $\imbalance_c^{\decmap,\partialorientation}$ is non-zero we call
$c$ an \emph{unbalanced cylinder}.
\end{definition}

Taking an equivalence class of function in the definition eliminates
the dependence on the arbitrary choice of orientation of $\modelspace_c$.

\begin{proposition}
Suppose $\decmap$ and $\partialorientation$ are
$\starmap(\modelspace)$--invariant.
  If there exists $\qi\in\starmap(\modelspace)$ such that $\qi_*$
  fixes a cylindrical vertex $c$ and reverses the  orientation of
  $\modelspace_c$ then
  $\imbalance_c^{\decmap,\partialorientation}$ is identically zero.
\end{proposition}
\begin{proof}
  Suppose $\ornament\in\ornaments$ is such that there exists a
  $\partialorientation$--oriented edge
  $e\in\link(c)\cap\decmap^{-1}(\ornament)$.
Let $v:=\terminal(e)$.
By $\starmap(\modelspace)$--invariance,
$\partialorientation(\modelspace_{\overline{\qi_*(e)}})=\qi_v(\partialorientation(\modelspace_{\bar{e}}))=\attachmap_{\qi_*(e)}\circ\qi_c\circ\attachmap_e^{-1}(\partialorientation(\modelspace_{\bar{e}}))$.
Since $\qi_c$ is orientation reversing, $\attachmap_e$
and $\attachmap_{\qi_*(e)}$ have opposite signs.
Therefore, $(\qi)_*|_{\link(c)}$ gives a bijection between edges in
  $\link(c)\cap\decmap^{-1}(\ornament)$ whose attaching map have positive sign and edges
  in $\link(c)\cap\decmap^{-1}(\ornament)$ whose attaching map have negative sign. 
Since this is true for every $\ornament\in\ornaments$ such that
$\link(c)\cap\decmap^{-1}(\ornament)$ is non-empty, $\imbalance_c^{\decmap,\partialorientation}$ is identically zero.
\end{proof}
\begin{corollary}
Suppose $\decmap$ and $\partialorientation$ are
$\starmap(\modelspace)$--invariant.
  If $G_c$ contains an infinite dihedral group then
  $\imbalance_c^{\decmap,\partialorientation}$ is identically zero.
\end{corollary}

\begin{proposition}\label{prop:imbalancewrtqi}
Suppose $\decmap$ and $\partialorientation$ are
$\starmap(\modelspace)$--invariant.
  For every $\qi\in\starmap(\modelspace)$ we have
  $\imbalance_c^{\decmap,\partialorientation}=\imbalance_{\qi_*(c)}^{\decmap,\partialorientation}$.
\end{proposition}
\begin{proof}
Choose some orientation on $\modelspace_c$ and
$\modelspace_{\qi_*(c)}$.

If no edge in $\link(c)\cap\decmap^{-1}(\ornament)$ is $\partialorientation$--oriented
then
$\imbalance_c^{\decmap,\partialorientation}(\ornament)=0$, and, by
$\starmap(\modelspace)$--invariance, the same are true for
$\qi_*(c)$.

Now consider $\ornament\in\ornaments$ such that there exists an edge $e\in\link(c)\cap\decmap^{-1}(\ornament)$ such that $e$
is $\partialorientation$--oriented.
 Let $v:=\terminal(e)$.
By $\starmap(\modelspace)$--invariance,
$\qi_*(e)\in\link(\qi_*(c))\cap\decmap^{-1}(\ornament)$ with $\partialorientation(\modelspace_{\overline{\qi_*(e)}})=\qi_v(\partialorientation(\modelspace_{\bar{e}}))=\attachmap_{\qi_*(e)}\circ\qi_c\circ\attachmap_e^{-1}(\partialorientation(\modelspace_{\bar{e}}))$.
If $\qi_c$ is orientation reversing then $\attachmap_e$
and $\attachmap_{\qi_*(e)}$ have opposite signs, so that: 
\[\sum_{\qi_*(e)\in\link(\qi_*(c))\cap\decmap^{-1}(\ornament)}\sign\attachmap_{\qi_*(e)}=-\left(\sum_{e\in\link(c)\cap\decmap^{-1}(\ornament)}\sign\attachmap_e\right)\]

If $\qi_c$ is orientation preserving then $\attachmap_e$
and $\attachmap_{\qi_*(e)}$ have the same signs, so that:
\[\sum_{\qi_*(e)\in\link(\qi_*(c))\cap\decmap^{-1}(\ornament)}\sign\attachmap_{\qi_*(e)}=\sum_{e\in\link(c)\cap\decmap^{-1}(\ornament)}\sign\attachmap_e\qedhere\]
\end{proof}

The previous proposition shows we can use cylinder imbalances to
distinguish different cylinders.
The following lemma shows this holds up under refinement of the decoration.
\begin{lemma}
Suppose $\decmap'\from\tree\to\ornaments'$ is a refinement of
$\decmap$ and $\partialorientation'$ is an extension of
$\partialorientation$.
Suppose that the $\decmap$--partition of edges of $\tree$ is finer than
the partition into $\partialorientation$--oriented edges and
$\partialorientation$--unoriented edges.
Let $c$ be a cylindrical vertex.
 If $\imbalance_c^{\decmap,\partialorientation}$ is non-zero then $\imbalance_c^{\decmap',\partialorientation'}$ is
  non-zero.

Let $c'$ be a cylindrical vertex distinct from $c$.
If for every $\ornament\in\ornaments$ there exist
$\partialorientation$--oriented edges in
$\link(c)\cap\decmap^{-1}(\ornament)$ if and only if there exist $\partialorientation$--oriented edges in
$\link(c')\cap\decmap^{-1}(\ornament)$
then
$\imbalance_c^{\decmap,\partialorientation}\neq\imbalance_{c'}^{\decmap,\partialorientation}$
implies $\imbalance_c^{\decmap',\partialorientation'}\neq\imbalance_{c'}^{\decmap',\partialorientation'}$.
\end{lemma}
\begin{proof}
If $c$ is unbalanced then there exists an $\ornament\in\ornaments$ such that
$\imbalance_c^{\decmap,\partialorientation}(\ornament)\neq 0$, which
implies that there are $\partialorientation$--oriented edges in $\link(c)\cap\decmap^{-1}(\ornament)$.
Since the $\decmap$--partition of edges of $\tree$ is finer than
the partition into $\partialorientation$--oriented edges and
$\partialorientation$--unoriented edges, all edges in $\link(c)\cap\decmap^{-1}(\ornament)$ are
$\partialorientation$--oriented.
Since $\partialorientation'$ extends $\partialorientation$, all
edges in $\link(c)\cap\decmap^{-1}(\ornament)$ are $\partialorientation'$--oriented, and since
$\decmap'$ refines $\decmap$, we have, with respect to
$\partialorientation'$ and  some fixed
orientation of $\modelspace_c$, that:
\[\sum_{e\in\link(c)\cap\decmap^{-1}(\ornament)}\sign\attachmap_e=\sum_{\ornament'\in\decmap'\circ\decmap^{-1}(\ornament)}\sum_{e\in\link(c)\cap\decmap^{-1}(\ornament')}\sign\attachmap_e\]

The left hand side is non-zero, so one of the terms of the outer sum
on the right hand side must be non-zero. Thus,  $\imbalance_c^{\decmap',\partialorientation'}$ is
  not identically zero.

For the second statement, suppose, for contraposition, that
$\imbalance_c^{\decmap',\partialorientation'}=\imbalance_{c'}^{\decmap',\partialorientation'}$.
Having chosen orientations on $\modelspace_c$ and
$\modelspace_{c'}$, there is an $\epsilon\in\pm 1$ such that for all
$\ornament'\in\ornaments'$:
\[\sum_{e\in\link(c)\cap(\decmap')^{-1}(\ornament')}\sign\attachmap_e=\epsilon \left(\sum_{e\in\link(c')\cap(\decmap')^{-1}(\ornament')}\sign\attachmap_e\right)\]

If $\ornament\in\ornaments$ is such that there are no
$\partialorientation$--oriented edges in either
$\link(c)\cap\decmap^{-1}(\ornament)$ or
$\link(c')\cap\decmap^{-1}(\ornament)$ then:
\[\sum_{e\in\link(c)\cap\decmap^{-1}(\ornament)}\sign\attachmap_e=
0=\epsilon\left(\sum_{e\in\link(c')\cap\decmap^{-1}(\ornament)}\sign\attachmap_e\right)\]

Otherwise, by hypothesis, there are $\partialorientation$--oriented edges in both
$\link(c)\cap\decmap^{-1}(\ornament)$ and
$\link(c')\cap\decmap^{-1}(\ornament)$.
We conclude that
$\imbalance_c^{\decmap,\partialorientation}=\imbalance_{c'}^{\decmap,\partialorientation}$
from the following computation, in which the first and third
equalities are from the facts that the $\decmap$--partition of edges of $\tree$ is finer than
the partition into $\partialorientation$--oriented edges and
$\partialorientation$--unoriented edges and  that $\decmap'$ refines
$\decmap$, and the second equality is from the hypothesis that $\imbalance_c^{\decmap',\partialorientation'}=\imbalance_{c'}^{\decmap',\partialorientation'}$.
\begin{multline*}
\sum_{e\in\link(c)\cap\decmap^{-1}(\ornament)}\sign\attachmap_e=\sum_{\ornament'\in\decmap'\circ\decmap^{-1}(\ornament)}\sum_{e\in\link(c)\cap(\decmap')^{-1}(\ornament')}\sign\attachmap_e\\
=\sum_{\ornament'\in\decmap'\circ\decmap^{-1}(\ornament)}\epsilon\left(\sum_{e\in\link(c')\cap(\decmap')^{-1}(\ornament')}\sign\attachmap_e\right)=\epsilon\left(\sum_{e\in\link(c')\cap\decmap^{-1}(\ornament)}\sign\attachmap_e\right)\qedhere
\end{multline*}
\end{proof}

Given $\decmap$ and $\partialorientation$ that are both
$\starmap(\modelspace)$--invariant we 
define the process of \emph{cylinder refinement}
as follows.
\begin{enumerate}
\item By passing to the coarsest common refinement, we may assume
that the $\decmap$--partition of edges of $\tree$ is finer than
the partition into $\partialorientation$--oriented edges and
$\partialorientation$--unoriented edges.
The refined $\decmap$ is still $\starmap(\modelspace)$--invariant.
\item Consider the $\partialorientation$--unoriented, unbalanced
  cylinders. It is possible to choose orientations of their vertex
  spaces so that if $c$ and $c'$ are two such cylinders with
  $\imbalance_c^{\decmap,\partialorientation}=\imbalance_{c'}^{\decmap,\partialorientation}$
  then for all $\ornament\in\ornaments$ we have
  $\sum_{e\in\link(c)\cap\decmap^{-1}(\ornament)}\sign\attachmap_e=\sum_{e\in\link(c')\cap\decmap^{-1}(\ornament)}\sign\attachmap_e$. 
Extend $\partialorientation$ to $\partialorientation'$ by taking these
orientations of the unbalanced cylindrical vertex spaces.
\item If $e$ is a $\partialorientation$--unoriented edge such that $c:=\initial(e)$ is cylindrical and $\partialorientation'$--oriented,
define
$\partialorientation'(\modelspace_{\bar{e}})=\attachmap_e(\partialorientation'(\modelspace_c))$.
\item Define $\ornaments':=\ornaments\times\{-1,0,1\}$.
Define $\decmap'(e):=(\decmap(e),\sign_{\partialorientation'}(e))$ for
each edge and $\decmap'(v):=(\decmap(v),0)$ for each vertex.
\end{enumerate}

\begin{lemma}\label{lemma:cylinderrefinement}
Suppose that $\decmap$ and $\partialorientation$ are both
$\starmap(\modelspace)$--invariant and 
 that the $\decmap$--partition of edges of $\tree$ is finer than
the partition into $\partialorientation$--oriented edges and
$\partialorientation$--unoriented edges.
Let $\decmap'$ and
$\partialorientation'$ be constructed via cylinder refinement, as
above. Then
  $\decmap'$ is a $\starmap(\modelspace)$--invariant refinement of $\decmap$ and 
  $\partialorientation'$ is a $\starmap(\modelspace)$--invariant
  extension of $\partialorientation$.
Moreover, the $\decmap'$--partition on edges is finer than the partition into $\partialorientation'$--oriented and
$\partialorientation'$--unoriented edges.
\end{lemma}
\begin{proof}
Suppose $e$ is a $\partialorientation$--unoriented edge with
$\initial(e):=c$ cylindrical and unbalanced.
Then $e$ is $\partialorientation'$--oriented and
$\sign_{\partialorientation'}\attachmap_e=1$, so
$\decmap'(e)=(\decmap(e),1)$.

By invariance of $\decmap$ and $\partialorientation$, \fullref{prop:imbalancewrtqi}, and our choice of orientation on
$\modelspace_c$ and $\modelspace_{\qi_*(c)}$, if
$\qi\in\starmap(\modelspace)$ 
then $\qi_*(c)$ is unbalanced, and $\qi_*(e)$ is
$\partialorientation$--unoriented but $\partialorientation'$ oriented
with $\sign_{\partialorientation'}\attachmap_{\qi_*(e')}=1$.
Moreover, $\qi_c$ is orientation preserving, so
$\qi(\partialorientation'(\modelspace_{\bar{e}}))=\partialorientation'(\modelspace_{\overline{\qi_*(e)}})$.
It also means that $\decmap'(\qi_*(e))=(\decmap(\qi_*(e)),1)=(\decmap(e),1)=\decmap'(e)$.

Now suppose $e$ is $\partialorientation$--oriented and $\initial(e):=c$
is cylindrical and unbalanced.
Invariance of $\partialorientation'$ on $e$ is inherited from
invariance of $\partialorientation$.
From the proof of \fullref{prop:imbalancewrtqi}, since $\qi_c$ is
orientation preserving, $\sign_{\partialorientation'}
\attachmap_e=\sign_{\partialorientation'}\attachmap_{\qi_*(e)}$.
Along with invariance of $\decmap$, this gives us
$\decmap'(e)=\decmap'(\qi_*(e))$.

For vertices and remaining edges, $\decmap'(t)=(\decmap(t),0)=(\decmap(\qi_*(t)),0)=\decmap'(\qi_*(t))$.

For the final claim, suppose $e$ is $\partialorientation'$--oriented
and $e'$ is $\partialorientation'$--unoriented.
Since $\partialorientation'$ extends $\partialorientation$, $e'$ is
also $\partialorientation$--unoriented.
If $e$ is $\partialorientation$--oriented then
$\decmap(e)\neq\decmap(e')$  because the $\decmap$--partition of edges of $\tree$ is finer than
the partition into $\partialorientation$--oriented edges and
$\partialorientation$--unoriented edges.
Thus, $\decmap'(e)\neq\decmap'(e')$, since
$\decmap'$ refines $\decmap$.

If $e$ is $\partialorientation$--unoriented then
$\decmap'(e)=(\decmap(e),\sign_{\partialorientation'}\attachmap_e)$ and
$\decmap'(e')=(\decmap(e'),0)$ differ in the second coordinate.
\end{proof}

\begin{lemma}\label{lemma:nozerosinimbalance}
  Suppose that $\decmap$ and $\partialorientation$ are
  $\starmap(\modelspace)$--invariant and stable under
  cylinder refinement. 
If $c$ is an unbalanced cylindrical vertex and
$\ornament\in\ornaments$ such that 
$\decmap^{-1}(\ornament)\cap\link(c)\neq\emptyset$ then either every
edge in $\decmap^{-1}(\ornament)\cap\link(c)$ has orientation
preserving attaching map or every
edge in $\decmap^{-1}(\ornament)\cap\link(c)$ has orientation
reversing attaching map.
\end{lemma}
\begin{proof}
  Cylinder refinement orients every edge in an unbalanced cylinder and
  distinguishes edges with orientation preserving attaching map from
  those with orientation reversing attaching map.
\end{proof}

\subsection{Non-elementary vertices}
Given a tree of spaces whose underlying tree is decorated
$\decmap\from\tree\to\ornaments$, we get a
decoration on the peripheral structure of each vertex space via
mapping to the tree and composing with $\decmap$.

Throughout this subsection we assume that 
   $\decmap\from\tree\to\ornaments$ is a
   $\starmap(\modelspace)$--invariant decoration and
   $\partialorientation$ is a  $\starmap(\modelspace)$--invariant partial orientation.

Define:
\[\ornaments'=\ornaments\times\!\!\!\coprod_{Q\in\mathrm{Maptypes}}
  \lquotient{\starmap(Z_Q,\per_Q)}{(\ornaments^{\per_Q}\times
    (\cup_{x\in\bdry\per_{Q}}x \amalg\{\nullvar\})^{\per_Q}\times
    (\per_Q\amalg \{\nullvar\}))}\]

The left action of $\starmap(Z_Q,\per_Q)$ is
given by $\qi.(\iso,\partialorientation,e):=(\iso\circ\qi^{-1},\partialorientation\circ\qi^{-1},
\qi_*(e))$.

If $e\in\tree$ is an edge with non-elementary terminus
$v:=\terminal(e)$, $Q=\llbracket
(\modelspace_v,\per_v)\rrbracket$,
$\modelmap_v\in\starmap((\modelspace_v,\per_v),(Z_Q,\per_Q))$, and
$\partialorientation_v$ is a partial orientation on $\per_v$, define:
\[\decmap'(v):=\left(\decmap(v),\starmap(Z_Q,\per_Q).(\decmap|_{\per_v}\circ\modelmap_v^{-1},\partialorientation_v\circ\modelmap_v^{-1},\texttt{`NULL'})\right)\]
\[\decmap'(e):=\left(\decmap(e),\starmap(Z_Q,\per_Q).(\decmap|_{\per_v}\circ\modelmap_v^{-1},\partialorientation_v\circ\modelmap_v^{-1}
  ,(\modelmap_v)_*(e))\right)\]
Note that the image is independent of the 
 choice of $\modelmap_v\in\starmap((\modelspace_v,\per_v),(Z_Q,\per_Q))$. 

Composition of $\decmap'$ with
projection to the first factor of $\ornaments'$ recovers $\decmap$, so
$\decmap'$ is a refinement of $\decmap$.

\begin{proposition}\label{prop:redecoartionisqiinvariant}
The refinement $\decmap'$ of $\decmap$ defined above is $\starmap(\modelspace)$--invariant.
\end{proposition}
\begin{proof}
Take $\qi\in\starmap(\modelspace)$.
If $e$ is an edge with $v:=\terminal(e)$ non-elementary,
$Q:=\llbracket\modelspace_v,\per_v\rrbracket$, and $\iso:=\modelmap_{\qi_*(v)}\circ
  \qi_v\circ\modelmap_v^{-1}\in\starmap(Z_Q,\per_Q)$, then:
\begin{align*}
\decmap'(\qi_*(e))&=\left(\decmap(\qi_*(e)),\starmap(Z_Q,\per_Q).(\decmap|_{\per_{\qi_*(v)}}\circ\modelmap_{\qi_*(v)}^{-1},\partialorientation_{{\qi_*(v)}}\circ\modelmap_{\qi_*(v)}^{-1},(\modelmap_{\qi_*(v)})_*(\qi_*(e)))\right)\\
                   &=\left(\decmap(e),\starmap(Z_Q,\per_Q).(\decmap|_{\per_{v}}\circ\qi_{v}^{-1}\circ\modelmap_{\qi_*(v)}^{-1},\partialorientation_{{v}}\circ\qi_{v}^{-1}\circ\modelmap_{\qi_*(v)}^{-1},(\modelmap_{\qi_*(v)})_*(\qi_*(e)))\right)\\
                   &=\left(\decmap(e),\starmap(Z_Q,\per_Q).(\decmap|_{\per_{v}}\circ\modelmap_{v}^{-1}\circ\iso^{-1},\partialorientation_{{v}}\circ\modelmap_{v}^{-1}\circ\iso^{-1},(\iso\circ\modelmap_{v})_*(e))\right)\\
                   &=\left(\decmap(e),\starmap(Z_Q,\per_Q).(\decmap|_{\per_{v}}\circ\modelmap_{v}^{-1},\partialorientation_{{v}}\circ\modelmap_{v}^{-1},(\modelmap_{v})_*(e))\right)\\
&=\decmap'(e)
\end{align*}
Thus, $\decmap'$ is $\starmap(\modelspace)$--invariant.
\end{proof}

\begin{proposition}\label{prop:redecorationvslocalqi}
For vertices $v,w\in\tree$,  $\decmap'(v)=\decmap'(w)$ if and only if
there exists
$\qi\in\starmap((\modelspace_v,\per_v,\decmap,\partialorientation),(\modelspace_w,\per_w,\decmap,\partialorientation))$.

For edges $e,f\in\tree$ with $v:=\terminal(e)$ and $w:=\terminal(f)$
both non-elementary, $\decmap'(e)=\decmap'(f)$ if and only if there
exists
$\qi\in\starmap((\modelspace_v,\per_v,\decmap,\partialorientation),(\modelspace_w,\per_w,\decmap,\partialorientation))$
with $\qi_*(e)=f$.
\end{proposition}
\begin{proof}
  We give the proof for edges. The proof for vertices is similar.

Let $Q=\llbracket \modelspace_v,\per_v\rrbracket$.
By definition, $\decmap'(e)=\decmap'(f)$ if and only if
$\decmap(e)=\decmap(f)$ and there exists $\iso\in\starmap(Z_Q,\per_Q)$
such that 
\begin{enumerate}
\item
  $\decmap|_{\per_v}\circ\modelmap_v^{-1}\circ\iso^{-1}=\decmap|_{\per_w}\circ\modelmap_w^{-1}$\label{item:actiondec}
\item $\partialorientation_{v}\circ\modelmap_v^{-1}\circ\iso^{-1}=\partialorientation_{w}\circ\modelmap_w^{-1}$\label{item:actionor}
\item $(\iso\circ\modelmap_v)_*(e)=(\modelmap_w)_*(f)$\label{item:invariance}
\end{enumerate}

Define $\qi:=\modelmap_w^{-1}\circ\iso\circ\modelmap_v\in
\starmap((\modelspace_v,\per_v),(\modelspace_w,\per_w))$.
Item \ref{item:actiondec} is equivalent to
$\decmap|_{\per_v}\circ\qi^{-1}=\decmap|_{\per_w}$.
Item \ref{item:actionor} is equivalent to
$\partialorientation_{v}\circ\qi^{-1}=\partialorientation_{w}$.
Item \ref{item:invariance} is equivalent to $\qi_*(e)=f$.
\end{proof}

\begin{corollary}\label{lemma:reversingmap}
 There exists a $\starmap(\modelspace)$--invariant extension
 $\partialorientation'$ of $\partialorientation$ such that for any
 edge $e$ with $v:=\terminal(e)$ non-elementary, $e$ is
 $\partialorientation'$--unoriented if and only if the stabilizer of
 $\modelspace_{\bar{e}}$ in
 $\starmap(\modelspace_v,\per_v,\decmap,\partialorientation)$ contains
 an infinite dihedral group. 
\end{corollary}
\begin{proof}
 If $e$
$\partialorientation$--unoriented and the stabilizer of
$\modelspace_{\bar{e}}$ in
$\starmap(\modelspace_v,\per_v,\decmap,\partialorientation)$ does not
contain an infinite dihedral group then
define an extension $\partialorientation'$ of $\partialorientation$ on
$(\decmap')^{-1}(\decmap'(e))$ as follows.
Choose an orientation of $\modelspace_{\bar{e}}$.
If $f$ is an edge with $\decmap'(f)=\decmap'(e)$ then, by
\fullref{prop:redecorationvslocalqi}, there exists $\qi\in\starmap((\modelspace_v,\per_v,\decmap,\partialorientation),(\modelspace_w,\per_w,\decmap,\partialorientation))$
with $\qi_*(e)=f$. 
This means that $f$ is $\partialorientation$--unoriented, so we extend
$\partialorientation$ by defining
$\partialorientation'(\modelspace_{\bar{f}}):=\qi(\partialorientation'(\modelspace_{\bar{e}}))$.

The orientation of $\modelspace_{\bar{f}}$ is independent of the
choice of $\qi$ because of the stabilizer condition on
$\modelspace_{\bar{e}}$.
\end{proof}

\begin{definition}
  Given $\starmap(\modelspace)$--invariant decoration $\decmap$ and
  partial orientation $\partialorientation$, the process of
  \emph{vertex refinement} produces the
  $\starmap(\modelspace)$--invariant $\decmap'$ and partial
  orientation $\partialorientation'$ defined above.
\end{definition}

\subsection{Combining the local restrictions}
In this section we have our main technical
tools. \fullref{thm:withvertexrestrictions} identifies
$\Aut(\tree,\decmap)$ orbits. 
\fullref{thm:starmapmain} leverages this information to understand
decoration preserving isomorphisms between two different trees.
\fullref{thm:starmapmain} provides a blueprint for the main
classification theorems in the next two sections.

\begin{theorem}\label{thm:withvertexrestrictions}
Suppose $\decmap\from\tree\to\ornaments$ is a
$\starmap(\modelspace)$-invariant decoration and $\partialorientation$ is
a $\starmap(\modelspace)$-invariant partial orientation.
Suppose $\decmap$ and $\partialorientation$ are stable under neighbor, cylinder, and vertex refinement.

For edges $e,f\in\tree$ we have 
${\decmap}(e)={\decmap}(f)$  
  if and only if there exists 
$\iso\in\Aut(\tree,\decmap)$ such that 
\begin{itemize}
\item $\iso(e)=f$
\item For every $u\in\tree$ there exists 
 $\qi_u\in\starmap((\modelspace_u,\per_u,\decmap,\partialorientation),(\modelspace_{\iso(u)},\per_{\iso(u)},\decmap,\partialorientation))$,
 such that $\iso|_{\link(u)}=(\qi_u)_*$.
\item For every edge $e'$ the map $(\attachmap_{\iso(e')}\circ\qi_{\initial(e')})\circ(\qi_{\terminal(e')}\circ\attachmap_{e'})^{-1}$
is orientation preserving on $\modelspace_{\overline{\iso(e')}}$.
\end{itemize}
\end{theorem}
\begin{proof}
If there exists $\iso\in\Aut(\tree,\decmap)$ such that 
$\iso(e)=f$ then $\decmap(e)=\decmap(f)$.
Conversely, supposing $\decmap(e)=\decmap(f)$, we construct $\iso$.

Define $\iso(e):=f$.

By \fullref{prop:redecorationvslocalqi}, there exists
$\qi_{\terminal(e)}\in\starmap((\modelspace_{\terminal(e)},\per_{\terminal(e)},\decmap,\partialorientation),(\modelspace_{\terminal(f)},\per_{\terminal(f)},\decmap,\partialorientation))$
with $(\qi_{\terminal(e)})_*(e)=f$.
Define $\iso|_{\link(\terminal(e))}:=(\qi_{\terminal(e)})_*$.

Now, suppose that we have $\iso$ satisfying the desired properties
defined on a subtree $\tree'$ of $\tree$ such that every leaf is
non-elementary and $\tree'$ contains every edge incident to every
non-leaf.
Given an edge $e_0$ with $c:=\initial(e_0)\notin\tree'$ and
$\terminal(e_0)\in\tree'$, we
show how to extend $\iso$ to $\link(c_0)$, satisfying the desired
properties. 
Then, by induction, we can extend $\iso$ to all of $\tree$.

Let $\iso(e_0):=(\qi_{\terminal(e_0)})_*(e_0)$.
Define $\qi_{c}:=\attachmap_{\iso(e_0)}^{-1}\circ\qi_{\terminal(e_0)}\circ\attachmap_{e_0}$ so that $(\attachmap_{\iso(e_0)}\circ\qi_{c})\circ(\qi_{\terminal(e_0)}\circ\attachmap_{e_0})^{-1}$
is orientation preserving on $\modelspace_{\overline{\iso(e_0)}}$.

\paragraph{Case 1: $c$ is unbalanced.}
Extend $\iso$ to $\link(c)$ by choosing a bijection between
$\link(c)\setminus\{e_0\}\cap\decmap^{-1}(\ornament)$ and
$\link(\iso(c))\setminus\{\iso(e_0)\}\cap\decmap^{-1}(\ornament)$ for each
$\ornament\in\ornaments$.
For each $\ornament$ these sets have the same cardinality by neighbor stability.
Since $c$ is unbalanced, 
cylindrical stability implies that $c$ and
all edges in $\link(c)$ are $\partialorientation$--oriented, and, for
each $\ornament\in\ornaments$, all edges in $\decmap^{-1}(\ornament)$ have attaching maps with
the same sign; recall \fullref{lemma:nozerosinimbalance}.
By $\starmap(\modelspace)$--invariance, the same is true for
$\iso(c)$, and for each $e_1\in\link(c)$ we have $\sign_\partialorientation\attachmap_{e_1}=\sign_\partialorientation\attachmap_{\iso(e_1)}$.

By \fullref{prop:redecorationvslocalqi}, there exists  
\[\qi_{\terminal(e_1)}\in\starmap((\modelspace_{\terminal(e_1)},\per_{\terminal(e_1)},\decmap,\partialorientation),(\modelspace_{\terminal(\iso(e_1))},\per_{\terminal(\iso(e_1))},\decmap,\partialorientation))\]
with $(\qi_{\terminal(e_1)})_*(e_1)=\iso(e_1)$.
Define $\qi|_{\modelspace_{\terminal(e_1)}}:=\qi_{\terminal(e_1)}$ and  $\iso|_{\link(\terminal(e_1))}:=(\qi_{\terminal(e_1)})_*$.
By construction, $(\attachmap_{\iso(e_1)}\circ\qi_{c})\circ(\qi_{\terminal(e_1)}\circ\attachmap_{e_1})^{-1}$
is orientation preserving on $\modelspace_{\overline{\iso(e_1)}}$.

\medskip
In the balanced cases, choose some orientation of $\modelspace_c$
and $\modelspace_{\iso(c)}$.
\paragraph{Case 2: $c$ is balanced and $\ornament\in\ornaments$ is such that
$\decmap^{-1}(\ornament)\cap\link(c)\neq \emptyset$ consists of
$\partialorientation$--oriented edges.}
By neighbor stability, the total number, $n$, of edges in
$\link(c)\cap\decmap^{-1}(\ornament)$ is equal to the total number of
edges in $\link(\iso(c))\cap\decmap^{-1}(\ornament)$.
Since $c$ is balanced, the number of edges in
$\link(c)\cap\decmap^{-1}(\ornament)$ with orientation preserving
attaching map is equal to the number of edges in
$\link(c)\cap\decmap^{-1}(\ornament)$ with orientation reversing
attaching map, so there are $n/2$ of each.
Cylinder stability implies $\iso(c)$ is also balanced, so there are $n/2$
edges in
$\link(\iso(c))\cap\decmap^{-1}(\ornament)$ with orientation preserving
attaching map and $n/2$ with orientation reversing attaching map.

If $\sign\attachmap_{e_0}=\sign\attachmap_{\iso(e_0)}$ then
 $\qi_c$ is orientation preserving. 
Define $\iso$ on
$\decmap^{-1}(\ornament)\cap\link(c)\setminus\{e_0\}$ by choosing any bijection with
$\decmap^{-1}(\ornament)\cap\link(\iso(c))\setminus\{e_0\}$  that
preserves the signs of the attaching maps. 

If $\sign\attachmap_{e_0}\neq\sign\attachmap_{\iso(e_0)}$
then $\qi_c$ is orientation reversing.
Define $\iso$ on
$\decmap^{-1}(\ornament)\cap\link(c)\setminus\{e_0\}$ by choosing any
bijection with
$\decmap^{-1}(\ornament)\cap\link(\iso(c))\setminus\{e_0\}$ that
exchanges the signs of the attaching maps.

Extend $\qi$ and $\iso$ as in the previous case.

\paragraph{Case 3: $c$ is balanced and
 $\ornament\in\ornaments$ is such that $\decmap^{-1}(\ornament)\cap\link(c)\neq \emptyset$ consists of
$\partialorientation$--unoriented edges.}
By neighbor stability, $\link(c)\setminus\{e_0\}\cap\decmap^{-1}(\ornament)$ and
$\link(\iso(c))\setminus\{e_0\}\cap\decmap^{-1}(\ornament)$ have the same cardinality, and
we extend $\iso$ by an arbitrary bijection between them. 
Take $e_1\in \decmap^{-1}(\ornament)\cap\link(c)\setminus\{e_0\}$.
By \fullref{prop:redecorationvslocalqi}, there exists 
$\qi'_{\terminal(e_1)}\in\starmap((\modelspace_{\terminal(e_1)},\per_{\terminal(e_1)},\decmap,\partialorientation),(\modelspace_{\terminal(\iso(e_1))},\per_{\terminal(\iso(e_1))},\decmap,\partialorientation))$
with $(\qi'_{\terminal(e_1)})_*(e_1)=\iso(e_1)$.
Since $e_1$ is $\partialorientation$--unoriented and
$\partialorientation$ is stable under vertex refinement, by \fullref{lemma:reversingmap} there exists an
element of
$\starmap(\modelspace_{\terminal(e_1)},\per_{\terminal(e_1)},\decmap,\partialorientation)$
reversing $\modelspace_{\bar{e}_1}$.
Define $\qi_{\terminal(e_1)}:=\qi'_{\terminal(e_1)}$ if 
 $(\attachmap_{\iso(e_1)}\circ\qi_{c})\circ(\qi'_{\terminal(e_1)}\circ\attachmap_{e_1})^{-1}$
is orientation preserving on $\modelspace_{\overline{\iso(e_1)}}$, and define
$\qi_{\terminal(e_1)}$ to be $\qi'_{\terminal(e_1)}$ precomposed
with a $\modelspace_{\terminal(e_1)}$--flip otherwise.
Extend $\iso$ to $\link(\terminal(e_1))$ by $(\qi_{\terminal(e_1)})_*$.
\end{proof}

Let $\partialorientation_0$ be the trivial partial orientation on
$\modelspace$ with constant value $\nullvar$.
Let $\decmap_0\from\tree\to\ornaments_0$ be any $\starmap(\modelspace)$--invariant decoration of $\tree$.
Perform neighbor, cylinder, and vertex refinement repeatedly until all
three stabilize, and let $\decmap\from\tree\to\ornaments$ be the
resulting decoration and $\partialorientation$ the resulting partial
orientation.

Now suppose $\modelspace'$ is a tree of spaces over $\tree'$ with
finite cylinders and such that every $\qi\in\starmap(\modelspace')$
splits as a tree of maps over $\tree'$.
Let $\partialorientation'_0$ be the trivial partial orientation, and let $\decmap_0'\from\tree'\to\ornaments_0$ be a
$\starmap(\modelspace')$--invariant decoration of $\tree'$. (Note that
$\decmap_0$ and $\decmap'_0$ map to the same set of ornaments!)
Let $\partialorientation'$
and $\decmap'$ be the partial orientation extending
$\partialorientation'_0$ and the decoration refining $\decmap_0'$ that
result from performing neighbor, cylinder, and vertex refinement repeatedly until all
three stabilize.

Recall that the process of cylinder refinement involved choosing 
$\starmap(\modelspace)$--invariant orientations.
We will need to account for the fact that these choices can be made
differently in $\modelspace$ and $\modelspace'$.
Let $\xi\in\{-1,1\}^\ornaments$.
Define $\xi\cdot\partialorientation$ to be the partial orientation:
\[\xi\cdot\partialorientation(\modelspace_t)=
\begin{cases}
  \nullvar & \text{ if }\partialorientation(\modelspace_t)=\nullvar\\
\partialorientation(\modelspace_t) & \text{ if } \xi(\decmap(t))=1\\
\text{opposite of }\partialorientation(\modelspace_t) & \text{ if } \xi(\decmap(t))=-1\\
\end{cases}
\]

\begin{theorem}\label{thm:starmapmain}
With the above notation, the following are equivalent:
  \begin{enumerate}
  \item There exists 
$\iso\in\Isomgp((\tree,\decmap_0),(\tree',\decmap_0'))$ such that:\label{item:treeofmaps}
\begin{enumerate}
\item For every vertex $v\in\tree$ there exists 
 $\qi_v\in\starmap((\modelspace_v,\per_v),(\modelspace'_{\iso(v)},\per'_{\iso(v)}))$,
 such that $\iso|_{\link(v)}=(\qi_v)_*$.
\item For every edge $e\in\tree$ we have
$(\attachmap_{\iso(e)}\circ\qi_{\initial(e)})\circ(\qi_{\terminal(e)}\circ\attachmap_{e})^{-1}$
is orientation preserving on $\modelspace'_{\overline{\iso(e)}}$.
\end{enumerate}
  \item
There exists a bijection $\beta\from\image
  \decmap\to\image\decmap'$ and $\xi\in\{-1,1\}^{\ornaments}$ such that:\label{item:samestructureinvariants}
\begin{enumerate}
\item $\decmap_0\circ\decmap^{-1}=\decmap_0'\circ(\decmap')^{-1}\circ\beta$\label{item:compatiblewithinitialdecoration}
\item When the rows and columns of
  $\struc(\tree',\decmap',\ornaments')$ are given the 
  $\beta$--induced ordering from $\struc(\tree,\decmap,\ornaments)$,
  we have
  $\struc(\tree,\decmap,\ornaments)=\struc(\tree',\decmap',\ornaments')$.\label{item:samestructure}
\item For
  every $\ornament\in\image \decmap$ such that
  $\decmap^{-1}(\ornament)$ consists of non-elementary vertices there exists (equivalently, for
  every) $v\in\decmap^{-1}(\ornament)$ and
  $v'\in(\decmap')^{-1}(\beta(\ornament))$ so that $\starmap((\modelspace_{v},\per_{v},\beta\circ\decmap,\xi\cdot\partialorientation),(\modelspace'_{v'},\per'_{v'},\decmap',\partialorientation')
)$ is nonempty.\label{item:realizable}
\item For
  every $\ornament\in\image \decmap$ such that
  $\decmap^{-1}(\ornament)$ consists of cylindrical vertices there exists (equivalently, for
  every) $c\in\decmap^{-1}(\ornament)$ and
  $\iso(c)\in(\decmap')^{-1}(\beta(\ornament))$ so that $\imbalance_c^{\decmap,\xi\cdot\partialorientation}=\imbalance_{\iso(c)}^{\decmap',\partialorientation'}\circ\beta$.\label{item:sameimbalance}
\end{enumerate}  
\end{enumerate}
\end{theorem}
\begin{proof}
If item \ref{item:treeofmaps} is true then
$\decmap_0=\decmap'_0\circ\iso$ and
$\partialorientation_0=\partialorientation'_0\circ\iso$.
Perform the same sequence of refinements on $\decmap_0$ and
$\decmap'_0$.
Each time the partial orientation on $\modelspace$ is extended by
choosing some orientation, push that choice forward to $\modelspace'$
using the appropriate $\qi_c$ or $\qi_v$.
We get the claims of  item
\ref{item:samestructureinvariants} with $\beta$ the identity and $\xi$
the constant map sending $\ornaments$ to 1.

We complete the proof by showing that the hypotheses of item
\ref{item:samestructureinvariants} allow us to build a isomorphism
$\iso\in\Isomgp((\tree,\beta\circ\decmap),(\tree',\decmap'))$ and a collection
of maps $\qi_v$ satisfying the conditions
of item \ref{item:treeofmaps}.
Condition \ref{item:compatiblewithinitialdecoration} implies $\iso\in\Isom(\tree,\decmap_0),(\tree',\decmap_0'))$.
The construction is along the lines of that in the proof of
\fullref{thm:withvertexrestrictions}: we inductively construct
$\iso$
and maps
$\qi_v\in\starmap((\modelspace_v,\per_v,\beta\circ\decmap,\xi\cdot\partialorientation),(\modelspace'_{\iso(v)},\per'_{\iso(v)},\decmap',\partialorientation))$
with $\iso|_{\link(v)}=(\qi_v)_*$.

To begin, pick a non-elementary vertex $v_0\in\tree$ and a vertex
$v_0'\in(\decmap')^{-1}(\beta\circ\decmap(v_0))$.
Define $\iso(v_0):=v'_0$.
By \ref{item:realizable} and \fullref{thm:withvertexrestrictions},
there exists:
\[\qi_{v_0}\in\starmap((\modelspace_{v_0},\per_{v_0},\beta\circ\decmap,\xi\cdot\partialorientation),(\modelspace'_{v_0'},\per'_{v'_0},\decmap',\partialorientation'))\]
Define $\qi|_{\modelspace_{v_0}}:=\qi_{{v_0}}$
and $\iso|_{\link(v_0)}:=(\qi_{{v_0}})_*$.

Let $e_0$ be an edge in $\link(v_0)$ with cylindrical initial vertex
$c:=\initial(e_0)$.
Define
$\qi_c:=(\attachmap_{e_0'})^{-1}\circ\qi_{v_0}\circ\attachmap_{e_0}$.

For the induction step we extend $\iso$ to $\link(c)$.
When $c$ is balanced the construction is virtually the same as that of
\fullref{thm:withvertexrestrictions}, so we omit those cases.
The remaining case is that $c$ is unbalanced.

Extend $\iso$ to $\link(c)$ by choosing a bijection between
$\link(c)\setminus\{e_0\}\cap\decmap^{-1}(\ornament)$ and
$\link(\iso(c))\setminus\{\iso(e_0)\}\cap\decmap^{-1}(\beta(\ornament))$ for each
$\ornament\in\ornaments$.
These sets have the same cardinality by condition \ref{item:samestructure}.
Since $c$ is unbalanced, 
cylindrical stability implies that $c$ and
all edges in $\link(c)$ are $\partialorientation$--oriented, and, for
each $\ornament\in\ornaments$, all edges in $\decmap^{-1}(\ornament)$ have attaching maps with
the same sign; recall \fullref{lemma:nozerosinimbalance}.
This implies that for each $\ornament\in\ornaments$,
$\imbalance_c^{\decmap,\xi\cdot\partialorientation}(\ornament)=\pm
\imbalance_c^{\decmap,\partialorientation}(\ornament)$, so, in
particular $\imbalance_c^{\decmap,\xi\cdot\partialorientation}$ is not
identically zero.
Condition \ref{item:sameimbalance} then implies $\iso(c)$ is
unbalanced, so $\iso(c)$ and all of the edges in $\link(\iso(c))$ are
$\partialorientation'$--oriented, and edges with the same ornament
have attaching maps with the same sign.

We may choose the orientations on $\modelspace_c$ and
$\modelspace'_{\iso(c)}$ to be those given by
$\xi\cdot\partialorientation$ and $\partialorientation'$, respectively.
Together with condition \ref{item:sameimbalance}, this implies there exists $\epsilon\in\pm 1$ such that for all
$\ornament\in\ornaments$:
\[\sum_{e\in\link(c)\cap\decmap^{-1}(\ornament)}\sign_{\xi\cdot\partialorientation}\attachmap_e=\epsilon\left(\sum_{e'\in\link(\iso(c))\cap(\decmap')^{-1}(\beta(\ornament))}\sign_{\partialorientation'}\attachmap'_{e'}\right)\]
Since all the edges with a particular ornament have attaching maps of
the same sign, this means that for all $e_1\in\link(c)$ we have $\sign_{\xi\cdot\partialorientation}\attachmap_{e_1}=\epsilon\sign_{\partialorientation'}\attachmap'_{\iso(e_1)}$.
Therefore, the sign of
\[\attachmap'_{\iso(e_1)}\circ\qi_c\circ\attachmap_{e_1}^{-1}=\attachmap'_{\iso(e_1)}\circ (\attachmap'_{\iso(e_0)})^{-1}\circ\qi_{v_0}\circ\attachmap_{e_0}\circ\attachmap_{e_1}^{-1}\] on
$\modelspace_{\bar{e}_1}$ with respect to
$\xi\cdot\partialorientation$ and $\partialorientation'$ is
$(\epsilon\cdot\sign_{\xi\cdot\partialorientation}\attachmap_{e_0}\cdot\sign_{\xi\cdot\partialorientation}\attachmap_{e_1})^2=+1$.

By \fullref{prop:redecorationvslocalqi} and condition \ref{item:realizable}, there exists 
\[\qi_{\terminal(e_1)}\in\starmap((\modelspace_{\terminal(e_1)},\per_{\terminal(e_1)},\beta\circ\decmap,\xi\cdot\partialorientation),(\modelspace'_{\terminal(\iso(e_1))},\per'_{\terminal(\iso(e_1))},\decmap',\partialorientation'))\]
with $(\qi_{\terminal(e_1)})_*(e_1)=\iso(e_1)$.
Define $\qi|_{\modelspace_{\terminal(e_1)}}:=\qi_{\terminal(e_1)}$ and  $\iso|_{\link(\terminal(e_1))}:=(\qi_{\terminal(e_1)})_*$.
We know $(\attachmap_{\iso(e_1)}\circ\qi_{c})\circ(\qi_{\terminal(e_1)}\circ\attachmap_{e_1})^{-1}$
is orientation preserving on $\modelspace'_{\overline{\iso(e_1)}}$
because:
\[\attachmap'_{\iso(e_1)}\circ\qi_c\circ\attachmap_{e_1}^{-1}(\xi\cdot\partialorientation(\modelspace_{\bar{e}_1}))=\partialorientation'(\modelspace'_{\overline{\iso(e_1)}})=\qi_{\terminal(e_1)}(\xi\cdot\partialorientation(\modelspace_{\bar{e}_1}))\]

We remark that it is not required that
$\qi_c(\xi\cdot\partialorientation(\modelspace_c))=\partialorientation'(\modelspace'_{\iso(c)})$,
but this can easily be arranged by redefining $\xi(\decmap(c))$ to be $\epsilon\cdot\xi(\decmap(c))$.
\end{proof}

\section{Classification of hyperbolic groups up to boundary homeomorphism from their two-ended JSJ splittings}
We are now ready to prove our first classification theorem, characterizing the homeomorphism type of the Gromov boundary of a one-ended hyperbolic group 
from its JSJ tree of cylinders.

\begin{theorem}\label{corollary:boundaryhomeomorphism}
Let $G$ be a one-ended hyperbolic group with non-trivial JSJ
decomposition over two-ended subgroups, with $\tree:=\cyl(G)$.
Let $\modelspace$ be an algebraic tree of spaces for $G$ over $\tree$.
Let $\partialorientation_0$ be the trivial partial orientation on
$\modelspace$.
Take the initial decoration $\decmap_0$ on $\tree$ to be by vertex
type (`cylindrical', `rigid', or `hanging') and relative boundary
homeomorphism type.
Perform neighbor, cylinder, and vertex refinement until all three
stabilize to give a decoration $\decmap\from\tree\to\ornaments$ and a
partial orientation $\partialorientation$ of $\modelspace$.

Let $G'$ be another one-ended hyperbolic group with non-trivial JSJ
decomposition over two-ended subgroups. Define $\tree'$,
$\modelspace'$, $\decmap'_0$, $\partialorientation'_0$,
$\decmap'\from\tree'\to\ornaments'$, and $\partialorientation'$ as we
did for $G$.
Then $\bdry G$ is homeomorphic to $\bdry G'$ if and only if
there exists a bijection $\beta\from\image
  \decmap\to\image\decmap'$ and a $\xi\in\{-1,1\}^{\ornaments}$ such that:
\begin{enumerate}
\item $\decmap_0\circ\decmap^{-1}=\decmap_0'\circ(\decmap')^{-1}\circ\beta$
\item When the rows and columns of
  $\struc(\tree',\decmap',\ornaments')$ are given the 
  $\beta$--induced ordering from $\struc(\tree,\decmap,\ornaments)$,
  we have
  $\struc(\tree,\decmap,\ornaments)=\struc(\tree',\decmap',\ornaments')$.
\item For
  every $\ornament\in\image \decmap$ such that
  $\decmap^{-1}(\ornament)$ consists of non-elementary vertices there exists (equivalently, for
  every) $v\in\decmap^{-1}(\ornament)$ and $v'\in(\decmap')^{-1}(\beta(\ornament))$ so that 
\[\Homeo((\bdry\modelspace_{v},\bdry\per_{v},\beta\circ\decmap,\xi\cdot\partialorientation),(\bdry\modelspace'_{v'},\bdry\per'_{v'},\decmap',\partialorientation')
)\] is nonempty.
\item For
  every $\ornament\in\image \decmap$ such that
  $\decmap^{-1}(\ornament)$ consists of cylindrical vertices there exists (equivalently, for
  every) $c\in\decmap^{-1}(\ornament)$ and
  $c'\in(\decmap')^{-1}(\beta(\ornament))$ such that
  $\imbalance_c^{\decmap,\xi\cdot\partialorientation}=\imbalance_{c'}^{\decmap',\partialorientation'}\circ\beta$.
\end{enumerate}  
\end{theorem}
\begin{proof}
Since the initial decorations $\decmap_0$ and $\decmap_0'$ are trivial, the given conditions are equivalent, by \fullref{thm:starmapmain} for boundary homeomorphism, to the existence of $\iso\in\Isomgp(\tree,\tree')$ such that:
\begin{enumerate}
\item For every vertex $v\in\tree$ there exists 
 $\qi_v\in\Homeo((\bdry\modelspace_v,\bdry\per_v),(\bdry\modelspace'_{\iso(v)},\bdry\per'_{\iso(v)}))$,
 such that $\iso|_{\link(v)}=(\qi_v)_*$.
\item For every edge $e\in\tree$ we have
$\bdry\attachmap_{\iso(e)}\circ\qi_{\initial(e)}=\qi_{\terminal(e)}\circ\bdry\attachmap_{e}$.
\end{enumerate}

These conditions say there exists an isomorphism
$\iso\from\tree\to\tree'$ and a tree of boundary homeomorphisms over
$\iso$ compatible with $\modelspace$ and $\modelspace'$.
By \fullref{thm:boundaryhomeoifftreeofboundaryhomeos}, this is
equivalent to the existence of a boundary homeomorphism between
$\bdry\modelspace$ and $\bdry\modelspace'$, hence between $\bdry G$
and $\bdry G'$.
\end{proof}

\begin{example}\label{ex:exampleboundary}
Recall the example of \fullref{sec:examplestretch}.
The trees of cylinders of these groups are isomorphic: they are
bipartite with one orbit of valence 2 cylindrical vertex and one orbit
of infinite valence relatively rigid vertex.
Recall also that the rigid vertices were discussed in
\fullref{sec:obstructions}, and the relative quasi-isometry (hence,
boundary homeomorphism) types of these three examples are the same.

There are two $G_i$--orbits of edge in $\cyl(G_i)$, so the only
remaining question about $\Homeo(\bdry G_i)$--orbits is whether there
are one or two orbits of edges.
For boundary homeomorphism we do not care about stretch factors on the
edges, so the edges have trivial initial decoration. 

The quasi-isometry group of the vertex stabilizer
preserving the peripheral structure coming from incident edge groups
is transitive on peripheral sets, so the vertex refinement is a
trivial refinement of the initial decoration.

Moreover, the quasi-isometry group of the vertex stabilizer
preserving the peripheral structure coming from incident edge groups
contains and infinite dihedral group in the stabilizer of each
peripheral subset, so all cylinders are balanced, and the cylinder
refinement is a trivial refinement of the initial decoration.

We conclude there is only one $\Homeo(\bdry G_i)$ orbit of edges in
$\cyl(G_i)$, so the initial decoration and initial (trivial) partial
orientation are stable. 
Since the structure invariants are the same for the three $G_i$, they
have homeomorphic boundaries.
\end{example}

\section{Quasi-isometry classification of groups from their two-ended JSJ splittings}
We are now almost ready to prove our second main theorem, characterizing the quasi-isometry type of a finitely presented one-ended group 
from its JSJ tree of cylinders. Before doing so, we explain the extreme flexibity provided by the hanging vertices of the tree.
\subsection{Quasi-isometric flexibility of hanging spaces}
Recall that the fixed model space for hanging vertices is the
universal cover of a fixed hyperbolic pair of pants $\Sigma$, with peripheral
structure consisting of the coarse equivalence classes of the boundary
components of $\tilde{\Sigma}$.
\begin{proposition}[{cf \cite[Theorem~1.2]{BehNeu08}}]\label{lemma:BN}
Let $G$ be a finitely presented, one-ended group admitting a JSJ
decomposition over two-ended subgroups with two-ended cylinder
stabilizers.
Let $\gog:=\lquotient{G}{\cyl(G)}$ and $\tree:=\cyl(G)=\tree(\gog)$.
Let $\modelspace$ be an algebraic tree of spaces
for $G$ over $\tree$.
Let $v$ be a hanging vertex in $\Gamma$.
Let $\decmap\from\tree\to\ornaments$ be a
$\QIgp(\modelspace)$--invariant decoration and let $\partialorientation$ be a
$\QIgp(\modelspace)$--invariant partial orientation.
For each edge $e\in\Gamma$ incident to $v$, choose a positive real
parameter $\sigma_e$.
Let $\decmap'\from\bdry\tilde{\Sigma}\to\ornaments$ be a decoration of
the peripheral structure of $\tilde{\Sigma}$ and let
$\partialorientation'$ be a partial orientation of $\bdry{\Sigma}$ such
that
$\QIgp(\tilde{\Sigma},\bdry\tilde{\Sigma},\decmap',\partialorientation')$
acts coboundedly on $\tilde{\Sigma}$.

Suppose that for some $e_0\in\link(\lift{v})$ we are given a coarse similitude
$\qi|_{\modelspace_{e_0}}$ from $\modelspace_{e_0}$ to a
component $B_0$ of $\bdry\tilde{\Sigma}$ that respects the decoration and
partial orientations.
Suppose further that there exists a 
$\qi'\in\QIsom((\modelspace_v,\per_v,\decmap,\partialorientation),(\tilde{\Sigma},\bdry\tilde{\Sigma},\decmap',\partialorientation'))$
such that  $\qi'(\modelspace_e)=B_0$.
Then there exists
$\qi\in\QIsom((\modelspace_v,\per_v,\decmap,\partialorientation),(\tilde{\Sigma},\bdry\tilde{\Sigma},\decmap',\partialorientation))$
extending $\qi|_{\modelspace_{e_0}}$
that, for each edge $e\in\link(v)\setminus\{e_0\}$, restricts to be a coarse similitude with multiplicative constant
$\sigma_{\quot{e}}$ on $\modelspace_{e}$.

The quasi-isometry constants of $\qi$ can be bounded in terms of
$G_v$, $\per_v$, the coboundedness constant for
$\QIgp(\tilde{\Sigma},\bdry\tilde{\Sigma},\decmap',\partialorientation')\act\tilde{\Sigma}$,
the constants of $\qi|_{\modelspace_{e_0}}$,
and the $\sigma_v$.
\end{proposition}
\begin{proof}[Sketch]
  The proof follows the same argument as \cite[Theorem~1.2]{BehNeu08}.
The idea is to build $\qi$ inductively, peripheral set by peripheral
set.
We start with $\qi_{\modelspace_{e_0}}$. 
Let $\sigma_0$ be the multiplicative constant of $\qi_{\modelspace_{e_0}}$.
Then we want to extend $\qi$ to peripheral sets that come close to
$\modelspace_{e_0}$ in $\modelspace_v$, sending these to
components of $\bdry\tilde{\Sigma}$ that come close to
$\qi(\modelspace_{e_0})$.

For another peripheral set $\modelspace_e$, the number of peripheral sets in the
$\QIgp(\modelspace_v,\per_v,\decmap,\partialorientation)$--orbit of
$\modelspace_e$ that
come within some fixed distance $K$ of a subsegment of $\modelspace_{e_0}$ of
length $l$ is coarsely $dl$ for some $d>0$.

Let $d'_r$ be such that there are coarsely $d'_rl$
peripheral sets in the
$\QIgp(\tilde{\Sigma},\bdry\tilde{\Sigma},\decmap',\partialorientation')$--orbit
of $\qi'(\modelspace_e)$ that come within $r$ of a subsegment of
$B_0$ of length $l$.
The fact that
$\QIgp(\tilde{\Sigma},\bdry\tilde{\Sigma},\decmap',\partialorientation')\act\tilde{\Sigma}$
is $C$--cobounded for some $C$
says that $d'_C>0$, and, in fact, $d'_r$ grows exponentially in $r$.
This means that there is a logarithmically growing function whose value $R$ at
$\frac{d}{\sigma_0}$ is such that $d'_R\geq\frac{d}{\sigma_0}$.
Thus, for any $l$ there is a way to send the $dl$ elements in the $\QIgp(\modelspace_v,\per_v,\decmap,\partialorientation)$--orbit of
$\modelspace_e$ that come within distance $K$ of a length $l$ subsegment $S$
of $\modelspace_{e_0}$ injectively to the peripheral sets in the
$\QIgp(\tilde{\Sigma},\bdry\tilde{\Sigma},\decmap',\partialorientation')$--orbit
of $\qi'(\modelspace_e)$ that come within $R$ of the length
approximately $\sigma_0l$ subsegment
$\qi|_{\modelspace_e}(S)$ of
$B_0$.

In this way one builds a matching between the peripheral sets that
come close to $\modelspace_{e_0}$ and the peripheral sets that come
close to $B_0$, respecting decorations and partial orientations.
Then $\qi$ is defined along such a matched pair to be a coarse
similarity with the appropriate $\sigma_e$ as multiplicative
constant, and the neighbor-matching is repeated for each such pair.
\end{proof}

\subsection{Geometric trees of spaces for groups with two-ended
  cylinders stabilizers}\label{sec:geometricmodel}
Let $G$ be a finitely presented, one-ended group admitting a JSJ
decomposition over two-ended subgroups with two-ended cylinder
stabilizers.
Let $\gog:=\lquotient{G}{\cyl(G)}$ and $\tree:=\cyl(G)=\tree(\gog)$.

Recall that in \fullref{sec:models} we built an algebraic tree of
spaces $Y$ quasi-isometric to $G$, and gave conditions for a
collection of quasi-isometries of the vertex spaces to patch together
to give a quasi-isometry of $Y$. 
Now we will construct a \emph{geometric tree of spaces} $\modelspace$ by
uniformizing the vertex spaces, that is, replacing each vertex space
by its uniform model from \fullref{sec:uniformization}.
The quasi-isometries between vertex spaces and their uniform models
will patch together to give a quasi-isometry from $Y$ to $X$.
Therefore, $X$ will be quasi-isometric to $G$.
The price to pay for uniformizing the vertex spaces is that in
general $G$ only admits a cobounded quasi-action on $X$, not a
cocompact action, but this will not affect us.

We use the same notation as in \fullref{sec:models}.
Let $Y$ be the algebraic tree of spaces constructed there.
For a relatively rigid vertex $v\in\Gamma$, fix a quasi-isometry $\fixedmodelmap_v\from 
(G_v,\per_v)\to (Z_{\llbracket 
  (G_v,\per_v)\rrbracket},\per_{\llbracket (G_v,\per_v)\rrbracket})$
from $G_v$ to the chosen model space for the relative quasi-isometry
type of $(G_v,\per_v)$.

If $G_v$ is virtually cyclic choose a cyclic subgroup $\langle
z_v\rangle<G_v$ of minimal index. 
Define $\fixedmodelmap_v$ by sending $G_v$ onto $\langle z_v\rangle$
by closest point projection and sending $z_v^k$ to
$k\sigma_v\in\mathbb{R}$, where $\sigma_v$ is a positive real
parameter chosen as follows.
If $v$ is not adjacent to any quasi-isometrically rigid vertices
then choose $\sigma_v:=1$.
Otherwise, choose $\sigma_v:=\min \len_{\modelspace_w}(z_v)$, where the
minimum is taken over quasi-isometrically rigid vertices $w$ adjacent
to $v$.

\begin{rem}\label{rem:cylinderparameters}
This
choice of $\sigma_v$'s is convenient because it will imply, for an edge $e$ with $\initial(e)$ cylindrical and
$\terminal(e)$ rigid, that the attaching map
$\attachmap_e^\modelspace$ constructed below is a coarse similitude whose
multiplicative constant is equal to the stretch factor $\rs(e)$
defined in \fullref{sec:normalization}.  
\end{rem}

For a hanging vertex $v\in\Gamma$ we define $\fixedmodelmap_v$ as
follows.
For each edge $e\in\link(v)$ choose a cyclic subgroup $\langle
z_e\rangle<G_e$ of minimal index.
Define \[\fixedmodelmap_v\from (G_v,\per_v)\to (Z_{\llbracket 
  (G_v,\per_v)\rrbracket},\per_{\llbracket (G_v,\per_v)\rrbracket})\]
to be a quasi-isometry such that for each $e\in\link(v)$, each coset
of $G_e$, which is  a peripheral set in $\per_v$, is sent to a
peripheral set in $\per_{\llbracket (G_v,\per_v)\rrbracket}$ by a
coarse similitude with multiplicative constant:
\[\frac{\frac{[\langle z_{\terminal(e)}\rangle:\langle
    z_{\terminal(e)}\rangle\cap\langle z_e\rangle]}{[\langle z_e\rangle:\langle
    z_{\terminal(e)}\rangle\cap\langle z_e\rangle]}\cdot\sigma_{\terminal(e)}}{\len_{G_v}(z_e)}\]
Here, $\len_{G_v}(z_e)$ is the translation length of $z_e$ in the
Cayley graph of $G_v$, which is non-zero since $G_v$ is hyperbolic, 
and $\sigma_{\terminal(e)}$ is the parameter for $G_{\terminal(e)}$ chosen
above.
Such a quasi-isometry can be constructed using \fullref{lemma:BN}.
These particular values are chosen to make
\fullref{lemma:hangingattachmentiscoarseisometry}, below,  true.

Now, for each vertex $v\in\verts\tree$ define $\modelspace_v$ to be a copy
of $\fixedmodel_{\llbracket 
  (G_v,\per_v)\rrbracket}$ with isometry $\modelmap_v\from \modelspace_v\to\fixedmodel_{\llbracket 
  (G_v,\per_v)\rrbracket}$.
Define $\qi_v\from Y_v\to\modelspace_v$ by $x\mapsto \modelmap_v^{-1}\circ\fixedmodelmap_{\quot{v}}(h^{-1}_{(\quot{v},i)} x)$.

We define edge spaces and attaching maps in $X$ to be compatible with
those of $Y$, as follows. 
Consider an edge $e$ with $v:=\initial(e)$ and
$w:=\terminal(e)$.
There are $h_{(\quot{v},i)}$ and $g_{(\quot{e},j)}$ such that
$v=h_{(\quot{v},i)}\widetilde{\quot{v}}$ and
$e=h_{(\quot{v},i)}g_{(\quot{e},j)}\widetilde{\quot{e}}$.
Define $\attachmap^X_e:=\qi_w\circ \attachmap^Y_e\circ \pi_{Y_e}\circ
\qi_v^{-1}$, where $\pi_{Y_e}$ denotes closest point projection to
$Y_e$. 
The map $\attachmap^X_e$ is coarsely well defined, since $\pi_{Y_e}$ moves
points of $\qi_v^{-1}(X_e)$ bounded distance.
Define
$\attachmap^X_{\bar{e}}:=\qi_v\circ\attachmap^Y_{\bar{e}}\circ\pi_{Y_{\bar{e}}}\circ\qi_w^{-1}$,
where $\pi_{Y_{\bar{e}}}$ is closest point projection from $Y_w$ to
the coarsely dense subset $Y_{\bar{e}}$.
This map is well defined, and is a quasi-isometry inverse to
$\attachmap^X_e$, since $\pi_{Y_{\bar{e}}}$ moves points bounded distance.

Chasing through these definitions on easily demonstrates:
\begin{lemma}\label{lemma:hangingattachmentiscoarseisometry}
  If $c=\initial(e)$ is cylindrical and $v=\terminal(e)$ is hanging
  then $\attachmap^X_e\from\modelspace_c\to\modelspace_v$ is a coarse
  isometric embedding.
\end{lemma}

\begin{proposition}\label{prop:groupisqitotreeofspaces}
With notation as above, $G$, $X$, and $Y$ are quasi-isometric to one another.
\end{proposition}
\begin{proof}
$G$ is quasi-isometric to $Y$ by \fullref{lemma:algebraictreeofspaces}.
\fullref{prop:treeofqis} implies $X$ and $Y$ are quasi-isometric,
since $(\qi_v)$ is a tree of quasi-isometries over the identity on $\tree$
compatible with $X$ and $Y$.
\end{proof}

\subsection{Quasi-isometries}
Let $G$ be a finitely presented, one-ended group with non-trivial JSJ
decomposition over two-ended subgroups such that:
\begin{itemize}
\item Every non-elementary
vertex is either hanging or quasi-isometrically rigid relative to the
peripheral structure coming from incident edge groups.
\item Cylinder stabilizers are two-ended.
\end{itemize}

If \fullref{question:arehyperbolicgroupsrigid} has positive answer then
every one-ended hyperbolic group with a non-trivial JSJ decomposition
is of this form.

\begin{theorem}\label{thm:qi}
Let $G$ and $G'$ be finitely presented, one-ended groups with
non-trivial JSJ decompositions over two-ended subgroups with two-ended
cylinder stabilizers.

Let $\tree:=\cyl(G)$.
Let $\modelspace$ be a geometric tree of spaces for $G$ over $\tree$,
as in \fullref{sec:geometricmodel}.
Let $\partialorientation_0$ be the trivial partial orientation on
$\modelspace$.
Let $\decmap_0$ the decoration on $\tree$ that sends an edge $e$
incident to a rigid vertex to its relative stretch factor
$\rs(e)$ as in \fullref{def:relstretch}, sends other edges to
$\nullvar$, and sends vertices to their vertex
type (`cylindrical', `rigid', or `hanging') and relative quasi-isometry type.
Let $\partialorientation_0$ be the trivial partial orientation on $\modelspace$.
Perform neighbor, cylinder, and vertex refinement until all three
stabilize to give a decoration $\decmap\from\tree\to\ornaments$ and a
partial orientation $\partialorientation$ of $\modelspace$.

Define $\tree'$,
$\modelspace'$, $\decmap'_0$, $\partialorientation'_0$,
$\decmap'\from\tree'\to\ornaments'$, and $\partialorientation'$ for
$G'$ as we
did for $G$.
In particular, $\modelspace'$ is uniformized with respect to the same
choice of model spaces from \fullref{def:qitypes}.

Then  $G$ and $G'$ are quasi-isometric if and only if
there exists a bijection $\beta\from\image
  \decmap\to\image\decmap'$ and $\xi\in\{-1,1\}^{\ornaments}$ such that:
\begin{enumerate}
\item $\decmap_0\circ\decmap^{-1}=\decmap_0'\circ(\decmap')^{-1}\circ\beta$
\item When the rows and columns of
  $\struc(\tree',\decmap',\ornaments')$ are given the 
  $\beta$--induced ordering from $\struc(\tree,\decmap,\ornaments)$,
  we have
  $\struc(\tree,\decmap,\ornaments)=\struc(\tree',\decmap',\ornaments')$.
\item For
  every $\ornament\in\image \decmap$ such that
  $\decmap^{-1}(\ornament)$ consists of non-elementary vertices there exists (equivalently, for
  every) $v\in\decmap^{-1}(\ornament)$ and
  $v'\in(\decmap')^{-1}(\beta(\ornament))$ so that \[\QIsom((\modelspace_{v},\per_{v},\beta\circ\decmap,\xi\cdot\partialorientation),(\modelspace'_{v'},\per'_{v'},\decmap',\partialorientation')
)\] is nonempty.\label{item:nonemptyqis}
\item For
  every $\ornament\in\image \decmap$ such that
  $\decmap^{-1}(\ornament)$ consists of cylindrical vertices, there exists (equivalently, for
  every) $c\in\decmap^{-1}(\ornament)$ and
  $c'\in(\decmap')^{-1}(\beta(\ornament))$ such that $\imbalance_c^{\decmap,\xi\cdot\partialorientation}=\imbalance_{c'}^{\decmap',\partialorientation'}\circ\beta$.
\end{enumerate} 
\end{theorem}
The construction is a modification of the proof of
\fullref{thm:starmapmain}.
Recall in that case we inductively built
$\iso\in\Isom((\tree,\decmap),(\tree,\decmap'))$ and quasi-isometries
\[\qi_v\in\QIsom((\modelspace_v,\per_v,\beta\circ\decmap,\xi\cdot\partialorientation),(\modelspace'_{\iso(v)},\per'_{\iso(v)},\decmap',\partialorientation'))\] 
such that $(\qi_v)_*=\iso|_{\link(v)}$.
The proof of \fullref{thm:starmapmain} mainly focuses on the inductive
step in the link of a cylindrical vertex, and chooses any $\qi_v$ as
above such that $(\qi_v)_*$ agrees with $\iso$ on the incoming edge to
$v$.

In the present context we must be more careful about the choices of the $\qi_v$.
The proof in \fullref{thm:starmapmain} gives us a collection of
quasi-isometries $(\qi_v)$ such that
for every edge $e\in\tree$ with $\initial(e)$ cylindrical we have that
$(\attachmap_{\iso(e)}\circ\qi_{\initial(e)})\circ(\qi_{\terminal(e)}\circ\attachmap_{e})^{-1}$
is orientation preserving on $\modelspace'_{\overline{\iso(e)}}$, but
now we require it to be coarsely the identity on
$\modelspace'_{\overline{\iso(e)}}$.
Furthermore, we need the quasi-isometry constants of the $\qi_v$ to be
uniformly bounded.

Here is how we achieve these requirements.
Hanging vertices present no obstacles, since by \fullref{lemma:BN}
they are so flexible. The real work is in dealing with the rigid
vertices. For these we choose in advance a finite number of
quasi-isometries to use as building blocks. Since the collection is
finite, the constants are uniformly bounded. 
We will choose the maps on cylinder spaces to be coarse isometries. 
It then remains to see that if $e\in\edges\tree$ is an edge with $c:=\initial(e)$ cylindrical and
$v:=\terminal(e)$ relatively rigid, that we can make $\qi_c$ agree
with a map $\qi_v$ constructed from the pre-chosen building blocks. 
We assume we have chosen enough building blocks so that we can make
$(\qi_v)_*(e)=(\qi_c)_*(e)$, with the correct orientation on
$\modelspace_{\bar{e}}$.
This is handled by the same considerations as \fullref{thm:starmapmain}.
Additionally, we have set up the geometric tree of spaces so that the edge inclusion
into a rigid vertex is a coarse similitude whose multiplicative
constant is the stretch factor of the edge.
Since we have incorporated the stretch factors into the decorations,
we are guaranteed that the stretch factor on $e$ matches the stretch
factor on $\iso(e)$. It follows that
$(\attachmap_{\iso(e)}\circ\qi_{c})\circ(\qi_{v}\circ\attachmap_{e})^{-1}$ is a coarse
isometry that is orientation preserving. 
Finally, we make it coarsely the identity by adjusting $\qi_v$ using
the group action.

\begin{proof}[{Proof of \fullref{thm:qi}}]
By \fullref{qiinvariance}, \fullref{qirestricted},
\fullref{prop:treeofqis}, and \fullref{prop:stretchpreserved},
$\modelspace$ and $\modelspace'$ are quasi-isometric if and only if
there exists a tree of quasi-isometries over an element of
$\Isom((\tree,\decmap_0),(\tree',\decmap_0'))$ compatible with
$\modelspace$ and $\modelspace'$.

The existence of a tree of quasi-isometries over an element of
$\Isom((\tree,\decmap_0),(\tree',\decmap_0'))$ compatible with
$\modelspace$ and $\modelspace'$ implies the above conditions.
Our goal is to show the converse.

Suppose $\ornament\in\ornaments$ is an ornament such that
$\decmap^{-1}(\ornament)$ consists of vertices that are relatively quasi-isometrically
rigid.
Choose representatives
$v_{\ornament,1},\dots,v_{\ornament,i_{\ornament}}$ of the $G$--orbits
contained in $\decmap^{-1}(\ornament)$.
Suppose $\ornament'\in\ornaments$ is an ornament such that
$\decmap^{-1}(\ornament')$ consists of edges incident to
$v\in\decmap^{-1}(\ornament)$.
For each $1\leq i\leq i_\ornament$ choose representatives
$e_{\ornament,i,\ornament',1},\dots,e_{\ornament,i,\ornament',j_{\ornament'}}$
of the $G_{v_{\ornament,i}}$--orbits in $\decmap^{-1}(\ornament')\cap\link(v_{\ornament,i})$.
For each $i$ and $j$ choose
\[\Phi_{\ornament,i,\ornament',j}\in\QIsom((\modelspace_{v_{\ornament,i}},\per_{v_{\ornament,i}},\decmap,\xi\cdot\partialorientation),(\modelspace_{v_{\ornament,1}},\per_{v_{\ornament,1}},\decmap,\xi\cdot\partialorientation))\]
such that
$(\Phi_{\ornament,i,\ornament',j})_*(e_{\ornament,i,\ornament',j})=e_{\ornament,1,\ornament',1}$.
Such quasi-isometries exist by \fullref{prop:redecorationvslocalqi}.

Similarly choose representatives
$v'_{\beta(\ornament),1},\dots,v'_{\beta(\ornament),i'_{\beta(\ornament)}}$ of the $G'$--orbits
contained in $(\decmap')^{-1}(\beta(\ornament))$ and representatives
$e'_{\beta(\ornament),i,\beta(\ornament'),1},\dots,e'_{\beta(\ornament),i,\beta(\ornament'),j'_{\beta(\ornament')}}$
of the $G'_{v_{\beta(\ornament),i}}$--orbits in
$(\decmap')^{-1}(\beta(\ornament'))\cap\link(v'_{\beta(\ornament),i})$ and
quasi-isometries $\Phi'_{\beta(\ornament),i,\beta(\ornament'),j}$.

Choose a quasi-isometry
$\Phi_{\ornament,\ornament'}\in\QIsom((\modelspace_{v_{\ornament,1}},\per_{v_{\ornament,1}},\decmap,\xi\cdot\partialorientation),(\modelspace'_{v'_{\beta(\ornament),1}},\per'_{v'_{\beta(\ornament),1}},\decmap',\partialorientation'))$
that takes $e_{\ornament,1,\ornament',1}$ to $e'_{\beta(\ornament),1,\beta(\ornament'),1}$.
Such a quasi-isometry exists by condition \ref{item:nonemptyqis} and
\fullref{prop:redecorationvslocalqi}.
If $e_{\ornament,1,\ornament',1}$ is $\partialorientation$--unoriented
then we also choose 
\[\Phi^-_{\ornament,\ornament'}\in\QIsom((\modelspace_{v_{\ornament,1}},\per_{v_{\ornament,1}},\decmap,\xi\cdot\partialorientation),(\modelspace'_{v'_{\beta(\ornament),1}},\per'_{v'_{\beta(\ornament),1}},\decmap',\partialorientation'))\]
that takes $e_{\ornament,1,\ornament',1}$ to
$e'_{\beta(\ornament),1,\beta(\ornament'),1}$ such that
$\Phi^-_{\ornament,\ornament'}\circ(\Phi_{\ornament,\ornament'})^{-1}$
orientation reversing on
$\modelspace'_{\bar{e}'_{\beta(\ornament),1,\beta(\ornament'),1}}$.
Such a quasi-isometry exists by \fullref{lemma:reversingmap}.

We have chosen finitely many quasi-isometries $\Phi$, so they have
uniformly bounded quasi-isometry constants.

\paragraph{Induction base case}
Begin the induction by choosing a cylindrical vertex $c\in\tree$ and a
cylindrical vertex $c'\in(\decmap')^{-1}(\beta(\decmap(c)))$.
Define $\iso(c):=c'$. 
By construction $\modelspace_c$ and $\modelspace'_{c'}$ are copies of
$\mathbb{R}$. 
Define $\qi_c\from\modelspace_c\to\modelspace'_{c'}$ to be an
isometry.
If $c$ is $\partialorientation$--oriented we choose $\phi_c$ so that $\phi_c(\xi\cdot\partialorientation(\modelspace_c))=\partialorientation'(\modelspace'_{c'})$.
Extend $\iso$ to $\link(c)$ as in \fullref{thm:starmapmain}.

\paragraph{Inductive steps for non-elementary vertices}
Suppose $v=\terminal(e)$ is a non-elementary vertex such that for
$c=\initial(e)$ we have already defined a coarse isometry $\qi_c\from\modelspace_c\to\modelspace'_{\iso(c)}$ and
$\iso|_{\link(c)}$. 
Suppose further that if $c$ is
$\partialorientation$--oriented then $\phi_c(\xi\cdot\partialorientation(\modelspace_c))=\partialorientation'(\modelspace'_{\iso(c)})$.
Let $c':=\iso(c)$ and $e':=\iso(e)$.

\subparagraph{Suppose $v$ is rigid.}
Now $gv=v_{\decmap(v_0),i}$ for some $g\in G$ and some $i$, and
$hge=e_{\decmap(v),i,\decmap(e),j}$ for some $j$ and some $h\in G_{v_{\decmap(v),i}}$.
Similarly, there are $g'\in G'$ and $i'$ such that
$g'v'=v'_{\beta(\decmap(v)),i'}$, and $j'$ and  $h'\in
G'_{v'_{\beta(\decmap(v)),i'}}$ such that
$h'g'e'=e'_{\beta(\decmap(v)),i',\beta(\decmap(e)),j'}$. 

The map
\[(h'g')^{-1}\circ(\Phi'_{\beta(\decmap(v_0)),i',\beta(\decmap(e)),j'})^{-1}\circ\Phi_{\decmap(v_0),\decmap(e)}\circ\Phi_{\decmap(v_0),i,\decmap(e),j}\circ hg\]
is an element of
$\QIsom(\modelspace_{v_0},\per_{v_0},\beta\circ\decmap,\xi\cdot\partialorientation),(\modelspace'_{v'_0},\per'_{v'_0},\decmap',\partialorientation'))$
taking $e$ to $e'$.

If $e$ is $\partialorientation$--unoriented then we also have that the map
\[(h'g')^{-1}\circ(\Phi'_{\beta(\decmap(v_0)),i',\beta(\decmap(e)),j'})^{-1}\circ\Phi^-_{\decmap(v_0),\decmap(e)}\circ\Phi_{\decmap(v_0),i,\decmap(e),j}\circ hg\]
is an element of
$\QIsom(\modelspace_{v_0},\per_{v_0},\beta\circ\decmap,\xi\cdot\partialorientation),(\modelspace'_{v'_0},\per'_{v'_0},\decmap',\partialorientation'))$
taking $e$ to $e'$. 
For one of these two, the composition with
$\attachmap_e\circ(\qi_c)^{-1}\circ(\attachmap'_{e'})^{-1}$ is orientation
preserving on $\modelspace_{\bar{e}}$. 
Choose this one as $\qi_v'$.

Finally, choose a point $x\in\modelspace_{\bar{e}}$.
Since edge stabilizers act uniformly coboundedly on their
corresponding peripheral sets, we can choose an element $k\in
G'_{\bar{e}'}$ that is orientation preserving on
$\modelspace'_{\bar{e}'}$ and such that $k\qi'_v(x)$ is boundedly
close to $\attachmap'_{e'}\circ\qi_c\circ(\attachmap_{e})^{-1}(x)$.
Define $\qi_v:=k\qi'_v$.

We have:
\begin{itemize}
\item $\qi_v\in
  \QIsom((\modelspace_{v_0},\per_{v_0},\beta\circ\decmap,\xi\cdot\partialorientation),(\modelspace'_{v'_0},\per'_{v'_0},\decmap',\partialorientation'))$
\item $(\qi_v)_*(e)=e'$
\item $\qi_v(x)$ is boundedly close to $\attachmap'_{e'}\circ\qi_c\circ(\attachmap_{e})^{-1}(x)$.
\item
  $\qi_v\circ(\attachmap'_{e'}\circ\qi_c\circ(\attachmap_{e})^{-1})^{-1}$
  is orientation preserving on $\modelspace'_{\bar{e}'}$.
\item $\qi_v$ is a composition of three of the pre-chosen $\Phi$ with
  multiplication by five group elements, so the quasi-isometry
  constants of $\qi_v$ are bounded in terms of those of the $\Phi$ and
  the constants for the group action.
\end{itemize}

By relative quasi-isometric rigidity, $\qi_v$ is a coarse isometry.

We also claim that
$\attachmap'_{e'}\circ\qi_c\circ(\attachmap_{e})^{-1}$ is a coarse
isometry.
This is because $\qi_c$ is a coarse isometry, by the induction
hypothesis, and $\attachmap_e$ and $\attachmap'_{e'}$ are, by
construction (recall \fullref{rem:cylinderparameters}), coarse
similitudes with multiplicative constants $\rs(e)$ and $\rs(e')$,
which are equal, since:
\[\rs(e)=\decmap_0\circ\decmap^{-1}(\decmap(e))=\decmap_0'\circ(\decmap')^{-1}\circ\beta(\decmap(e))=\decmap_0'\circ(\decmap')^{-1}(\decmap'(e'))=\decmap_0'(e')=\rs(e')\]

Thus, $\qi_v\circ(\attachmap'_{e'}\circ\qi_c\circ(\attachmap_{e})^{-1})^{-1}$
  is orientation preserving coarse isometry on
  $\modelspace'_{\bar{e}'}$ that coarsely fixes a point.
It follows that $\qi_v|_{\modelspace{\bar{e}}}$ and
$\attachmap'_{e'}\circ\qi_c\circ(\attachmap_{e})^{-1}$ are coarsely
equivalent.
 
Define $\iso|_{\link(v)}:=(\phi_{v})_*$.

For each edge $e''\in\link(v)\setminus\{\bar{e}\}$ define
$\phi_{\terminal(e'')}:=\attachmap'_{\iso(e'')}\circ\qi_{v}\circ(\attachmap_{e''})^{-1}$.
Since $\phi_v$ is a coarse isometry and $\attachmap_{e''}$ and
$\attachmap'_{\iso(e'')}$ are coarse similitudes with the same
multiplicative constant, as above, we have that
$\phi_{\terminal(e'')}$ is a coarse isometry.

\subparagraph{Suppose $v$ is hanging.}
The map $\attachmap'_{e'}\circ\qi_c\circ(\attachmap_e)^{-1}\from\modelspace_{\bar{e}}\to\modelspace'_{\bar{e}'}$ is a
coarse isometry, since attaching maps to hanging vertex spaces are
coarse isometries by \fullref{lemma:hangingattachmentiscoarseisometry} and $\qi_c$ is a
coarse isometry by the induction hypothesis.

Use condition \ref{item:nonemptyqis} and \fullref{lemma:BN} to produce
a quasi-isometry \[\qi_{v}\in \QIsom((\modelspace_{v},\per_{v},\beta\circ\decmap,\xi\cdot\partialorientation),(\modelspace'_{v'},\per'_{v'},\decmap',\partialorientation')
)\]
that is a coarse isometry along each peripheral subset and that coarsely
agrees with $\attachmap'_{e'}\circ\qi_c\circ(\attachmap_e)^{-1}$ on $\modelspace_{\bar{e}}$.

Define $\iso|_{\link(v)}:=(\qi_{v})_*$.

For each
$e''\in\link(v)\setminus\{\bar{e}\}$ the map
$\qi_{\terminal(e'')}:=\attachmap'_{\iso(e'')}\circ\qi_{v}\circ(\attachmap_{e''})^{-1}$
is a coarse isometry, since attaching maps to hanging vertex spaces are
coarse isometries by
\fullref{lemma:hangingattachmentiscoarseisometry}, and $\qi_{v}$ is a
coarse isometry along peripheral sets by construction.

\paragraph{Inductive step for cylindrical vertices}
Suppose $c=\initial(e)$ is cylindrical, $\iso$ is defined on $e$,
$\phi_{\terminal(e)}$ is defined, and $\phi_c$ is a coarse isometry
such that $\phi_c$ is coarsely equivalent to $(\attachmap'_{\iso(e)})^{-1}\circ\phi_{\terminal(e)}\circ\attachmap_e$.
Extend $\iso$ to $\link(c)\setminus\{e\}$ as in \fullref{thm:starmapmain}.
\bigskip

This completes the induction.
The result is $\iso\in\Isom((\tree,\decmap),(\tree',\decmap'))$ and
uniform quasi-isometries $(\qi_v)$ satisfying the conditions of
\fullref{corollary:treeofqis}, so $(\qi_v)$ is a tree of
quasi-isometries over $\iso$ compatible with $\modelspace$ and
$\modelspace'$, as desired.
\end{proof}

 \bibliographystyle{hypershort}
 \bibliography{JSJ_Cashen_Martin}

\bigskip
  \footnotesize

  \textsc{Fakult\"at f\"ur Mathematik, Universit\"at Wien,
    Oskar-Morgenstern-Platz 1, 1090 Wien, \"{O}sterreich}\par\nopagebreak
\texttt{\href{mailto:christopher.cashen@univie.ac.at}{christopher.cashen@univie.ac.at}}\par\nopagebreak
\texttt{\href{mailto:alexandre.martin@univie.ac.at}{alexandre.martin@univie.ac.at}}
\end{document}